\documentclass[11pt]{article}
\usepackage[letterpaper, margin=1in]{geometry}
\usepackage[T1]{fontenc}
\usepackage{appendix}
\usepackage{subfigure}
\usepackage{xcolor}
\usepackage{amsmath, amssymb, amsthm, authblk, bbm, mathtools, subcaption}
\usepackage[linktocpage=true]{hyperref}
\hypersetup{
    colorlinks,
    linkcolor={red!50!black},
    citecolor={blue!50!black},
    urlcolor={blue!80!black}
}

\allowdisplaybreaks
\numberwithin{equation}{section}

\usepackage{titlesec}
\usepackage{algorithm,algpseudocode}
\usepackage{enumerate}
\usepackage{float}
\usepackage{bm}
\mathtoolsset{showonlyrefs}

\makeatletter
\newenvironment{subroutine}[1][htb]{%
    \renewcommand{\ALG@name}{Sub-Routine}% Update algorithm name
   \begin{algorithm}[#1]%
  }{\end{algorithm}}
\makeatother

\numberwithin{equation}{section}
\newtheorem{theorem}{Theorem}[section]
\newtheorem{corollary}[theorem]{Corollary}
\newtheorem{definition}[theorem]{Definition}
\newtheorem{proposition}[theorem]{Proposition}
\newtheorem{lemma}[theorem]{Lemma}
\newtheorem{remark}[theorem]{Remark}

\newcommand{\keywords}[1]{\textbf{\textit{Keywords ---}} #1}

\newcommand{\R}{\mathbb{R}}
\newcommand{\N}{\mathbb{N}}

\DeclareMathOperator*{\argmin}{\arg\!\min}

\newcommand{\robd}{{\textsc{robd}}}
\newcommand{\opt}{{\textsc{opt}}}
\newcommand{\acord}{{\textsc{acord}}}
\newcommand{\lpc}{{\textsc{lpc}}}
\newcommand{\ftm}{{\textsc{ftm}}}
\newcommand{\loc}{{\textsc{local}}}
\newcommand{\dospa}{{\textsc{d-ospa}}}
\newcommand{\ftpl}{{\textsc{ftpl}}}
\newcommand{\am}{{\textsc{am}}}
\newcommand{\acordcaps}{{\textsc{ACORD}}}
\newcommand{\lpccaps}{{\textsc{LPC}}}
\newcommand{\alg}{{\textsc{alg}}}

\begin{document}

\title{Optimal Decentralized Smoothed Online Convex Optimization\footnote{This work was partially supported by the NSF grants CIF-2113027 and CPS-2240982. }}
\date{}

\author[1]{Neelkamal Bhuyan\thanks{Email: \href{mailto:nbhuyan3@gatech.edu}{nbhuyan3@gatech.edu}}\textsuperscript{,}}
\author[1]{Debankur Mukherjee}
\author[2]{Adam Wierman}

\affil[1]{Georgia Institute of Technology}
\affil[2]{California Institute of Technology}

\maketitle

\begin{abstract}
We study the multi-agent Smoothed Online Convex Optimization (SOCO) problem, where $N$ agents interact through a communication graph. In each round, each agent~$i$ receives a strongly convex hitting cost function $f^i_t$  in an online fashion and selects an action $x^i_t \in \mathbb{R}^d$. The objective is to minimize the global cumulative cost, which includes the sum of individual hitting costs $f^i_t(x^i_t)$, a temporal ``switching cost'' for changing decisions, and a spatial ``dissimilarity cost’’ that penalizes deviations in decisions among neighboring agents. 
We propose the first truly decentralized algorithm \acord{} for multi-agent SOCO that provably exhibits asymptotic optimality. Our approach allows each agent to operate using only local information from its immediate neighbors in the graph. For finite-time performance, we establish that the optimality gap in the competitive ratio decreases with time horizon $T$ and can be conveniently tuned based on the per-round computation available to each agent. Our algorithm benefits from a provably scalable computational complexity that depends only logarithmically on the number of agents and almost linearly on their degree within the graph. Moreover, our results hold even when the communication graph changes arbitrarily and adaptively over time. Finally, \acord{}, by virtue of its asymptotic-optimality, is shown to be provably superior to the state-of-the-art \lpc{} algorithm, while exhibiting much lower computational complexity. Extensive numerical experiments across various network topologies further corroborate our theoretical claims.
\end{abstract}
\section{Introduction}\label{sec:intro}
We study a class of multi-agent smoothed online convex optimization (SOCO) problems where each agent $i \in \{1,\ldots, N\} = [N]$ has to take \textit{online} decision $x^i_t \in \R^d$ in response to strongly convex \textit{hitting costs} $f^i_t(\cdot)$ while keeping in mind that it is additionally penalized for a temporal \emph{switching cost} $\frac{1}{2}\|x^i_t - x^i_{t-1}\|_2^2$ and a spatial \emph{dissimilarity cost} 
$s_t^{(i,j)}\big(x^i_t, x^j_t\big)$, with respect to agent $j$'s action, when they share an edge in graph $\mathcal{G}_t = ([N],\mathcal{E}_t)$. 
The dissimilarity cost penalizes deviations in decisions among neighboring agents.
These costs emerge from the necessity for neighboring agents to coordinate their actions, and they are particularly important in formation control for unmanned aerial vehicles (UAVs) \cite{dimarogonas2008stability,krick2009stabilisation,kuriki2015formation}, dynamic multi-product pricing \cite{SongChintagunta06,LiuTimothy18}, graph based combinatorial optimization \cite{Hochba97,GamarnikGoldberg10,GamarnikGoldberg14} and economic team theory \cite{Marschak55,Marschak72}. We provide detailed explanations for some of these applications in Appendix \ref{appendix_sec: examples}.
We consider this problem over a finite time horizon $T$, where at time $t\in [T]$, agents can only communicate their actions amongst each other along the edges of the graph $\mathcal{G}_t$.

% \begin{figure}
%     \centering
%     \includegraphics[width=\linewidth]{AISTATS_illustration.drawio (7).pdf}
%     \caption{Decentralized SOCO}
%     \label{fig:enter-label}
% \end{figure}

While the single-agent SOCO problem has received significant attention over the last decade due to its wide range of applications in data center management \cite{LinWiermanAndrew11,LinWierman12,LuAndrew13}, power systems \cite{NarayanaswamyGarg12,LuTuChau13,KimGiannakis17}, electrical vehicle charging \cite{GanTopcu13,KimMatthews15}, video transmission \cite{JosephVeciana12,ChenWierman2024} and chip thermal management \cite{ZaniniAtienza09,ZaniniAtienza10}, its decentralization has remained a crucial challenge for large-scale implementations.  
This dynamic multi-agent problem emerges in areas like power systems control \cite{MolzahnLow17,ShiGuannanWierman22}, formation control \cite{ChenZhongmin05,Kwang-Kyo15} and multi-product price optimization \cite{CaroGallien12,CandoganOzdaglar12,HansenDetlefsen19}. Appendix \ref{appendix_sec: lit_review} summarizes the literature in SOCO, online decentralized optimization and its applications.

Most works in offline decentralized optimization study coupling constraints ($x^i = x^j$ for an edge $(i,j)$)
\cite{ShiWotao14ADMM,MaNikolakopoulos18fast,MaNikolakopoulos18graph,MaNikolakopoulos18hybrid,lan2020communication,xin2020,Chunlei2019,kovalev2020optimal}. 
Online decentralized optimization too is focused on this idea \cite{AkbariLinder15,HosseiniChapman16,YiJohansson20,YuanProutiere21,XuanyuTamer21,JiangZhu21,CaoTamer21,ZhangLijun24} while also having other limitations in the form of static benchmarks, instead of dynamic, and bounded action spaces. We contrast these aspects in detail in Appendix \ref{appendix_sec: DOCO} by considering recent works like \cite{wang2023distributed}.

A prominent example of our multi-agent online set-up is swarm-control for Unmanned Aerial  Vehicles (UAVs) \cite{ben2008distributed,morbidi2011estimation,SherstjukUAV15,zavadskiy2018dynamics,seraj2020coordinated}. Such a problem involves a collection of $N$ UAVs traveling from point A to B in a desired formation while tackling external obstacles. The biggest difficulty here is formation control while ensuring local collision-avoidance. Simple consensus algorithms or constraints, as discussed in the previous paragraph, fail to take this into account \cite{hu2020formation}. Works like \cite{kuriki2015formation,dubay2018distributed,yasin2019formation,hu2020voronoi,yasin2020formation} explicitly use distance based penalty/thresholding in their formation control algorithms to prevent UAVs from crashing into each other. Further applications like dynamic multi-product pricing \cite{CandoganOzdaglar12}, decentralized battery management \cite{zhao2022distributed} and geographical load-balancing \cite{khalil2022renewable} have been explained in detail in Appendix \ref{appendix_sec: examples}.

The combination of dissimilarity costs and switching costs induces a complex spatio-temporal coupling of decisions across agents, which makes the design of decentralized algorithms particularly challenging.
It is not hard to see that a naive handling of this coupling, especially constraining agent actions' to be same, may lead to arbitrarily bad performance over the horizon $T$ (e.g., see Appendix~\ref{section: naive}).

To the best of our knowledge, the only work that attempts to provide a decentralized competitive algorithm in a related framework is~\cite{lin2022decentralized}, which proposed a variation of Model Predictive Control (MPC) called Localized Predictive Control (\lpc{}). However, their algorithm depends on perfect predictions of future cost functions and involves communication of infinite-dimensional hitting cost functions among all agents. This motivates the following question, which forms the basis of this paper:
\begin{quote}
    \textit{``Is it possible to design a scalable, decentralized online algorithm that matches the performance of the centralized optimal?"}
\end{quote}
\paragraph{Algorithmic Contributions.}  We design the first decentralized algorithm: \underline{A}lternating \underline{C}oupled \underline{O}nline \underline{R}egularized \underline{D}escent (\acord{}, Algorithm \ref{alg:ACORD}) that maintains a near-optimal competitive ratio (Theorem \ref{main_thm:CR}) without communication of hitting costs amongst the agents and without the use of predictions.  Further, \acord{} is computationally efficient, using only $d$-dimensional agent-local computations.
Additionally, we prove that \acord{} is \textit{asymptotically online-optimal} (Theorem \ref{main_thm: lower_bound_2}, Corollary \ref{corr: acord_asym_opt}), without requiring agents to communicate hitting costs. This makes it the first algorithm in the decentralized SOCO literature with \textit{zero inherent bias}.% resulting from decentralized operations.

We emphasize that \acord{}'s performance guarantees hold in  general and dynamic environments, including  heterogeneous hitting costs between agents, heterogeneous dissimilarity costs across edges of the graph, and \textit{dynamically changing graph} structures $\mathcal{G}_t.$ 
Our results also highlight the dependence of the performance guarantees for \acord{} on the graph properties.  In particular, the runtime complexity of \acord{} for $\mathcal{D}$-regular graphs is $\Theta\left(\mathcal{D}\log (N\mathcal{D})\right)$ (Theorem \ref{main_thm:graph_dependence}). The weak dependence on $N$ highlights the scalability of the algorithm, and the strong dependence on $\mathcal{D}$ indicates the dissimilarity's cost's influence.% on performance.
\paragraph{Beating the State of the Art.} In addition to the aforementioned results, we perform an in-depth comparison to the state-of-the-art algorithm in this setting, \lpc{} \cite{lin2022decentralized}. Corollary \ref{corr:lpc_sub_opt} shows that \lpc{} can match \acord{} only in the limit of full-network access and future predictions over the entire horizon $T$. We further show that \acord{} provably utilizes much less resources in computation and in communication than \lpc{} (Corollary \ref{corr: resource_utlization}). Our numerical experiments in Section \ref{sec:num_exp} further back these results.
\paragraph{Analytical Contributions.} Analytical advancements, across multiple fronts have facilitated the aforementioned results. The key methodological contributions are the introduction of (i) ADMM-inspired auxiliary variables and (ii) block-iterative optimization, to the SOCO framework. Carefully implementing these mechanisms in an online set-up allows \acord{} to achieve arbitrarily near-optimal performance (Theorem \ref{main_thm:no_bias_CR}). As a novel algorithmic approach in the SOCO literature, we develop a broader analysis framework that quantifies performance guarantees for algorithms like \acord{} which employ numerical approximations (Theorem \ref{main_thm: approx_robd_cr}).
\section{Model and Preliminaries}\label{sec:model_prelims}
\subsection{A coupled online environment}
Consider a set of agents $[N] = \{1,\ldots,N\}$ and a finite time horizon $[T] = \{1,\ldots,T\}$. In every round $t \in [T]$, each agent $i$ receives a $\mu_i$-strongly convex hitting cost $f^i_t(\cdot)$ and in response, has to take an action $x^i_t \in \R^d$. Along with the hitting costs, agent $i$ needs to take into account its switching cost\footnote{Although, we consider squared $\ell_2$-norm, our methods easily extend to Bregman Divergences like in \cite{GoelLinWierman19}.} $c(x^i_t,x^{i}_{t-1}) = \frac{1}{2}\|x^i_t - x^i_{t-1}\|_2^2$ and a dissimilarity cost as further detailed below.  

\begin{definition}[Dissimilarity Costs]
    At each round $t$, if two agents $i$ and $j$ are \textit{coupled}, they are penalized for the mismatch between their actions $x^i_t$ and $x^j_t$ through an additional cost
    \begin{align}\label{eqn: dissimilarity_cost_defn}
        \begin{split}
            s_t&^{(i,j)}\left(x^i_t, x^j_t\right)  
            = \frac{\beta}{2}\left\|A_t^{(i,j)} x^i_t  - A^{(i,j)}_t x^j_t \right\|_2^2\qquad \ 
        \end{split}
    \end{align}
    where $A_t^{(i,j)} \in \R^{r_t \times d}$ ($r_t \geq d$) is full rank with singular values in $[\sqrt{m}$, $\sqrt{l}]$ and $\beta\geq0$. Only agents $i$ and $j$ have knowledge of $A_t^{(i,j)}$ at round $t$.
\end{definition}
The motivation to study dissimilarity costs of the form \eqref{eqn: dissimilarity_cost_defn} is to generalize such costs encountered in distributed online control \cite{DunbarMurray06}, cooperative multi-agent networks \cite{RezaMurray07}, coupled network games \cite{Grammatico17}, and multi-product pricing \cite{ChenZhiLong15}, including the vanilla squared-$\ell_2$ norm costs. Appendix \ref{appendix_sec: examples} documents the extensive use of dissimilarity costs of this form, across various domains.
\begin{remark}
    The matrix $A_t^{(i,j)}$ needs to have rank $d$ so that $s_t^{(i,j)}$ is not degenerate, that is 
    \begin{align}
        s_t^{(i,j)}\left(x^i_t, x^j_t\right) > 0 \text{ if }x^i_t \neq x^j_t.        
    \end{align}
\end{remark}
\begin{definition}[Adaptive Graph]
    Agents $[N]$ are coupled amongst each other through an undirected graph $\mathcal{G}_t = ([N],\mathcal{E}_t)$, where $\mathcal{E}_t$ is the set of edges. Note that $\mathcal{G}_t$ can change adaptively and adversarially across rounds and is not necessarily connected.
\end{definition}
For each pair of agents $(i,j)$ sharing an edge in $\mathcal{G}_t$, there is a dissimilarity cost $s^{(i,j)}_t$. Additionally, agents are allowed to interact only with their \textit{immediate} neighbors and can share only their actions (or lower dimensional entities). This can be contrasted with previous works \cite{lin2022decentralized} that involve exchange of functions across an $r$-hop neighborhood. 
\subsection{Performance measure}
The online sequence of actions $\alg{} = \left\{ \mathbf{x}_1,\ldots, \mathbf{x}_T \right\}$, where $\mathbf{x}_t = (x^1_t,\ldots,x^N_t)$ is the vector of decisions at time $t$, is aimed at minimizing the following cumulative cost over the horizon $T$,
\begin{align}\label{eqn: cumulative_obj}
    \text{Cost}_{\alg{}}[1,T] =& \sum_{t=1}^T \bigg\{\sum_{i=1}^N f^i_t(x^i_t) + \frac{1}{2}\|x^i_t - x^i_{t-1}\|_2^2 + \frac{\beta}{2}\sum_{(i,j) \in \mathcal{E}_t} \left\|A_t^{(i,j)} x^i_t - A^{(i,j)}_t x^j_t\right\|_2^2 \bigg\}
\end{align}
To evaluate the performance of the algorithm, the online sequence of actions is compared to the hindsight optimal sequence of actions $\opt{} = \left\{ \mathbf{x}^*_1,\ldots, \mathbf{x}^*_T \right\}$ that optimizes the above cumulative cost \textit{offline}.

\begin{definition}[Asymptotic Competitive Ratio]
    An online algorithm \alg{} is said to have an \textit{asymptotic competitive ratio} of $\text{CR}_{\alg{}}$ if it satisfies for all problem instances
    \begin{align}\label{eqn: CR_defn}
        \text{Cost}_{\alg{}}[1,T] \leq \text{CR}_{\alg{}} \cdot \text{Cost}_{\opt{}}[1,T] + o(1)
    \end{align}
    where the constants in $o(1)$ are instance-independent.
\end{definition}
Here the $o(1)$ term is an error that converges to $0$ with $T \to \infty$ and that does not depend on the input or the environment. It is worthwhile to emphasize that this is a stronger guarantee than the more commonly studied \emph{generalized competitive ratio}~\cite{LinWierman12,LuAndrew13,RuttenMukherjee23,AntoniadisSimon21,AntoniadisSimon23mixing,AntoniadisSimon23untrusted}, where the error term is allowed to be a non-vanishing fixed constant.

Our goal in this work is to design a decentralized algorithm that provably achieves near-optimal worst-case performance while allowing (i) communication of only \textit{local} actions across edges of the dynamic network and (ii) strictly local optimization within agents. This is a major leap from existing literature \cite{lin2022decentralized}, where one resorts to communication of entire hitting costs functions $f^i_t$, which can be infinite-dimensional, and computationally expensive neighborhood optimization. Keeping this in mind, we present our main results in the next section.
\begin{algorithm} [ht]
\caption{\acordcaps{}: Alternating Coupled Online Regularized Decent}\label{alg:ACORD}
\flushleft \textbf{Input:} strong convexity parameters $\{\mu_i\}_{i=1}^N$.
\flushleft \textbf{Initialize:} $\lambda_1^i = \frac{2}{1+\sqrt{1+\frac{4}{\mu_i}}}$ for $i \in \{1,\ldots, N\}$ %\COMMENT{Algorithm hyperparameters}
\begin{algorithmic}
\For{$t = 1,2,\ldots,T$}
\State Receive $f_t^1(x), \ldots, f_t^N(x)$, $\beta$ and $\{A^e_t\}_{e\in \mathcal{E}_t}$ %from environment
\For{$k \in \{1,\ldots, K_t\}$}
\State $ \left(x^i_t \right)_k \gets \argmin \limits_{x \in \R^d} \bigg\{ f_t^i(x) + \frac{\lambda_1^i}{2}\|x - x_{t-1}^{i}\|_2^2 +\beta \sum\limits_{\mathcal{E}_t \ni e \ni i}\|A_t^{e} x - A_t^{e} (z_e)_{k-1} \|_{2}^2 \bigg\}$ $\forall$ $i$ %\hfill\COMMENT{local decentralized}
\State $(z_e)_k \gets \frac{(x^i_t)_k + (x^j_t)_k}{2}$ 
%\hfill \COMMENT{local computation post sharing $(x^i)_{k}, (x^j)_{k}$ for $e = (i,j)$}
\EndFor
\State $x^i_t \gets \left(x^i_t \right)_{K_t}$ $\forall$ $i$
\EndFor
\end{algorithmic}
\end{algorithm}
\section{Main Results}\label{sec: main_results}
We now present the main contribution of this paper, the Alternating Coupled Online Regularized Descent (\acord{}) algorithm.
It is given above in Algorithm \ref{alg:ACORD}, and below we discuss its ramifications. 
\subsection{The \acordcaps{} algorithm}\label{subsec:acord}
Algorithm \ref{alg:ACORD} commands the following dynamics among the $N$ agents. In each round $t\in [T]$, agent $i$ receives $f_t^i(\cdot)$ along with the relevant dissimilarity costs $\{A^e_t\}_{e \ni i}$. Each agent $i$ maintains a collection of auxiliary variables $\{z_e\}_{e \ni i}$, where $e$ represents an edge that agent $i$ participates in. 

Now, each agent $i$ performs a pair of steps iteratively $K_t$ times. First it optimizes a completely local objective, with the shared auxiliary variables fixed, to get an estimate for its local action. Next, it communicates with its neighbors to update the auxiliary variables for each edge $(i,\cdot)$. This iterative local optimization and communication is much more efficient than neighborhood optimization done in \lpc{} \cite{lin2022decentralized}. We elaborate on this further in Section \ref{sec:lpc_comparison}.

We begin by addressing finite-iteration performance when there is a bound $M_f$ on the hitting costs and bounded diameter $M_s$ on the action space. While the performance is identical in the unbounded case, it requires one extra step, which we discuss at the end of this section. Consider,
\begin{align}\label{eqn:Kt_value}
    K_t &= \frac{\log \left(\frac{T^4 \cdot 128 \cdot\beta l N(M_f + M_s^2/2)(\max_i \mu_i + \lambda_1^i)}{(\sigma_t\min_i \lambda_1^i)^2} \right)}{\log\left( \frac{4\beta l}{4\beta l - \sigma_t} \right)} = \mathcal{O}(\log T)
\end{align}
where $\sigma_t$ is a function of strong-convexity parameters $\{\mu_1,\ldots,\mu_N\}$, $\beta$, and the graph $\mathcal{G}_t$. Setting $K_t$ as such gives the following asymptotic competitive ratio.
\begin{theorem}\label{main_thm:CR}
    \acord{} (Algorithm \ref{alg:ACORD}) along with $K_t$ as in \eqref{eqn:Kt_value} guarantees the following performance for any $T > 0$:
\begin{align*}
    \text{Cost}_{\acord{}}[1,T] \leq \underbrace{\left(\frac{\text{CR}_{*} + 1/2T^2}{1-1/2T^2} \right)}_{\text{CR}_\acord} \text{Cost}_{\opt{}}[1,T] + \frac{1+\beta (m^2/l)}{4T}
\end{align*}
where $\text{CR}_{*} = \frac{1}{2} + \frac{1}{2}\sqrt{1+\frac{4}{\min_i \mu_i}}$.
\end{theorem}
Here, $\text{CR}_{\acord{}}$ is bounded by $(2\text{CR}_{*}+1)$, converging to $\text{CR}_{*}$ for longer horizons. We defined $K_t$ as in \eqref{eqn:Kt_value} to highlight this converging trend to $\text{CR}_{*}$, with respect to $T$. In fact, this convergence is \textit{exponentially fast} with $K_t$, as remarked below, with detailed explanation in Appendix \ref{appendix_ssec: ACORD_CR} as a part of Theorem \ref{main_thm:CR}'s proof.
\begin{remark}\label{remark: error_term_decay}
For $\text{CR}_{\acord{}}$ to be $\epsilon$-close to $\text{CR}_{*}$, the following choice suffices
\begin{align}
    K_t = \mathcal{O}\left(\log \frac{1}{\epsilon}\right) \text{ } \forall \text{ } \epsilon>0.
\end{align}
\end{remark}
% \red{KEEP AS LESS ASSOCIATION WITH DECENTRALIZED ADMM AS POSSIBLE, TO AVOID REVIEWER CONFUSION}
% \red{
% \begin{remark}\label{remark: bounded_Kt}
%     In decentralized optimization, $K_t$ is a global parameter \cite{ShiWotao14ADMM,MaNikolakopoulos18fast,MaNikolakopoulos18graph,MaNikolakopoulos18hybrid} tied to convergence to the optimal $N$-tuple action. However, under the setting described above, $K_t$ can be set for all agents without them having any global information.
% \end{remark}
% }
\subsection{Lower Bound and Asymptotic Optimality}
While there are lower bounds established for the centralized setup \cite{GoelLinWierman19}, it remains unclear if the multi-agent setup or the presence of dissimilarity costs permit a ``lower" baseline performance bound.
\begin{theorem}\label{main_thm: lower_bound_2}
    In the decentralized setting, consider the agents' hitting costs $\{f^i_t\}_{i,t}$ to be $\mu_i$-strongly convex and the switching costs to be squared $\ell_2$-norm. Further, consider any dissimilarity cost that is convex in the actions. For weight $\beta > 0$, the competitive ratio is lower bounded as 
    \begin{align}
        \text{CR}_{ALG} \geq \frac{1}{2} + \frac{1}{2}\sqrt{1+\frac{4}{\min_i \mu_i}} = \text{CR}_{*}
    \end{align}
    with no helping effect from the dissimilarity cost.
\end{theorem}
We find that the adversary can still devise a strategy to circumvent the dissimilarity cost even in the case $\beta > 0$. The worst-case instance sets all $\mu_i = \mu$ and provides the same minimizer to all agents. Creating symmetry across the agents, the adversary takes advantage of a general characteristic of the dissimilarity cost: \textit{degeneracy along} $x^1 = \ldots = x^N$. This tactic allows the adversary to have a zero dissimilarity cost.
\begin{remark}
    In absence of dissimilarity costs, constraining the adversary through heterogeneity, by enforcing $\{\mu_i\}_i$ to be distinct, does not change the lower bound. Now, the adversary exploits the agent with the smallest strong-convexity parameter and forces the other agents to remain stationary in space, forcing any \alg{} to make poor decisions.
\end{remark}
We detail the exact worst-case instances and the dynamics that lead to these two results in Appendix \ref{appendix_sec: lower_bounds}. Theorem \ref{main_thm:CR}, in light of the aforementioned lower bound, reveals \acord{}'s following asymptotic optimality,
\begin{corollary}\label{corr: acord_asym_opt}
    \acord{} (Algorithm \ref{alg:ACORD}) with $K_t \to \infty$ produces the \textbf{online optimal} sequence of actions.    
\end{corollary}

The expression of $\text{CR}_\acord$ with $K_t \to \infty$ returns the lower bound we obtained in Theorem \ref{main_thm: lower_bound_2}. \acord{} is the first algorithm in decentralized SOCO literature exhibiting such tight guarantees, \textit{with the correct constants}, while relying only on local computations. Existing algorithms like \lpc{} \cite{lin2022decentralized} rely on exact future predictions and neighborhood optimization while not being able to close the gap with the online optimal (more on this in Sections \ref{sec:lpc_comparison} and \ref{sec:num_exp}). 

\subsection{Graph Dependence}
We study the effect of graph structure on \acord{} by analyzing its performance when $\mathcal{G}_t$ belongs to the class of $\mathcal{D}$-regular graphs. 
\begin{definition}
    An undirected graph $\mathcal{G}_t = ([N],\mathcal{E}_t)$ is said to $\mathcal{D}$-regular if each node has exactly $\mathcal{D}$ neighbors. Note that $\mathcal{D}$-regularity does not necessarily imply connectivity.
\end{definition}
Here, a clear and direct relationship can be established between the \textit{required} number of iterations $K_t$ and the common degree $\mathcal{D}$, as stated below.
\begin{theorem}\label{main_thm:graph_dependence}
    For $\mu_i = \mu$ and $\mathcal{D}$-regular graphs $\mathcal{G}_t$, $K_t$ scales \textbf{tightly} as $$K_t = \Theta\left(\mathcal{D}\log \left(N\mathcal{D}^2 T^4 \right)\right) \text{ } \forall \text{ } \mathcal{D} \geq 2,$$ for the performance guarantee in Theorem \ref{main_thm:CR}, where $\Theta(\cdot)$ is independent of $\mathcal{D}$, $N$ and $T$.
\end{theorem}
The above guarantee naturally extends to Remark \ref{remark: error_term_decay}, with $T^4$ replaced by $1/\epsilon^2$. The strong linear dependence on the graph’s edge density, indicated by the regular degree $\mathcal{D}$, and the weaker logarithmic dependence on the number of agents $N$ demonstrates that the coupling among agents via the graph $\mathcal{G}_t$ is the primary driving force behind algorithm design in this domain.
Consequently, the above analysis of \acord{} highlights the influence that the graph $\mathcal{G}_t$ has on decentralized SOCO.
\begin{remark}
    For general graphs, we find $K_t$ to loosely scale as
    \begin{align}
        K_t \propto \frac{\mathcal{D}_{\min}}{\frac{\mu + \lambda_1}{\beta} + \left(\mathcal{D}_{\min} - \sigma^{\mathcal{G}_t}_{\max}\right)}
    \end{align}
    where $\mathcal{D}_{\min}$ is the minimum degree and $\sigma^{\mathcal{G}_t}_{\max}$ is the maximum eigenvalue of unsigned Laplacian of the graph $\mathcal{G}_t$. Further ramifications on this dependence can be found in the proof of Theorem \ref{main_thm:graph_dependence} in Appendix \ref{appendix_sec: graph_dependence}.
\end{remark}
\subsection{Handling Unbounded Costs and Space}
In Section \ref{subsec:acord}, we asserted that Theorem \ref{main_thm:CR} remains unaffected by upper bounds on hitting costs and action space diameter. In this section, we formalize this.

Iterative algorithms for decentralized optimization, like ADMM \cite{ShiWotao14ADMM}, depend on the global starting point. The hyperparameter $K_t$ in Algorithm \ref{alg:ACORD} also exhibits such dependence,
\begin{align}\label{eqn: K_t_value_unbounded}
K_t = \frac{\log \left(\frac{T^4 \cdot 128 \cdot \beta l \mathbbm{F}_t((\mathbf{x}_{t})_0,\mathbf{z}_0)}{\left(\sigma_t\min_i \lambda_1^i\right)^2} \right)}{\log\left( \frac{4\beta l}{4\beta l - \sigma_t} \right)}.
\end{align}
All the agents must have the value of $\mathbbm{F}_t((\mathbf{x}_{t})_0,\mathbf{z}_0)$ to set $K_t$ at the start of round $t$. In the bounded case, this is a fixed quantity throughout the horizon and, hence, can be set preemptively, as stated in \eqref{eqn:Kt_value}. 

For the unbounded case, once each agent knows the above value of $K_t$ \eqref{eqn: K_t_value_unbounded}, the performance exhibited by the \acord{} algorithm is exactly as stated in Theorem \ref{main_thm:CR}. This is an attribute of our competitive analysis for \acord{} that leads to Theorem \ref{main_thm:CR}, which by default considers unbounded hitting costs and action space.
\setcounter{algorithm}{0}
\begin{subroutine} [ht]
\caption{NETWORK CRAWL}\label{alg:consensus}
\flushleft \textbf{Input}: Graph $\mathcal{G}_t$, $\left\{f^i_t(\mathbf{0}) + \frac{\lambda_1^i}{2}\|x^i_{t-1}\|_2^2 \right\}_{i=1}^N$
\flushleft \textbf{Initialize:} For each agent $i$, $Y^i_0 = \begin{bmatrix}
    0 & \ldots & \underbrace{\left(f^i_t(\mathbf{0}) + \frac{\lambda_1^i}{2}\|x^i_{t-1}\|_2^2 \right)}_i & \ldots & 0
\end{bmatrix}^T$
\begin{algorithmic}
\For{$m = 1,2,\ldots,D$}
    \State $Y^i_m \gets \text{Updated from }\{Y^i_{m-1}: i \in \mathcal{N}^i\}$
\EndFor
\State $K_t \gets \frac{\log \left(\frac{T^4 \cdot 32\beta l \left( \sum_{j=1}^N (Y^i_m)_j \right)\left(\max_i \mu_i + \lambda_1^i\right)}{\left(\sigma_t\min_i \lambda_1^i\right)^2} \right)}{\log\left( \frac{4\beta l}{4\beta l - \sigma_t} \right)}$ $\forall$ $i$
\end{algorithmic}
\end{subroutine}
This minor issue of $K_t$'s shared knowledge among agents can be solved by running the simple Sub-routine~\ref{alg:consensus}, as described below, before the iterative optimization starts. It will run for at most $D$ steps, where $D$ is the diameter of $\mathcal{G}_t$. It ends with all agents possessing $
    \mathbbm{F}_t((\mathbf{x}_{t})_0,\mathbf{z}_0) = \sum_{i=1}^N f^i_t(\mathbf{0}) + \frac{\lambda_1^i}{2}\|x^i_{t-1}\|_2^2,
$
and hence, $K_t$. This sub-routine handles unbounded hitting costs in unbounded decision space $\R^d$ while preserving \acord{}'s decentralized nature and the near-optimal competitive ratio in Theorem \ref{main_thm:CR}.

\section{\acordcaps{} vs. Model Predictive Approaches}\label{sec:lpc_comparison}

Our \acord{} algorithm is a significant leap from existing algorithms in decentralized online optimization literature. In this section, we highlight the drawbacks of the celebrated LPC approach via two aspects (i) \emph{achievable performance} and (ii) \emph{resource-efficiency}, by comparing it to \acord{}. We further highlight these through extensive numerical experiments in the next section.

The \lpc{} algorithm in \cite{lin2022decentralized} considers an online set-up closest to ours. The authors consider strongly convex hitting costs with Lipschitz gradients with switching costs and dissimilarity cost that are convex with Lipschitz gradients. A stationary graph $\mathcal{G}$ is assumed.
The environment considered in their work is one where each agent $i$ at round $t$ \textbf{communicates with an $\mathbf{r}$-hop neighborhood $\mathcal{N}_i^r$ around it} to access the following information, in addition to $f^i_t(\cdot)$: (i) exact future local hitting cost functions up to a prediction window of $k$ into the future, that is, $\{f^i_\tau\}_{\tau = t+1}^{t+k}$, (ii) exact future hitting cost functions of its entire $r$-hop neighborhood, that is, $\{f^j_\tau: j \in \mathcal{N}_i^r \backslash i, t \leq \tau \leq t+k\}$, where $\mathcal{N}_i^r = \{j \in [N]: d(i,j) \leq r$\} and, (iii) spatial cost functions of the entire sub-graph represented by the $r$-hop neighborhood.

Under this information access model, the \lpc{} algorithm solves the following minimization problem at each agent in each round,
\begin{align}\label{eqn: lpc formulation}
    \min \sum_{\tau = t}^{t+k-1} f_\tau^{\mathcal{N}_i^{r-1}}\left(\mathbf{x}_\tau^{\mathcal{N}_i^r}\right) + c_\tau^{\mathcal{N}_i^r}\left(\mathbf{x}_\tau^{\mathcal{N}_i^r},\mathbf{x}_{\tau-1}^{\mathcal{N}_i^r}\right)
\end{align}
over variables $\left\{\mathbf{x}_\tau^{\mathcal{N}_i^r}\right\}_{\tau=t}^{t+k-1}$, where $\mathbf{x}_\tau^{\mathcal{N}_i^r} = (x_\tau^i)_{i \in \mathcal{N}_i^r}$. The action taken by each agent is the $(t,i)^{th}$ variable, with the rest being discarded. Corollary 3.7 from \cite{lin2022decentralized}, stated below, summarizes \lpc{}'s performance and the overall effectiveness of this approach.
\begin{corollary}[3.7, \cite{lin2022decentralized}]\label{corr:lpc_sub_opt}
    Suppose the online optimal decentralized algorithm achieves a competitive ratio $c(k^*,r^*)$ with prediction horizon $k^*$ and communication radius $r^*$. As $k^* \to \infty$ and $r^* \to \infty$, \lpc{} achieves a competitive ratio at least as good as that of the online optimal decentralized algorithm when \lpc{} uses prediction horizon of $k = (4+o(1))k^*$ and communication radius $r = (16\Delta \log \Delta + o(1))r^*$, where $\Delta$ is largest degree in $\mathcal{G}$.
\end{corollary}

The above performance measure from \cite{lin2022decentralized} establishes that \lpc{} can only match the online-optimal when given \textit{unrestricted} information access across the network and the future horizon. Corollary \ref{corr: acord_asym_opt} already established that \acord{} is arbitrarily close to the online optimal while using only local and causal information. This suggests that \lpc{} can match \acord{}'s performance only in the limit of $k$ and $r$, more on which we will see in Section \ref{sec:num_exp}.

Next, we compare \acord{} and \lpc{} based on the resource-consumption and quantify it through the following result.

\begin{corollary}\label{corr: resource_utlization}
    \lpc{} with an $r$-hop communication radius and a future prediction window of $k$ is $\Omega \left(\frac{ |\mathcal{N}^r(\mathcal{D})|^3}{\mathcal{D}}k^3 \right)$ times more computationally expensive than \acord{}, where $|\mathcal{N}^r(\mathcal{D})|$ is the size of an $r$-hop neighborhood for any node in a $\mathcal{D}$-regular graph.
\end{corollary}
The proof, presented in Appendix \ref{proof: resource_comp}, is quite simple and uses the most basic hitting costs: quadratic, to quantify the computational load for the two algorithms. Even without future predictions ($k=1$),
\begin{align}
    \mathcal{D} < \mathcal{D}^3 \ll |\mathcal{N}_i^r|^3 = |\mathcal{N}^r(\mathcal{D})|^3
\end{align}
for large $\mathcal{D}$-regular networks.
\begin{remark}[Ring graph]
    When $\mathcal{G}_t$ is a ring graph, $\mathcal{D} = 2$ and $|\mathcal{N}^r(2)| = 2r+1$, and the aforementioned factor becomes $\mathbf{(2r+1)^3k^3}$, illustrating \lpc{$(r)$}'s computational overhead scaling polynomially with $r$.
\end{remark}

Next, we conduct a one-to-one performance comparison between \acord{} and \lpc{} to corroborate the above claims.

\section{Numerical Experiments}\label{sec:num_exp}
To further understand the performance of \acord{} relative to \lpc{}, and compare these two against naive decision techniques, we conduct empirical studies across various network topologies.
\subsection{Setup}
We build an online environment where the decision space is one-dimensional ($d=1$), the agents receive quadratic hitting costs $\alpha^i_t (x-v^i_t)^2$ and experience quadratic switching costs with homogeneous dissimilarity costs (weights are same for each edge). Each round and for each agent, the $\alpha^i_t$ is chosen randomly in $\R^+$ with sudden spikes throughout the horizon to create an adversarial environment. The minimizers $v^i_t$ are also generated in a similar fashion.

Although we can study various graph structures, we stick to $\mathcal{D}$-regular graphs. The reason is that we want to highlight the effect of the number of agents $N$, the graph density (represented by $\mathcal{D}$) and the effect of dissimilarity costs (through the weight $\beta$ on it) in these experiments. We explain the set-up in detail in Appendix \ref{appendix_sec: num_exp_descp}.

In addition to comparing \acord{} and \lpc{}, we also consider two naive algorithms: (i) $x^i_t = v^i_t$ and, (ii) $x^i_t = \argmin_x f^i_t(x) + \frac{1}{2}(x-x^i_{t-1})^2$, which we denote as Follow the Minimizer (\ftm{}) and Local (\loc{}), respectively. We compare these algorithms with total cumulative cost over the horizon $T$ as the performance metric. In all the results to follow, we consider $T$ to range from $1$ to $20$. We normalize the y-axis of all plots to $[0,1]$ as we only focus on relative performance among the algorithms.
\begin{figure}[H]% [hpbt] what you need
    \centering
     \subfigure[$K_t = 12$]{%
        \includegraphics[width=0.35\linewidth]{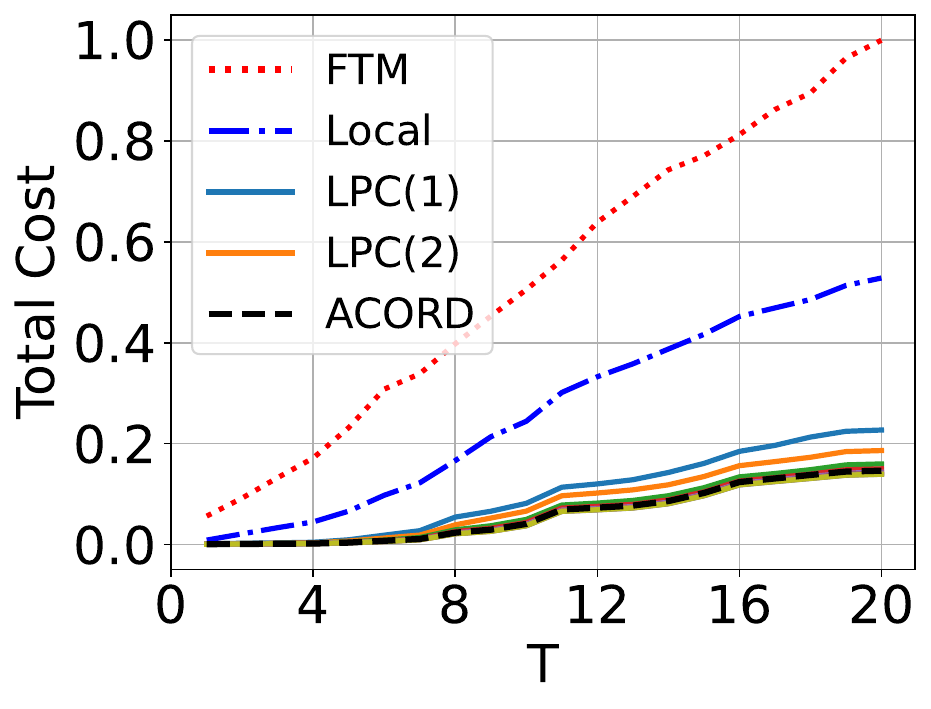}
    \label{fig:all_algs_1}}
     \subfigure[$N=20,\mathcal{D}=2$]{%{0.45\textwidth}
        \includegraphics[width=0.35\linewidth]{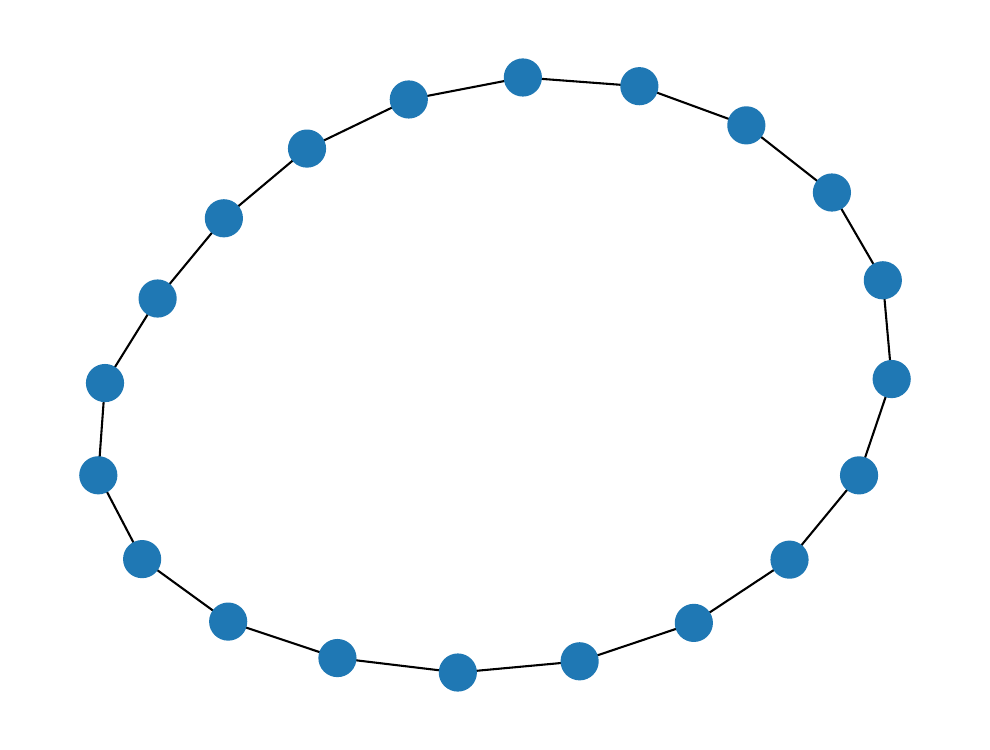}
        \label{fig:all_algs_2}}
    \caption{All algorithms' performance for  $N=20$, $\mathcal{D}=2$, $T=20$ and $\beta=50$}
    \label{fig:all_algs}
\end{figure}
From Figure \ref{fig:all_algs}, we observe that \ftm{} and \loc{} are much worse than nuanced approaches like \acord{} and \lpc{}. The poor performance of the former two can be attributed to the lack of attention to dissimilarity costs, with \ftm{} being worse by being agnostic to switching costs too. 
%This plot does not give us any new insights other than a suggestion that \acord{} performs better than \lpc{}, which we investigate further below.
\subsection{\acordcaps{} vs. \lpccaps{}}

We now solely focus on the empirical performance comparison of \acord{} with the state-of-the-art \lpc{} algorithm. In the following experiments, we consider variants of \lpc{} that do not have future predictions. This is to ensure the comparison with \acord{} is fair, as the latter too does not access future hitting costs. Further, we consider \lpc{$(r)$} with per-agent network-access ranging from $r=1$-hop (immediate neighborhood) to as large as $r=d_{\max}$-hop (full graph), where $d_{\max}$ is the diameter of the graph $\mathcal{G}$.

It is worthwhile to note that only \lpc{$(1)$} and \acord{} are comparable in terms of network resources, as in both protocols, agents are restricted to contact only their \textit{immediate neighbors}. Note that for $r>1$, \lpc{$(r)$} utilizes network and agents' information \emph{not} available to \acord{}.

In all the experiments, we choose the minimum $K_t$ such that \acord{} matches the performance of \lpc{$(d_{\max})$}. We make this choice of $K_t$ to highlight the performance of \acord{} against the most powerful state-of-the-art.

\paragraph{Common observation.} Across all values of $N$, $\mathcal{D}$, $\beta$ and $T$ we see that the total cost of \acord{}, is always lower than that of \lpc{$(r)$} for any $r>0$. 
%$K_t$ given in \eqref{eqn:Kt_value} is for the worst-case instance, often not encountered in practice, as is evident here. We also observe that the performance of \lpc{$(r)$} in the limit $r \to d_{\max}$, matches that of \acord{}.
This corroborates Corollary \ref{corr:lpc_sub_opt} [3.7 from \cite{lin2022decentralized}] that \lpc{} is always worse than the online-optimal, which we proved to be \acord{} asymptotically, in Corollary \ref{corr: acord_asym_opt}.
\begin{figure}[H]% [hpbt] what you need
    \centering
     \subfigure[$\beta = 10, K_t = 6$]{%
        \includegraphics[width=0.35\linewidth]{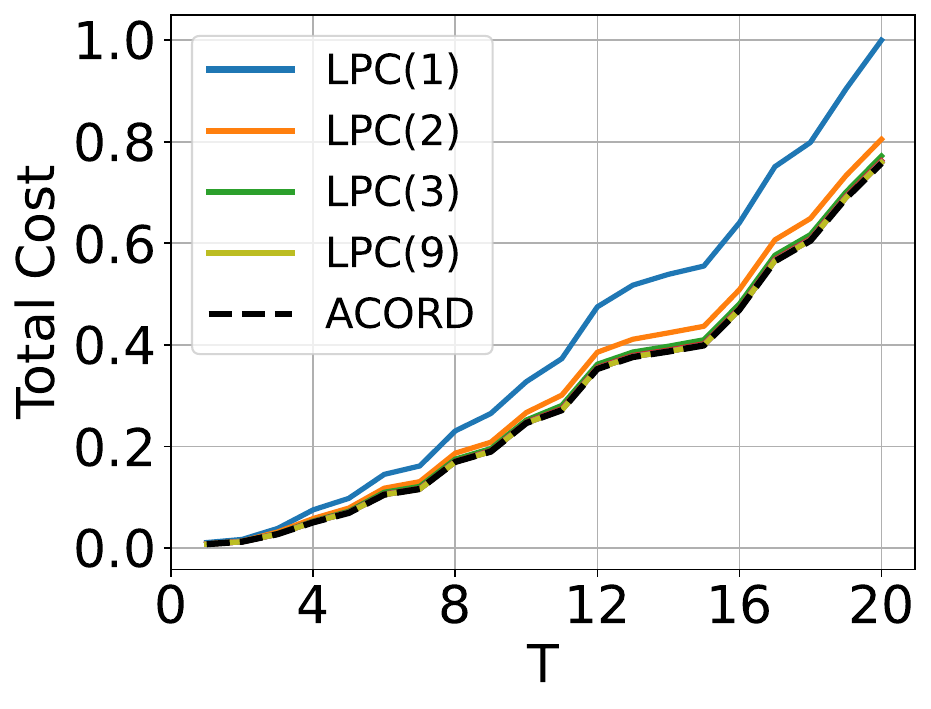}
        \label{fig:beta_sweep_1}}
     \subfigure[$\beta = 50, K_t = 12$]{%{0.45\textwidth}
        \includegraphics[width=0.35\linewidth]{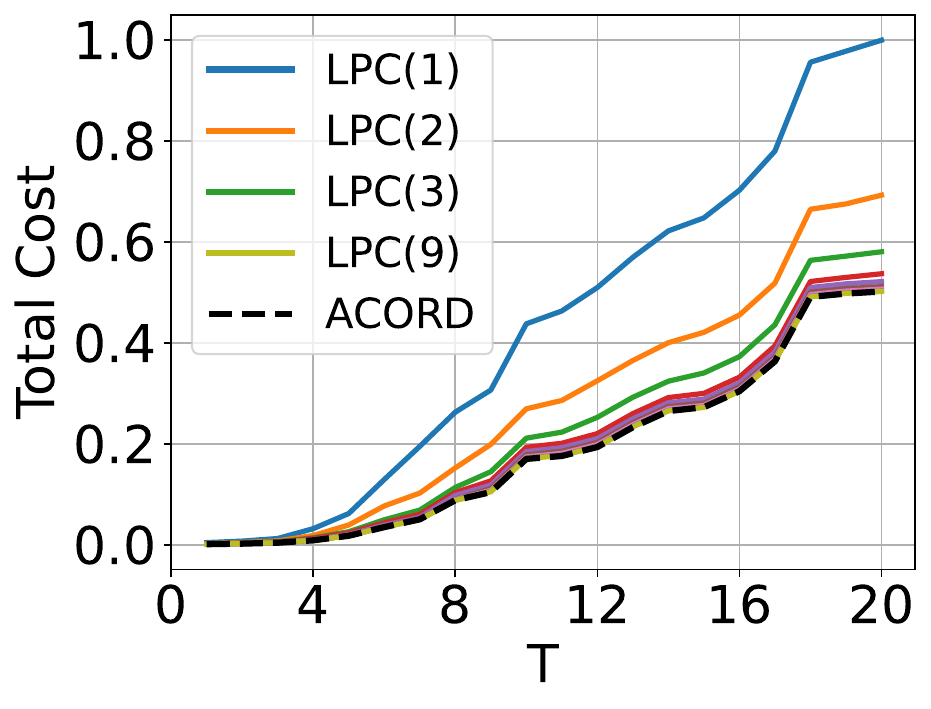}
        \label{fig:beta_sweep_2}}\\
    \subfigure[$\beta = 250, K_t = 15$]{%{0.45\textwidth}
        \includegraphics[width=0.35\linewidth]{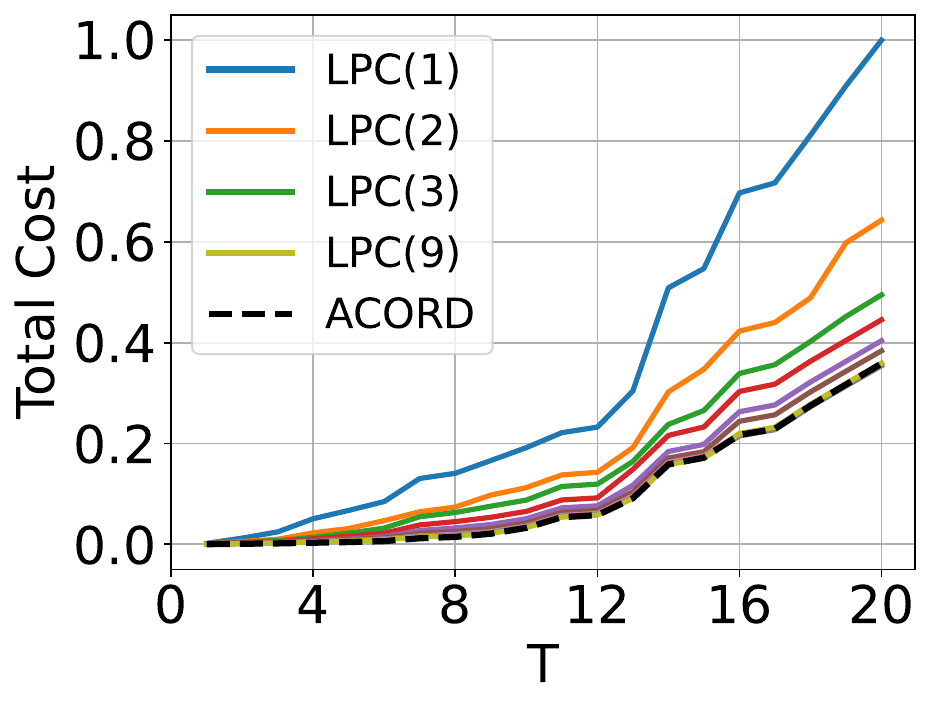}
        \label{fig:beta_sweep_3}}
     \subfigure[$\frac{\text{Cost}_{\acord{}}[1,20]}{\text{Cost}_{\lpc{(1)}}[1,20]} (\beta)$]{%
    \includegraphics[width=0.35\linewidth]{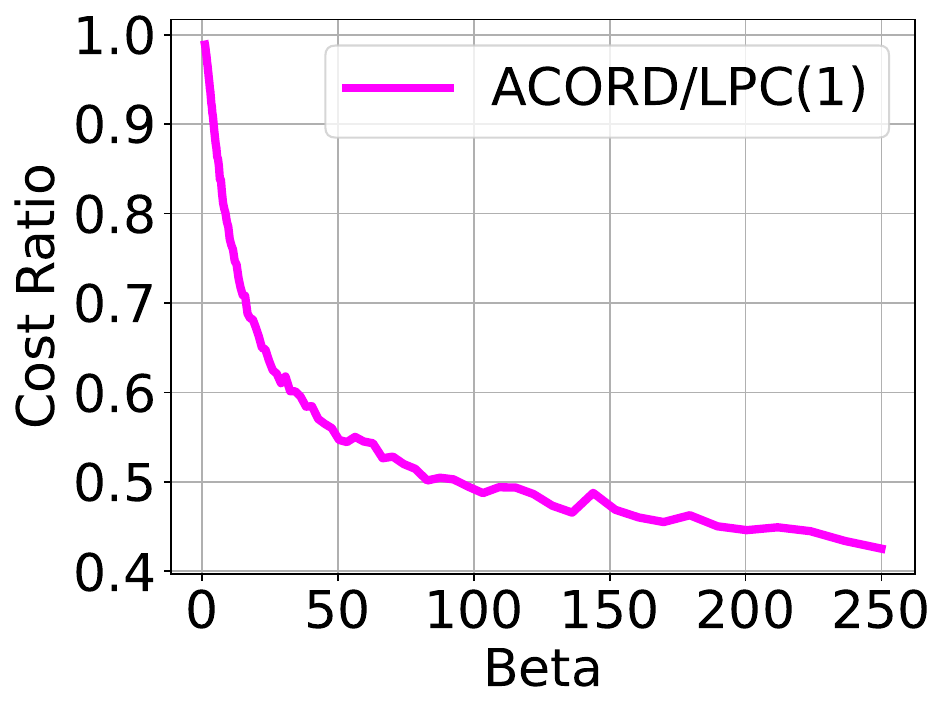}
    \label{fig:beta_sweep_trend}}
    \caption{$N=20$, $\mathcal{D}=2$, $T=20$ and different $\beta$ values.}
    \label{fig:beta_sweep}
\end{figure}
\paragraph{Trend with $\beta$.} Figures \ref{fig:beta_sweep_1}, \ref{fig:beta_sweep_2}, \ref{fig:beta_sweep_3} show that as $\beta$ increases, the gap between \lpc{(1)} and \acord{} increases. This trend is verified in Figure \ref{fig:beta_sweep_trend} where we see the progressively superior performance of \acord{} over the state-of-the-art as $\beta$ increases. Since $\beta$ is the weight on the dissimilarity cost, it highlights that \acord{} \textit{handles these costs much more efficiently} compared to \lpc{}.

Before studying the trend with graph edge density $\mathcal{D}$, we first explain the non-trivial relationship between $\mathcal{D}$ and the maximum possible $r$ ($d_{\max}$, graph diameter).

\begin{remark}\label{rem: D_and_r}
     Fixing $N$, in a $\mathcal{D}$-regular graph, $d_{\max} \downarrow$ as $\mathcal{D} \uparrow$, with maximum being $\left \lfloor \frac{N}{2} \right\rfloor $ for $\mathcal{D}=2$ and minimum being $1$ for $\mathcal{D}=N-1$. Hence, Figures \ref{fig:D_sweep_1}, \ref{fig:D_sweep_2}, \ref{fig:D_sweep_3} have progressively less number of plots.
\end{remark}
\paragraph{Trend with graph density $\mathcal{D}$.} Figures \ref{fig:D_sweep_1}, \ref{fig:D_sweep_2}, \ref{fig:D_sweep_3} indicate a decreasing gap between \lpc{(1)} and \acord{} as $\mathcal{D}$ increases, further verified by Figure \ref{fig:D_sweep_trend}. In sparser graphs, where communication-restricted decentralization is more difficult, we see that \acord{} does a much better job than \lpc{} (see lower $\mathcal{D}$ values in Figure \ref{fig:D_sweep_trend}). This further goes to show that \acord{} handles decentralization much more robustly than \lpc{}.
\begin{figure}[]% [hpbt] what you need
    \centering
     \subfigure[$\mathcal{D} = 2, K_t=12$]{%
        \includegraphics[width=0.35\linewidth]{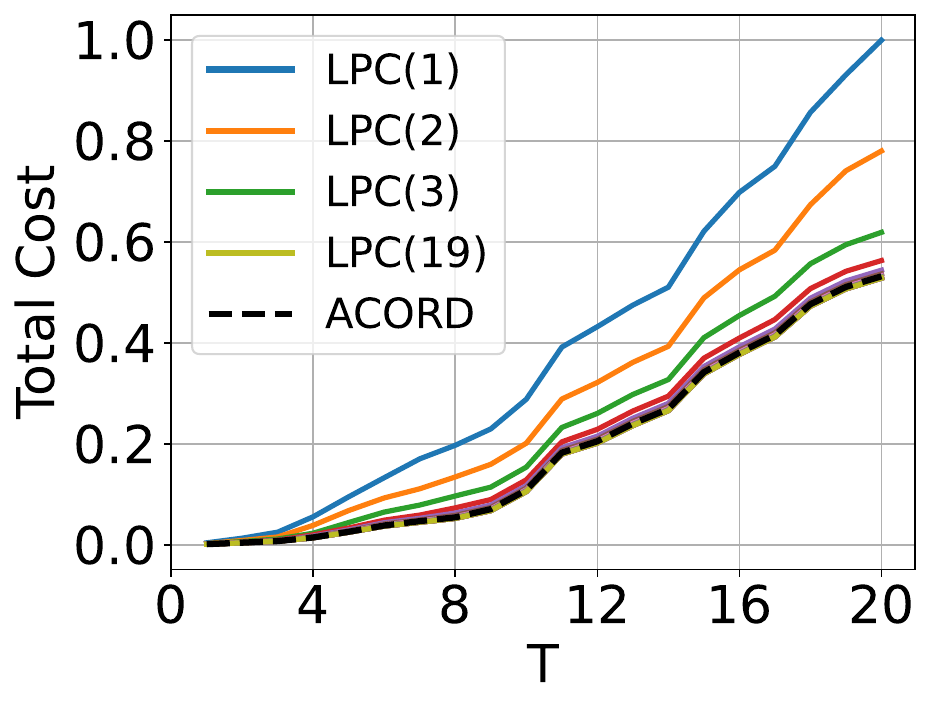}
        \label{fig:D_sweep_1}}
     \subfigure[$\mathcal{D}= 10, K_t=12$]{%{0.3\textwidth}
        \includegraphics[width=0.35\linewidth]{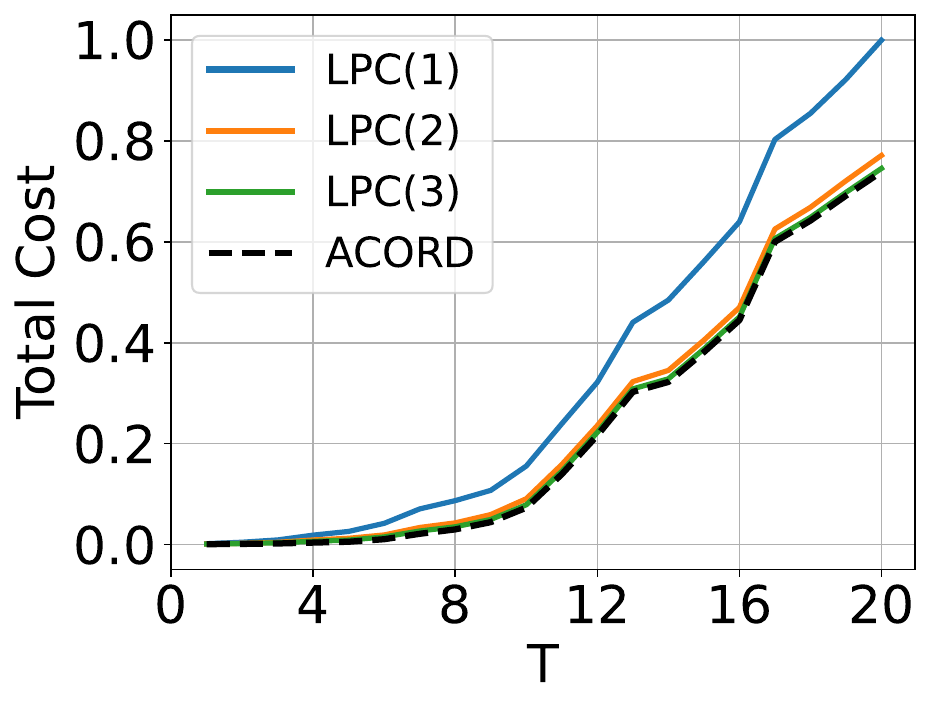}
        \label{fig:D_sweep_2}}\\
    \subfigure[$\mathcal{D}= 39, K_t=12$]{%{0.3\textwidth}
        \includegraphics[width=0.35\linewidth]{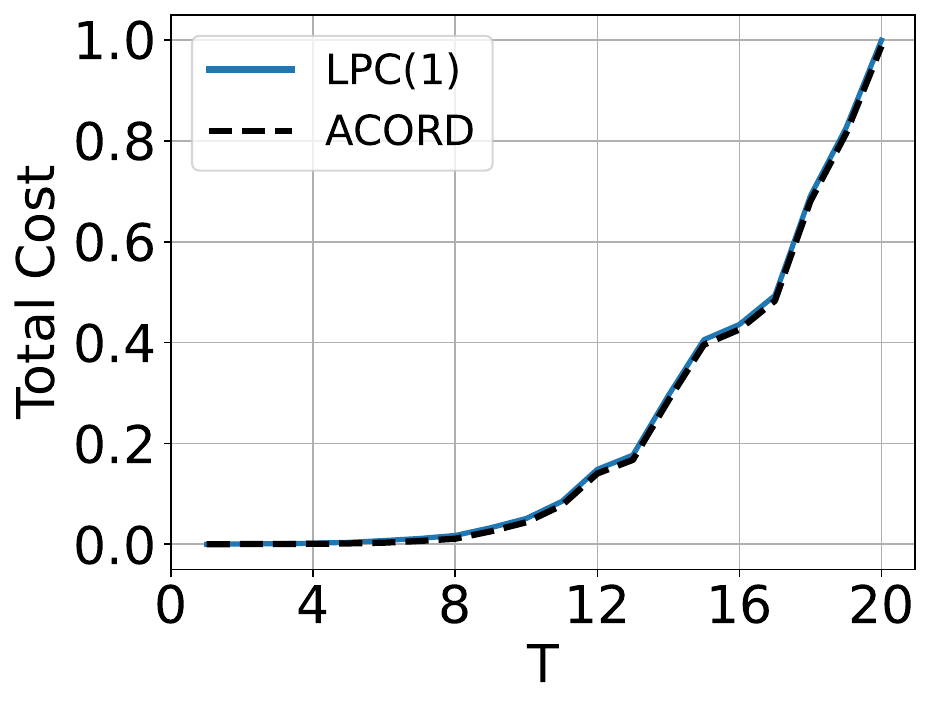}
        \label{fig:D_sweep_3}}
     \subfigure[$\frac{\text{Cost}_{\lpc{(1)}}[1,20]}{\text{Cost}_{\acord{}}[1,20]}(\mathcal{D})$]{%
    \includegraphics[width=0.35\linewidth]{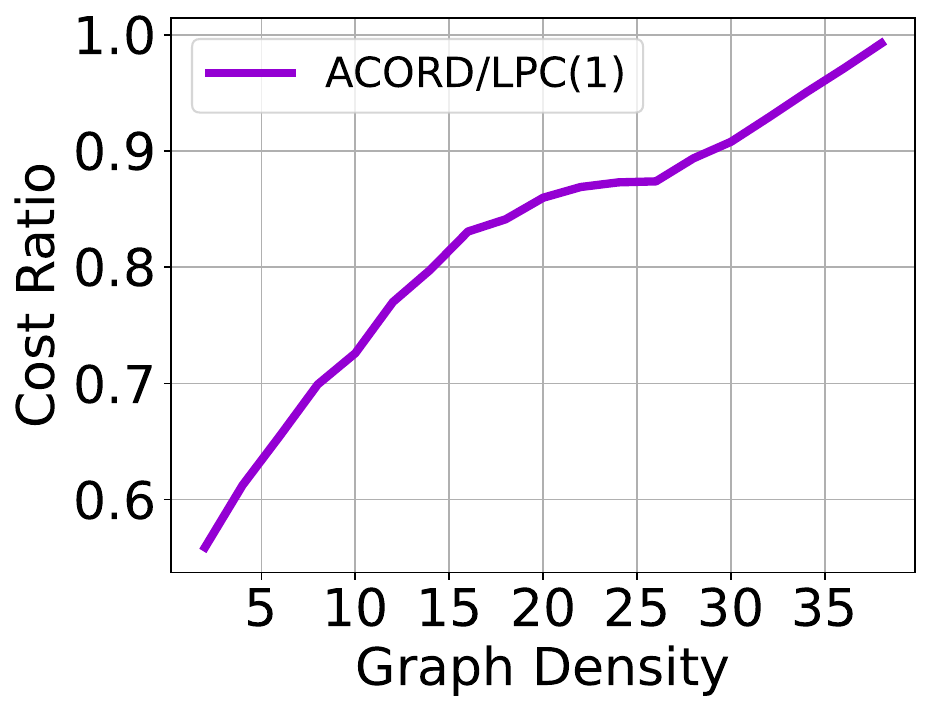}
    \label{fig:D_sweep_trend}}
    \caption{$N=40$, $\beta=50$, $T=20$ and different $\mathcal{D}$ values.}
    \label{fig:D_sweep}
\end{figure}
\paragraph{Trend with $N$.} The performance gap's trend with increasing $N$ is dependent on the graph's sparsity, in the sense, how $\mathcal{D}$ is related to $N$. We first look at Figure \ref{fig:N_sweep_trend}. Observe that for the sparsest ring graph $(\mathcal{D}=2)$, \acord{} is \textit{uniformly better} than \lpc{} irrespective of network size. On the other end, for the fully-connected case $(\mathcal{D}=N-1)$, \lpc{} is close to \acord{} for small and large networks. Finally, for various levels in between, that is $\mathcal{D} = N/4, N/2, 3N/4$, we see that \acord{} is superior to \lpc{} for \textit{smaller networks}, with the gap reducing as network size increases. Figures \ref{fig:N_sweep_1}, \ref{fig:N_sweep_2}, \ref{fig:N_sweep_3} represent how relative performance behaves for increasing $N$ with $\frac{\mathcal{D}}{N}$ set to $1/4$.

\begin{figure}[]% [hpbt] what you need
    \centering
     \subfigure[$N=20, \mathcal{D} = 5,K_t=6$]{%
        \includegraphics[width=0.35\linewidth]{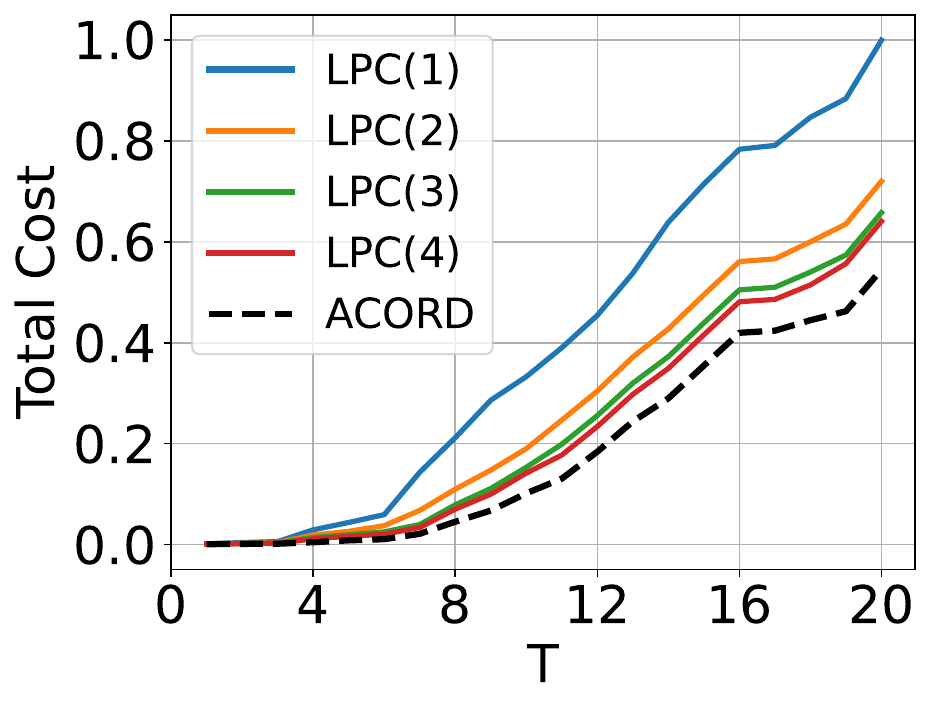}
        \label{fig:N_sweep_1}}
     \subfigure[$N=60, \mathcal{D}= 15, K_t=6$]{%{0.3\textwidth}
        \includegraphics[width=0.35\linewidth]{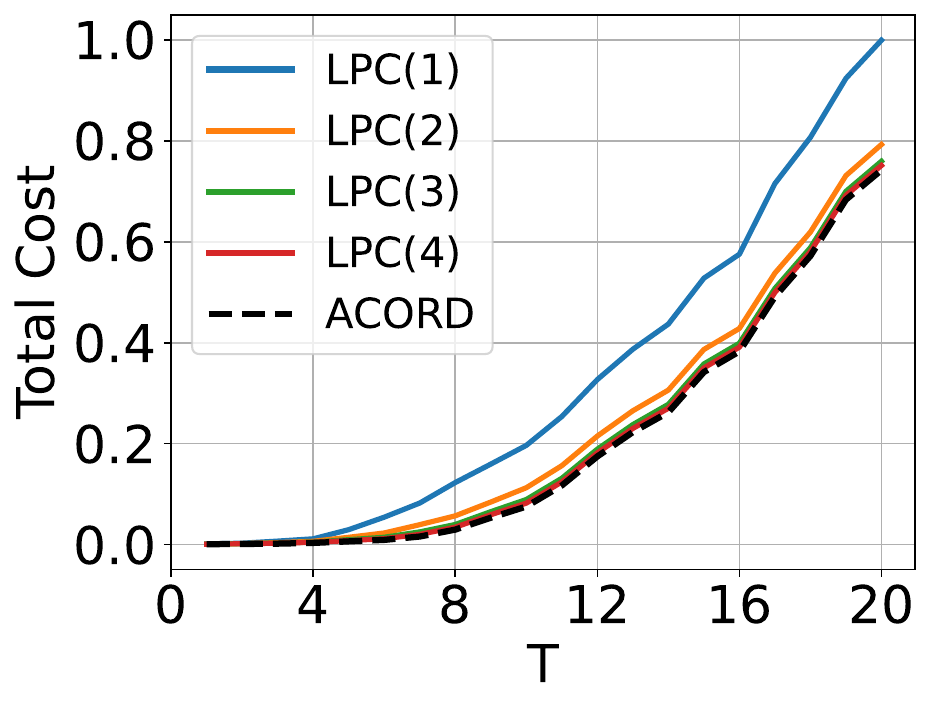}
        \label{fig:N_sweep_2}}\\
    \subfigure[$N=200, \mathcal{D}= 50,K_t=6$]{%{0.3\textwidth}
        \includegraphics[width=0.35\linewidth]{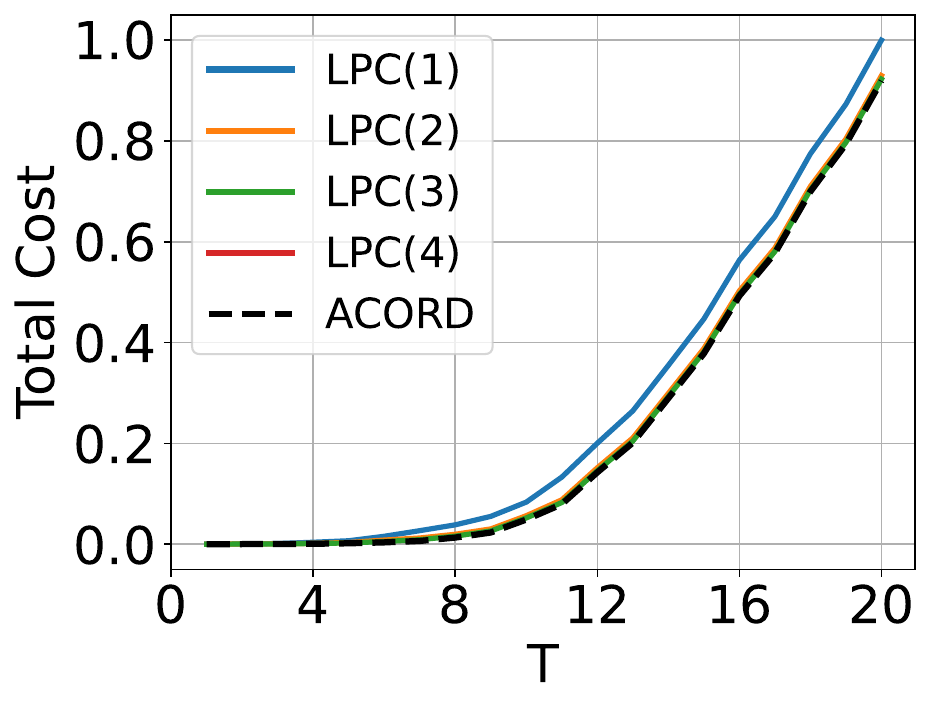}
        \label{fig:N_sweep_3}}
     \subfigure[$\frac{\text{Cost}_{\lpc{(1)}}[1,20]}{\text{Cost}_{\acord{}}[1,20]}(N,\mathcal{D})$]{%
    \includegraphics[width=0.35\linewidth]{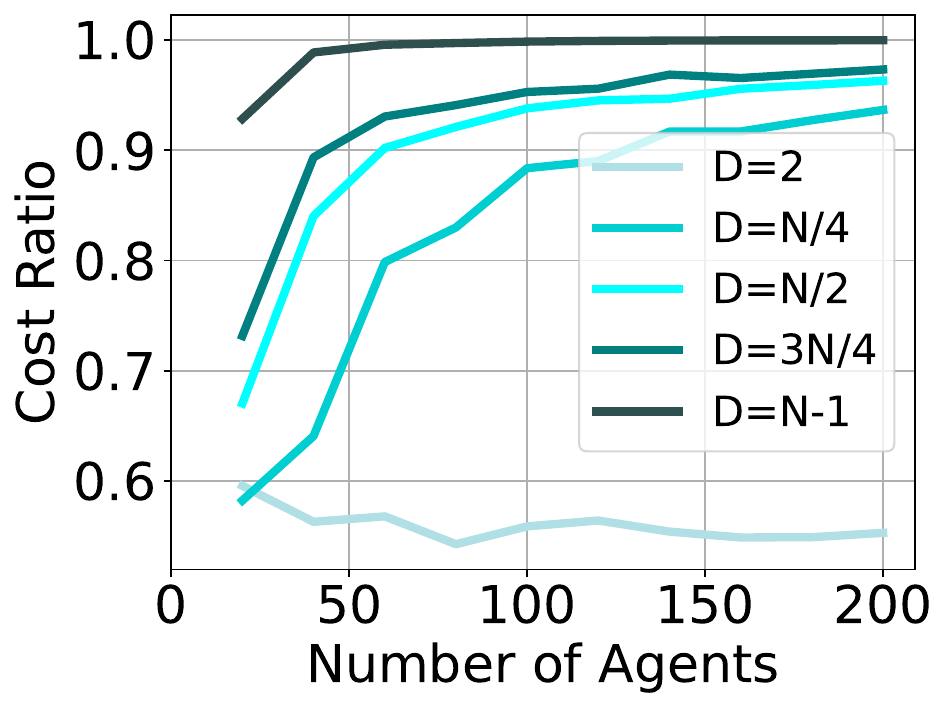}
    \label{fig:N_sweep_trend}}
    \caption{\acord{} v/s \lpc{} for $T=20$, $\beta=50$, and different $N$ and $\mathcal{D}$. The trend varies with how $\frac{\mathcal{D}}{N}$ is set.}
    \label{fig:N_sweep}
\end{figure}
\subsection{Empirical Resource efficiency} 
The two algorithms in question take two different approaches to solve the per-round decentralization, \acord{} performs an iterative scalable-optimization and \lpc{} does a one-shot neighborhood optimization. To understand which approach is faster in practice, we will now perform a timing-analysis of \acord{} against \lpc{}.

Corollary \ref{corr: resource_utlization} suggests that a denser graph highlights the resource-efficiency of \acord{}. We corroborate it here by comparing the average runtimes of \acord{} and \lpc{$(r)$} for $r \in \{1,\ldots,d_{\max}\}$ for various values of graph edge densities $\mathcal{D} \in \{2,\ldots,N-1\}$. We perform this runtime comparison on a single thread on an Apple M1 pro chip. We calculate the runtime of \acord{} by vectorizing then dividing the average run-time by $N$ for the per-agent runtime, as matrix-inversion of an $N$-entry diagonal matrix is $\mathcal{O}(N)$ on a single thread. For \lpc{}, we compute the action for each agent sequentially, so average agent runtime is calculated by again normalizing by $N$. Since we build a homogeneous graph $\mathcal{G}$ over the agents, compute-time is expected to be similar across the agents. Detailed explanation of the simulation set-up can be found in Appendix \ref{appendix_sec: num_exp_descp}.

The metric we consider here is $\frac{\tau_{\acord{}}}{\tau_{\lpc{(r)}}}$, with $\tau_{\alg{}}$ defined as \alg{}'s run-time per agent averaged across the horizon $T$ and across the sample runs. It allows us to infer the relative resource-efficiency of \acord{} to \lpc{$(r)$} per agent. We take $N=50$ and $\beta=50$ for this comparison, with $K_t = 15$ for \acord{}. Figures \ref{fig:perf_comp_3} and \ref{fig:perf_comp_4} confirm that this $K_t$ value is enough for \acord{} to beat \lpc{$(r)$} for any $r>0$.
\begin{figure}[]% [hpbt] what you need
    \centering
     \subfigure[Increasing $\mathcal{D}$]{%
        \includegraphics[width=0.35\linewidth]{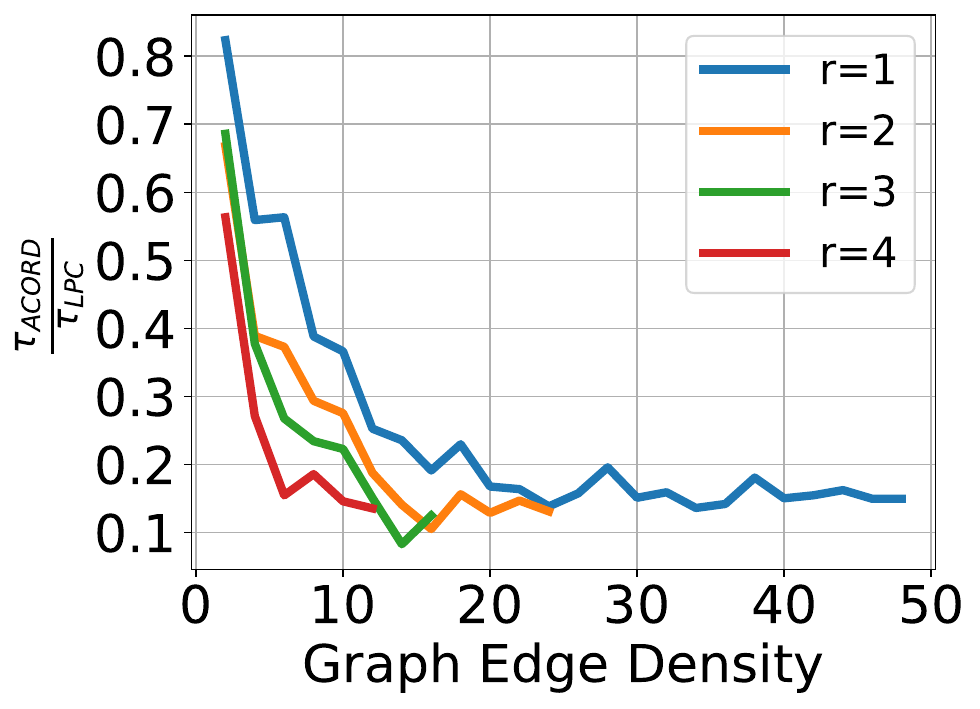}
        \label{fig:perf_comp_1}}
     \subfigure[Increasing $r$]{%{0.3\textwidth}
        \includegraphics[width=0.35\linewidth]{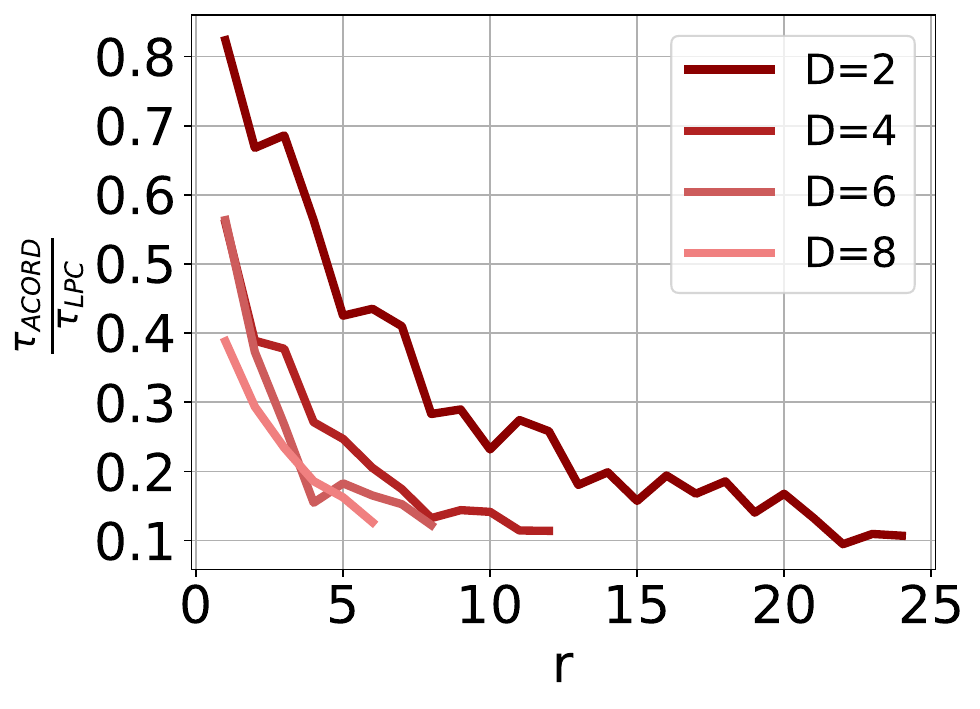}
        \label{fig:perf_comp_2}}\\
    \subfigure[$\mathcal{D}= 2,K_t=15$]{%{0.3\textwidth}
        \includegraphics[width=0.35\linewidth]{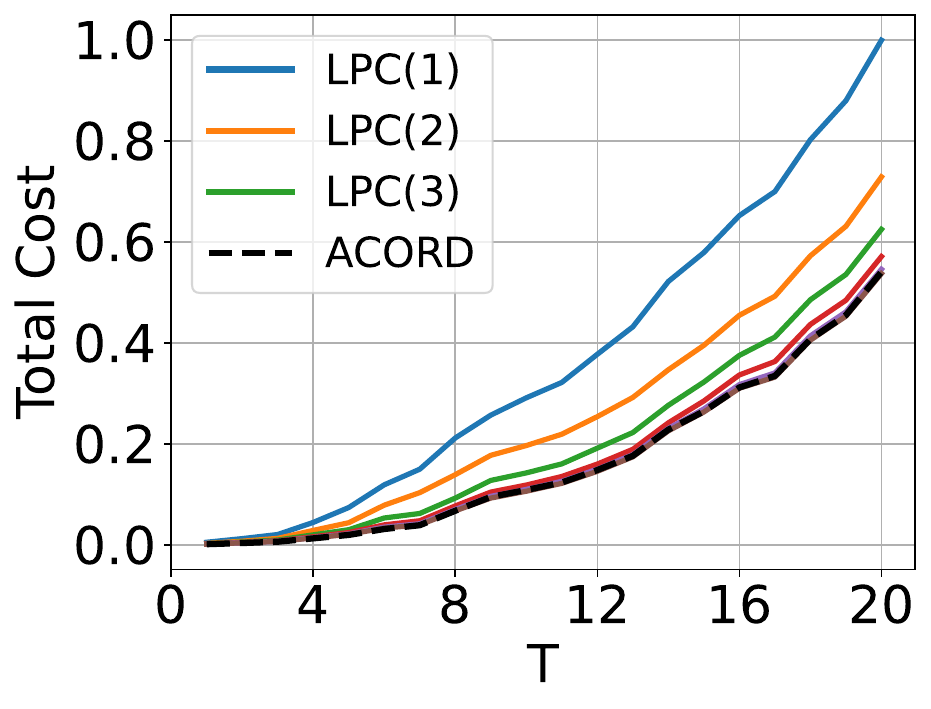}
        \label{fig:perf_comp_3}}
     \subfigure[$\mathcal{D}=8, K_t=15$]{%
    \includegraphics[width=0.35\linewidth]{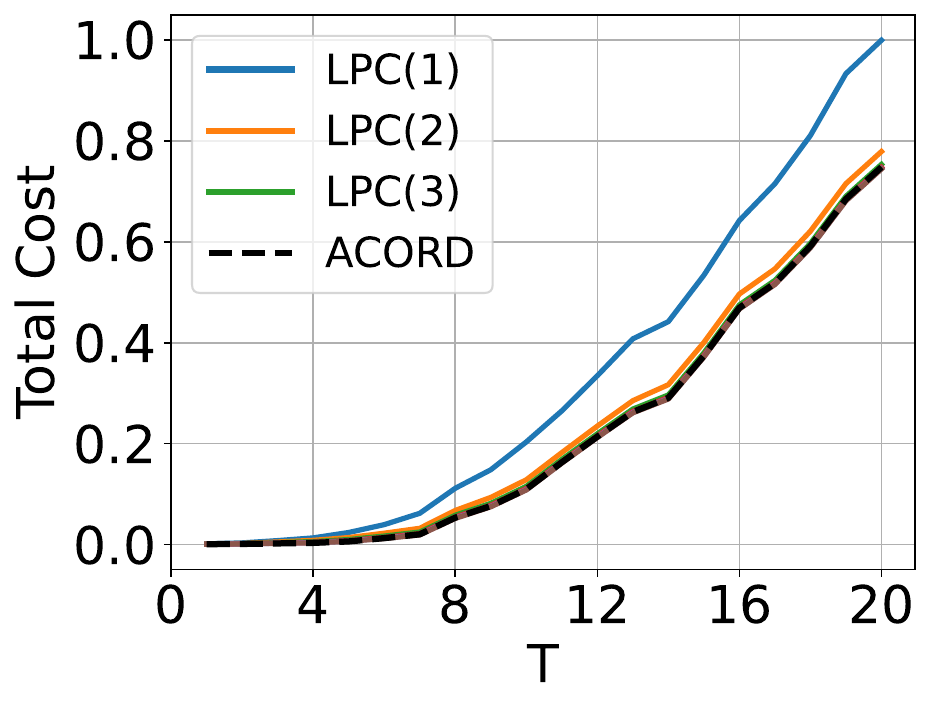}
    \label{fig:perf_comp_4}}
    \caption{\acord{} v/s \lpc{} for $N=50$, $T=20$, $\beta=50$, and different $r$ and $\mathcal{D}$.}
    \label{fig:perf_comp}
\end{figure}
Figure \ref{fig:perf_comp_1} illustrates the computational-expense ratio established in Corollary \ref{corr: resource_utlization}: we see that \lpc{(r)} is progressively slower to \acord{} as $\mathcal{D}$ increases, for any $r>0$. Figure \ref{fig:perf_comp_2}'s trend is rather obvious: \lpc{$(r)$} is progressively slower than \acord{} for increasing $r$. Finally, the varying x-axis ranges in \ref{fig:perf_comp_1} and \ref{fig:perf_comp_2} plots are reflective of the feasible $\mathcal{D}$ and $r$ values for a fixed $N$ (recall Remark \ref{rem: D_and_r}).
\subsection{Summary}
From Section \ref{sec:lpc_comparison} and the numerical experiments, it is understood that \acord{} outperforms the state-of-the-art \lpc{} across the board:
\begin{enumerate}
    \item Only \lpc{$(1)$} has the same level of network access per agent as \acord{}. In fact, \lpc{$(1)$} also employs function-level information sharing. Despite this, \acord{} performs markedly better than \lpc{$(1)$} across all settings. 
    \item For $r$ $>$ $1$, the comparison between \lpc{$(r)$} and \acord{} is in fact unfair, as \lpc{$(r)$} involves function-level information sharing across $\mathbf{r}$\textbf{-hop} neighborhoods for each agent. \acord{} on the other hand uses strictly local information and only $1$-hop vector-sharing among agents. Still \acord{} manages to beat \lpc{$(r)$} for all $r$ $>$ $0$ with the latter able to match only in the limit $r \to d_{\max}.$
    \item \acord{} is able to provide such superior performance with markedly lower runtime (Figures \ref{fig:perf_comp_1}, \ref{fig:perf_comp_2}), which we provably establish in Corollary \ref{corr: resource_utlization}.
\end{enumerate}

Such a leap of performance is only possible through a new approach to the online decentralization problem. \lpc{} involves neighborhood optimization and concepts from receding horizon control. However, it does not attempt to \textit{de-couple} the online objective each round. This is where we believe lies the fundamental error. Offline decentralization literature \cite{ShiWotao14ADMM} points out that decentralization is possible only for a decoupled objective.

Our approach to decentralization starts from this fundamental requirement: de-coupling. Once it is achieved, we design a decentralization protocol. Finally, we tie all this into the SOCO framework. This three-part design process gives rise to new analytical contributions, some of which go beyond the set-up considered in this work. We explain these in detail in Appendix \ref{appendix_sec: analytical_contribs}, right after the main paper.
\section{Conclusion}
In this work, we present the first decentralized algorithm in multi-agent SOCO that handles dynamic dissimilarity costs and adaptive networks without exchanging hitting cost functions among agents. In addition to asymptotic optimality, our algorithm is deployable in practice with guaranteed near-optimal performance. Tight guarantees on $K_t$'s network dependence give a complete picture of our algorithm's resource utilization. Better cost minimization and runtime performance than the state-of-the-art highlight the effectiveness of our approach.

The above guarantees are attributed to the use of auxiliary variables $\{z_e\}_{e\in\mathcal{E}_t}$, Alternating Minimization techniques and a unifying approximation framework, allowing convergence to the online optimal action without the exchange of hitting costs among agents.

\bibliography{refs_icml}
\bibliographystyle{abbrv}

%%%%%%%%%%%%%%%%%%%%%%%%%%%%%%%%%%%%%%%%%%%%%%%%%%%%%%%%%%%%%%%%%%%%%%%%%%%%%%%
%%%%%%%%%%%%%%%%%%%%%%%%%%%%%%%%%%%%%%%%%%%%%%%%%%%%%%%%%%%%%%%%%%%%%%%%%%%%%%%
% APPENDIX
%%%%%%%%%%%%%%%%%%%%%%%%%%%%%%%%%%%%%%%%%%%%%%%%%%%%%%%%%%%%%%%%%%%%%%%%%%%%%%%
%%%%%%%%%%%%%%%%%%%%%%%%%%%%%%%%%%%%%%%%%%%%%%%%%%%%%%%%%%%%%%%%%%%%%%%%%%%%%%%
\newpage
\onecolumn
\begin{appendices}

\section{Examples of Dissimilarity Costs in Practice}\label{appendix_sec: examples}
We present application examples to highlight the practicality of our model in this section: (i) UAV swarm control with local collision avoidance (ii) dynamic multi-product pricing, (iii) decentralized battery management and (iv) geographical server provisioning. 

\subsection{Formation Control for UAV Swarm with Local Collision Avoidance}
Formation control refers to the coordination of multiple agents, such as robots, drones, or autonomous vehicles, to achieve and maintain specific spatial arrangements while performing tasks. It involves the use of algorithms and control strategies to regulate the relative positions and orientations of the agents within the formation. This is crucial in applications like surveillance, search and rescue, and environmental monitoring, where teamwork enhances efficiency and effectiveness. Formation control typically relies on principles such as communication between agents, sensing of relative positions, and centralized or decentralized decision-making to ensure robustness and adaptability in dynamic environments.

Specifically, distance-based formation control has advantages in scenarios involving low communication bandwidth and low memory/power. The desired formation is achieved by actively controlling the distances between agents, based on the specified target inter-agent distances. Each agent is assumed to have the capability to sense the relative positions of its neighboring agents within its own local coordinate system. These local coordinate systems may have different orientations and are not required to be aligned \cite{Kwang-Kyo15}. This independence from the global coordinate system allows decentralization in formation control.

Specifically, \cite{dorfler2009formation,krick2009stabilisation,kuriki2015formation} use a distance-based penalty function of the form:
\begin{align}\label{eqn: dron_quad_diss_1}
    \gamma^a_{ij}(\|p_i-p_j\|) = k_p (\|p_i-p_j\|^2 - \|p^*_i-p^*_j\|^2)^2,
\end{align}
where $k_p>0$, $p_i,p_j$ are $i^{th}$ and $j^{th}$ agent's positions and $\|p^*_i-p^*_j\|$ is the desired distance between these two agents. Works like  \cite{dimarogonas2008stability,dimarogonas2010stability} further normalize the above penalty function by the squared inter-agent distance, that is 
\begin{align}\label{eqn: dron_quad_diss_2}
     \gamma^b_{ij}(\|p_i-p_j\|) = k_p \frac{(\|p_i-p_j\|^2 - \|p^*_i-p^*_j\|^2)^2}{\|p_i-p_j\|^2}
\end{align}
to ensure collision avoidance between neighboring agents, as the penalty blows up when $\|p^*_i-p^*_j\| \to 0$. Functions of the above form can be easily decoupled in the following manner:
\begin{align}\label{eqn: dron_diss_cost_1}
\begin{split}
    \gamma^a_{ij}(\|p_i-z_{ij}\|) &= \frac{k_p}{2} (4\|p_i-z_{ij}\|^2 - \|p^*_i-p^*_j\|^2)^2\\
    \gamma^a_{ij}(\|p_j-z_{ij}\|) &= \frac{k_p}{2} (4\|p_j-z_{ij}\|^2 - \|p^*_i-p^*_j\|^2)^2    
\end{split}
\end{align}
with $z_{ij} = \frac{p_i + p_j}{2}$. Similarly, $\gamma^b_{ij}$ can be modified without bias as follows:
\begin{align}\label{eqn: dron_diss_cost_2}
\begin{split}
    \gamma^b_{ij}(\|p_i-z_{ij}\|) &= \frac{k_p}{8} \frac{(4\|p_i-z_{ij}\|^2 - \|p^*_i-p^*_j\|^2)^2}{\|p_i-z_{ij}\|^2}\\
    \gamma^b_{ij}(\|p_j-z_{ij}\|) &= \frac{k_p}{8} \frac{(4\|p_j-z_{ij}\|^2 - \|p^*_i-p^*_j\|^2)^2}{\|p_j-z_{ij}\|^2}.
\end{split}
\end{align}

Non-linear model-predictive control (MPC) has been extensively studied in relation to various aspects of UAV or quad-rotor control \cite{ru2017nonlinear,kang2009linear,gavilan2015iterative,saccani2022multitrajectory,chao2012uav,baca2016embedded}. Simpler models in the form of Linear-Quadratic-Regulator (LQR) control has also been explored \cite{rinaldi2013linear,elkhatem2022robust,khatoon2014pid,lee2011modeling,alkhoori2017pid,budiyono2015control,kim2017dynamic}. Decentralized/distributed control of UAV swarms has also garnered attention \cite{richards2004decentralized,wehbeh2020distributed,bemporad2011decentralized,yuan2017outdoor,bassolillo2020decentralized}. However, most of these works do not have worst-case guarantees or assume the noise to be simple, like Gaussian, instead of adversarial. Further, as seen in \cite{lin2022decentralized}, predictive control based decentralization is very computationally expensive, while suffering from bias. 

\cite{LiGuannan20} and \cite{lin2021perturbation} have already shown that smoothed online convex optimization (SOCO) and predictive control are equivalent problems with general assumptions on the switching costs, like strong convexity and smoothness. Combining it with dissimilarity costs in \eqref{eqn: dron_quad_diss_1} or \eqref{eqn: dron_quad_diss_2} completes a SOCO model for UAV swarm control that avoids local collisions. With the decoupling shown in \eqref{eqn: dron_diss_cost_1} and \eqref{eqn: dron_diss_cost_2} and our proposed \acord{} method in Algorithm \ref{alg:ACORD} helps low-powered UAVS achieve formation control.

%\red{WRITE ABOUT LQR BASED UAV CONTROL AND CONNECT LQR WITH SOCO USING \cite{lin2021perturbation}}

\subsection{Dynamic Multi-product Pricing}
This class of problems study an organization selling $N$ different products, that seeks to optimize its revenue by dynamically adjusting its catalog prices in response to shifts in the market over the horizon $T$. Each agent $i$ in graph $\mathcal{G}_t$ represents a product with price $x^i_t$ at time $t$. The non-stationary relationship between two complementary products $i$ and $j$ is modelled by an edge in a dynamically changing $\mathcal{G}_t$. Pricing models used in \cite{TalluriGarrett06,CandoganOzdaglar12,GallegoTopaloglu19} reflect our set-up.

Solving the global multi-product pricing problem precisely can be highly challenging in practice, particularly when dealing with large networks. For instance, major online e-commerce platforms handle millions of products, making it difficult to store the entire network, let alone perform complex computations on it. Additionally, since prices can be adjusted easily, these companies frequently implement dynamic pricing strategies, often updating prices daily or even more frequently, further increasing the computational complexity. Decentralized computation is, therefore, necessary to reduce the compute resources needed at the back-end for solving these large-scale problems.

\cite{lin2022decentralized} presents a form of this problem, similar to \cite{CandoganOzdaglar12}, where the demand model per product $i$ is given as
\begin{align}
    d^i_t = a_t^i - k_t^i x_t^i - \sum_{j \in \mathcal{N}^1_i\backslash i} \eta_t^{(j \to i)} x_t^j + b_t^i x_{t-1}^i
\end{align}
where $a_t^i, k_t^i, b_t^i > 0$, respectively representing minimum demand, negative effect of the current price and positive effect of the previous price. Further, it is common to assume to $\eta_t^{(j \to i)} < 0$ as the high price of a related item $j$ increases the demand of item $i$. The complete revenue maximization problem can be written as,
\begin{align}
    \max_{\{\mathbf{x}_t\}_t} \sum_{t=1}^T \sum_{i=1}^N x^i_t d^i_t
\end{align}
which can be translated to the following minimization problem
\begin{align}
    \min_{\{\mathbf{x}_t\}_t} \sum_{t=1}^T \left(\sum_{i=1}^N \xi_t^i \left(x^i_t - \frac{a^i_t}{\xi^i_t}\right)^2 + \frac{b^i_t}{2} (x^i_t - x^i_{t-1})^2 + \sum_{(i,j) \in \mathcal{E}_t} |\gamma_t^{(i,j)}|(x^i_t - x^j_t)^2 \right)
\end{align}
where $\gamma^{(i,j)}_t = \frac{\eta_t^{(j\to i)} + \eta_t^{(i\to j}}{2}$ and $\xi_t^i = k^i_t - \sum_{j \in \mathcal{N}_t^1\backslash i} |\gamma^{(i,j)}_t| - \frac{b_t^i + b_{t+1}^i}{2} > 0$. It is justified to consider parameters such that $\xi_t^i > 0$ as the current price of the item $i$ has the strongest effect $(k_t^i)$ on the demand. Comparing to the generalized objective considered in our model in \eqref{eqn: cumulative_obj}, we see that \acord{} is directly applicable to this set-up. Further details can be found in \cite{lin2022decentralized}.

\subsection{Decentralized Battery Management}
Although renewable energy sources like solar and wind are ideal for powering sustainable data centers, their availability is highly unpredictable due to fluctuating weather conditions. This inconsistency presents significant challenges in meeting the constant energy needs of data centers. To address this, large-scale energy storage systems, made up of multiple battery units, have become essential for capturing and utilizing intermittent renewable energy. However, managing such a complex system for maximum efficiency is no easy task. Each battery must independently regulate its charging and discharging cycles to maintain its energy levels within an optimal range (e.g., 35-75\%). At the same time, keeping the state-of-charge (SoC) levels uniform across all battery units is critical for prolonging the overall lifespan of the batteries and maximizing energy efficiency \cite{ZeraatiGolshan16,ahmadi2021optimum,zhao2022distributed}. This challenge can be effectively illustrated by adapting a standard form to our model. Each battery unit independently adjusts its state of charge (SoC) by charging or discharging, which results in a localized cost that reflects how far its SoC deviates from the desired range. Simultaneously, a temporal cost arises from the fluctuations in SoC resulting from these charging and discharging actions. 
The SoC difference between battery units $i$ and $j$, based on their physical connection in the network $\mathcal{G}_t$, can negatively impact performance and lifespan. For example, variations in battery voltage due to differing SoCs can lead to overheating or even battery damage \cite{wang2015balanced}. Spatial costs represent this need for coordination by penalizing deviations in charge levels.

As the number of coordinated devices grows, centralized power management approaches \cite{jabr2017linear} become computationally inefficient and rely significantly on expensive communication systems. Meanwhile, relying solely on locally available information results in instability and sub-optimal performance \cite{bolognani2019need}. Communication-assisted decentralized control techniques have proven to be far more effective \cite{liu2017distributed,tang2018fast,fan2019consensus,xu2020accelerated,xie2020robust}, highlighting the need for communication-efficient decentralized online optimization. Further details can be found in \cite{LiShaoleiWierman23}.

\subsection{Environmental-Aware Geographical Load Balancing}
Online service providers now rely heavily on a web of data centers strategically positioned near users to ensure low-latency services. However, the trade-off for this proximity is the considerable energy demand these facilities impose. To handle varying user loads throughout the day, it is crucial to continuously adjust the number of active servers in each data center, aligning server usage with energy efficiency goals while keeping operational expenses under control \cite{WiermanAndrew11,LinWierman12,WangHuang14,radovanovic2022carbon,kwon2020ensuring,khalil2022renewable}. Although increasing the number of active servers can enhance service responsiveness, it also escalates energy consumption, contributing to a larger environmental footprint.

Beyond merely reducing the total environmental impact of these data centers, it is imperative to consider environmental justice. This means minimizing the unequal distribution of the negative ecological effects caused by data center operations across different regions \cite{USDoEjust}. Ignoring such inequalities can pose significant risks to business stability and provoke unintended societal challenges \cite{metasustain}. Given this context, we can conceptualize the network of data centers as a set of interconnected agents $i$ within a dynamic graph $\mathcal{G}_t$. Each agent's varying local environmental impact is captured through a time-dependent hitting cost $f^i_t(\cdot)$. The energy costs associated with adjusting server capacities are quantified by the switching costs, $\frac{1}{2}\|x^i_t - x^i_{t-1}\|_2^2$, representing the costs of bringing servers online or taking them offline \cite{LinWiermanAndrew11}. Additionally, the equity in impact across locations is represented through dissimilarity costs, $\left\|A^{(i,j)} x^i_t - A^{(i,j)} x^j_t \right\|_2^2$, where $A^{(i,j)}$ models the relationship between the geographical locations $i$ and $j$. This model provides a holistic approach to balancing service quality, energy consumption, and regional environmental impacts within a distributed data center network. 

The optimization of the global objective, however, is very high dimensional given the complexity of a data-center operations \cite{buyya2010energy,sun2016optimizing} and the number of datacenters spread across the world \cite{azure,aws}. Hence, a centralized operation is not possible due to the scale and the communication costs involved, necessitating our decentralization methodology and communication model discussed in Section \ref{sec:model_prelims}.

\section{Analytical \& methodological contributions}\label{appendix_sec: analytical_contribs}
Our results in Sections \ref{sec: main_results}, \ref{sec:lpc_comparison} and \ref{sec:num_exp} of the main paper are consequence of the new methodologies and analytical tools we introduce to Decentralized SOCO and the broader area of SOCO. In this section, we take a deeper dive into these, while contrasting with existing techniques.
\subsection{Decoupling before decentralization}
Traditional work in decentralized optimization \cite{ShiWotao14ADMM,MaNikolakopoulos18fast,MaNikolakopoulos18graph,MaNikolakopoulos18hybrid} lays down a fundamental requirement for this process: \textbf{decoupling}. It means that the multi-agent objective one is looking to minimize in a \textit{localized} fashion (with communication) \textit{cannot have direct coupling} among agents' decision variables.

This is where \lpc{} algorithm takes a wrong turn, as it takes the coupled objective as is, and performs neighborhood optimization within each agent. Not only does it lead to computational and communication inefficiency, this method evidently \textit{cannot achieve optimality} even with exact future predictions, as shown in Corollary \ref{corr:lpc_sub_opt}.

Our first line of action, in the design of \acord{}, is to solve this unaddressed issue in Decentralized SOCO. We approach it from the fundamentals in offline decentralized optimization, introducing separability among agent decision variables in the minimized objective, so that local minimization is possible. Below, we present an alternative optimization objective formulation whose solution achieves the optimal performance of a fully centralized algorithm while allowing separability of decisions across agents. 
\begin{theorem}\label{main_thm:no_bias_CR}
    Denote 
    \begin{align}
        \begin{split}
    \mathbb{F}_t(\mathbf{x},\mathbf{z}) &=
    \sum_{i=1}^N \bigg\{ f_t^i(x^i) + \frac{\lambda_1^i}{2}\|x^i - x_{t-1}^{i}\|_2^2 + \beta \sum\limits_{\mathcal{E}_t \ni e \ni i}\|A_t^{e} x^i - A_t^{e}z_e\|_{2}^2 \bigg\},
\end{split}
    \end{align}
   where $\lambda_1^i = 2/\left(1 + \sqrt{1+4/\mu_i} \right)$. Then the sequence of actions $\{\Tilde{\mathbf{x}}_t\}_{t=1}^T$ obtained by solving the following at each round 
   \begin{equation}
   \label{eqn: decentralized_robd_obj}
       \Tilde{\mathbf{x}}_t, \mathbf{z}_* = \argmin_{\mathbf{x} = (x^i)_{i=1}^N, \mathbf{z} = (z_e)_{e \in \mathcal{E}_t}} \mathbb{F}_t(\mathbf{x},\mathbf{z})
   \end{equation}
   is guaranteed to have the optimal competitive ratio of
    \begin{align}
        \frac{1}{2} + \frac{1}{2}\sqrt{1 + \frac{4}{\min_i \mu_i}}.
    \end{align}
\end{theorem}

Recall the competitive ratio lower bound of $\left(\frac{1}{2} + \frac{1}{2}\sqrt{1 + \frac{4}{\min_i \mu_i}}\right)$ in Theorem \ref{main_thm: lower_bound_2}. That performance is in fact achieved by an existent \textit{centralized} online-optimal algorithm \robd{} \cite{GoelLinWierman19}. However, our setting does not allow its implementation due to it highly centralized and coupled nature. 

Now, The variables $\{z_e\}_e$ corresponding to each edge ${e \in \mathcal{E}_t}$ helps to decouple the agents' actions. The big question is the careful placement of these variables to ensure no inherent bias introduced when we are changing the optimization domain from $\mathbf{x} \in \R^{Nd}$ to $(\mathbf{x},\mathbf{z}) \in \R^{Nd + |\mathcal{E}_t|d}$.

We circumvent the bias-problem by focusing on the optimality condition per-round and making changes at that level, and then bring it to objective level. The centralized optimality condition for each agent $i$ is
\begin{align}\label{eqn:centralized_optimality}
    &\nabla_{x^i} f_t^i(x^i) + \lambda_1^i (x^i - x^i_{t-1}) + \beta \sum_{\{j : (i,j) \in \mathcal{E}_t\}} (A^{(i,j)}_t)^T A^{(i,j)}_t \left[ A^{(i,j)}_t x^i - A^{(i,j)}_t x^j \right]= \mathbf{0}.
\end{align}
We want to decouple this entity, which can be done as
\begin{align}
        &\nabla_{x^i} f_t^i(x^i) + \lambda_1^i (x^i - x^i_{t-1}) + 2\beta \sum_{\mathcal{E}_t \ni e \ni i} (A^{(i,j)}_t)^T A^{(i,j)}_t \left[ A^{(i,j)}_t x^i - A^{(i,j)}_t z_e \right] = \mathbf{0} \label{eqn33} \\
        &(A^{(i,j)}_t)^T A^{(i,j)}_t \left[ A^{(i,j)}_t z_e - A^{(i,j)}_t x^i \right]+ (A^{(i,j)}_t)^T A^{(i,j)}_t \left[ A^{(i,j)}_t z_e - A^{(i,j)}_t x^j \right] = \mathbf{0}.
\end{align}
where the second condition amounts to $z_e = \frac{x^i + x^j}{2}$, owing to $(A^{(i,j)}_t)^T A^{(i,j)}_t \succ 0$. Putting it back in \eqref{eqn33} gives back \eqref{eqn:centralized_optimality}. The complete proof of Theorem \ref{main_thm:no_bias_CR}, which also explains the centralized optimality condition \eqref{eqn:centralized_optimality}, can be found in Appendix \ref{appendix_sec: decoup_robd}

\subsection{Decentralization without Lipschitz gradients}
Now, that we have an objective $\mathbb{F}_t(\mathbf{x},\mathbf{z})$ where agent decisions are not directly entangled, we can proceed to the decentralization process. Let's look at the $x^i$ component of $\mathbb{F}_t(\mathbf{x},\mathbf{z})$,
\begin{equation} \label{eqn: decentral_obj_indiv}
\begin{split}
     x^i_t = \argmin_{x^i} \sum_{i=1}^N &\bigg\{ f_t^i(x^i) + \frac{\lambda_1^i}{2}\|x^i - x_{t-1}^{i}\|_2^2 + \beta \sum\limits_{\mathcal{E}_t \ni e \ni i}\|A_t^{e} x^i - A_t^{e}z_e\|_{2}^2 \bigg\}.
\end{split}
\end{equation}
and it is clear that $x^i$ still has implicit coupling with $\{x^j: j\in \mathcal{N}_i\}$ through the auxiliary variables. The only way to solve $\Tilde{\mathbf{x}}_t, \mathbf{z}_* = \argmin \mathbb{F}_t(\mathbf{x},\mathbf{z})$ locally within the agents is by separating the optimization in $\mathbf{x}$ and $\mathbf{z}$.

Here, we bring in block optimization techniques, specifically the \textit{Alternating Minimization (\am{})} method \cite{BeckTetruashvili13}. The key idea is that solving \eqref{eqn: decentralized_robd_obj} iteratively separates the minimization in $\mathbf{x}$ and $\mathbf{z}$ while ensuring convergence towards the $\Tilde{\mathbf{x}}_t$. Starting from any point $((\mathbf{x}_t)_0,\mathbf{z}_0)$, chosen at the user's convenience, the iterative minimization proceeds as follows: At the $k$-th iteration,
\begin{align}
    (\mathbf{x}_t)_k = \argmin_\mathbf{x} \mathbb{F}_t(\mathbf{x},\mathbf{z}_{k-1})\label{eqn21}\\
    \mathbf{z}_k = \argmin_\mathbf{z} \mathbb{F}_t((\mathbf{x}_t)_k,\mathbf{z})\label{eqn22}.
\end{align}
The first minimization step, given by \eqref{eqn21},  is separable in $\mathbf{x}$, thanks to $\{z_e\}_{e\in \mathcal{E}_t}$, and can be performed \textit{locally} by each agent:
\begin{align}\label{eqn23}
    \left(x^i_t \right)_k = &\argmin\limits_{x \in \R^d} \bigg\{ f_t^i(x) + \frac{\lambda_1^i}{2}\|x - x_{t-1}^{i}\|_2^2 + \beta \sum\limits_{\mathcal{E}_t \ni e \ni i}\left\|A^e_t x - A^e_t (z_e)_{k-1}\right\|_{2}^2 \bigg\} \text{ } \forall \text{ } i.
\end{align}
Now, observe that the next minimization step \eqref{eqn22} translates to
\begin{align}\label{eqn20}
\begin{split}
    (A_t^{e})^T \times \left[ A_t^{e}(z_e)_k - \frac{A_t^{e}(x^i_t)_k + A_t^{e}(x^j_t)_k }{2} \right] = \mathbf{0}
\end{split}
\end{align}
for each edge $e = (i,j) \in \mathcal{E}_t$.
Since $(A_t^{e})^T A_t^{e} \succ 0$, this has the following unique solution:
\begin{align}\label{eqn: z-block_obj}
    \left(z_{(i,j)}\right)_k = \frac{(x^i_t)_k + (x^j_t)_k}{2}.
\end{align}
Communication of $(x^i_t)_k$ and $(x^j_t)_k$ across the edge $e = (i,j)$ allows each agent to calculate this average locally.

Recall the applications of SOCO and its decentralized forms we discussed in the literature. The hitting costs observed do not necessarily exhibit Lipschitz gradients. It is for this reason, worst-case guarantees of related algorithms in the literature \cite{ChenWierman15,ChenWierman16,ChenWierman18,GoelLinWierman19,GoelWierman19,LinGoelWierman20,ChristiansonWierman22,RuttenMukherjee23,ChristiansonWierman23} do not make this assumption. 

Our decentralization goal is to accommodate for non-smooth hitting costs, and the following result reflects our choice of \am{}:
\begin{proposition}\label{prop: alt_min_covergence}
Consider $\mathbbm{F}_t(\mathbf{x}_t,\mathbf{z}) = \sum_{i=1}^N f^i_t(x^i_t) + \frac{\lambda_1^i}{2}\|x^i_t - x^i_{t-1}\|_2^2 + \beta \sum_{\mathcal{E}_t \ni e \ni i} \|A^e_t x^i_t - A^e_t z_e\|_2^2,$ where $f^i_t$ is $\mu_i$-strongly convex and \textbf{possibly non-smooth}. Performing $K_t$ iterations of 
\begin{align}
    (\mathbf{x}_t)_k &= \argmin_\mathbf{x} \mathbb{F}_t(\mathbf{x},\mathbf{z}_{k-1})\\
    \mathbf{z}_k &= \argmin_\mathbf{z} \mathbb{F}_t((\mathbf{x}_t)_k,\mathbf{z})
\end{align}

allows the following convergence guarantee
\begin{align*}
    \mathbbm{F}_t(&(\mathbf{x}_t)_{K_t},\mathbf{z}_k) - \mathbbm{F}_t(\mathbf{\Tilde{x}}_t,\mathbf{z}_*) \leq \left(1-\frac{\sigma_t}{4\beta l} \right)^{K_t-1} \left(\mathbbm{F}_t((\mathbf{x}_t)_0,\mathbf{z}_0) - \mathbbm{F}_t(\Tilde{\mathbf{x}_t},\mathbf{z}_*)\right),
\end{align*}
where $\mathbf{\Tilde{x}}_t,\mathbf{z}_* = \argmin_{\mathbf{x},\mathbf{z}} \mathbbm{F}_t(\mathbf{x},\mathbf{z})$ and $\sigma_t$ is the strong convexity parameter of $\mathbbm{F}_t$.
\end{proposition}
It is common knowledge that iterative optimization methods require Lipschitz gradients to guarantee descent towards the optimum at every iteration. However, we avoid this requirement for the hitting costs $\{f^i_t\}_{i=1}^N$ by utilizing a unique property of \am{}, that Lipchitz gradients are required only with respect to one of the blocks $\{\mathbf{x},\mathbf{z}\}$ and,
\begin{lemma}\label{main_lemma: smoothness}
$\mathbbm{F}_t(\mathbf{x},\mathbf{z})$ is $4\beta l$ smooth in $\mathbf{z}$.    
\end{lemma}
The contraction in $\mathbbm{F}_t$ is a function of $\min \{L_1,L_2\}$ where $L_1$ and $L_2$ are smoothness coefficients of $\mathbf{x}$ and $\mathbf{z}$ respectively. Lemma \ref{main_lemma: smoothness} allows $L_1 \to \infty$, meaning, non-smooth hitting costs $\{f^i_t\}_{i,t}$. The proof of the Proposition \ref{prop: alt_min_covergence} and the above lemma can be found in Appendix \ref{appendix_sec: decental_robd_l2}.

\subsection{SOCO Approximation framework}
Theorem \ref{main_thm:no_bias_CR} and Proposition \ref{prop: alt_min_covergence} together hint that $K_t \to \infty$ allows \acord{} to have a competitive ratio of $CR_* = \frac{1}{2} + \frac{1}{2}\sqrt{1+\frac{4}{\min_i \mu_i}}$, as we note in Corollary \ref{corr: acord_asym_opt}.

In practice, $K_t \to \infty$ cannot be implemented. However, the need for worst-case guarantees in applications like power grid operations, data-center management, etc \cite{LinWierman12,NarayanaswamyGarg12,LuAndrew13,KimGiannakis17}, demands quantification of \acord{}'s performance for $K_t < \infty$. In this case, the optimization problem in Theorem \ref{main_thm:no_bias_CR} is not being solved exactly, hence, it is not known if finite $K_t$ allows for a worst-case guarantee. In fact, it requires solving an unanswered question in the broader SOCO literature:
\begin{quote}
    \textit{``What is the cumulative effect of per-round numerical approximation errors on the worst-case performance for SOCO algorithms?"}
\end{quote}

% \red{To understand the effect of decentralization and the graph structure $\mathcal{G}_t$, we need to quantify \acord{}'s performance for $K_t < \infty$.  }

All algorithms in the SOCO literature \cite{ChenWierman15,ChenWierman16,ChenWierman18,GoelWierman19,GoelLinWierman19,LiGuannan20,ShiGuannanWierman22,lin2022decentralized,ChristiansonWierman22,RuttenMukherjee23,LiShaoleiWierman23} model the online action $x_t$, at round $t$, as a solution of an optimization sub-problem, and assume that exact computation is possible. Consequently, numerical approximation errors and its effects have not been studied so far. Our following result is the first one to address this issue,
\begin{theorem}\label{main_thm: approx_robd_cr}
Consider a sequence of $\mu$-strongly convex hitting costs $\{f^i_t(\cdot)\}_{t=1}^T$ and squared $\ell_2$-norm switching costs. The online action sequence $\{x_t\}_{t=1}^T$ \textbf{approximating}
\begin{align}
    \Tilde{x}_t = \argmin_{x \in \R^d} \underbrace{f_t(x) + \frac{\lambda_1}{2}\|x - x_{t-1}\|_2^2}_{F_t(x)}
\end{align}
such that 
\begin{align}
    F_t(x_t) - F_t(\Tilde{x}_t) &= \epsilon_1\\
    \|x_t - \Tilde{x}_t\|_2 &= \epsilon_2
\end{align}
satisfies the following
\begin{align}
    \text{Cost}_{\alg{}} \leq \Bigg( &\frac{CR_{\robd{}} + 2\epsilon_2 \frac{\sqrt{\mu + \lambda_1}}{\lambda_1}}{1 - 2\epsilon_2 \frac{\sqrt{\mu + \lambda_1}}{\lambda_1}} \Bigg) \text{Cost}_{\opt{}}+ \left(\frac{\frac{\epsilon_2 \sqrt{\mu + \lambda_1}}{2\lambda_1} + \frac{\epsilon_1}{\lambda_1}}{1 - 2\epsilon_2 \frac{\sqrt{\mu + \lambda_1}}{\lambda_1}} \right)T
\end{align}
where $\epsilon_1, \epsilon_2$ can be suitably chosen to obtain the desired asymptotic competitive ratio as per \eqref{eqn: CR_defn}.
\end{theorem}
The standard potential method analysis typically used for proving competitive guarantees in SOCO \cite{ChenWierman18,GoelLinWierman19,RuttenMukherjee23,ChristiansonWierman22,ChristiansonWierman23,BhuyanMukherjee24} fails due to approximation errors at each round. To circumvent this issue, we translate this cumulative bias to the product of the approximation error $\epsilon_2$ and the total \emph{path length gap} between the online algorithm \alg{} and \opt{}, that is, $\epsilon_2 \sum_{t=1}^T \|x_t - x^*_t\|_2$. Converting it into the hitting costs of the \alg{} and \opt{}, along with some algebra, gives the result in Theorem \ref{main_thm: approx_robd_cr}.

For \acord{}, the heterogeneous parameters $\{\lambda_1^i\}_{i=1}^N$ necessitate a more delicate handling of the aforementioned error sequence, especially, converting it back into the sum of individual hitting and switching costs for \textit{heterogeneous} agents. We present the modified potential method and the analysis for \acord{} in detail in Appendix \ref{appendix_sec: approx_robd}. 

% Despite these, there are certain similarities shared by \acord{} and \dospa{}. Primarily, the averaging of actions and gradients across edges is shown to enough for both decentralized online set-ups and further goes to show that function-sharing done in \cite{lin2022decentralized} is overkill.
%%%%%%%%%%%%%%%%%%%%%%%%%%%%%%%%%%%%%%%%%%%%%%%%%%%%%%%%%%%%%%%%%%%%%%%%%%%%%%%
%%%%%%%%%%%%%%%%%%%%%%%%%%%%%%%%%%%%%%%%%%%%%%%%%%%%%%%%%%%%%%%%%%%%%%%%%%%%%%%  

\section{Multi-agent SOCO Preliminaries}\label{appendix_sec: decoup_robd}
% \begin{algorithm} [ht]
% \caption{Decoupled ROBD}\label{alg:ROBD_decoupled}
% \flushleft \textbf{Input:} strong convexity parameters $\{\mu_i\}_{i=1}^N$. 
% \flushleft \textbf{Initialize:} $\lambda_1^i = \frac{2}{1+\sqrt{1+\frac{4}{\mu_i}}}$ for $i \in \{1,\ldots, N\}$
% \begin{algorithmic}
% \Procedure{Decoup-ROBD}{}
% \For{$t = 1,2,\ldots,T$}
% \State Receive $f_t^1(x), \ldots, f_t^N(x)$, $\beta$ and $G_t$ from environment
% \State $x^R_t \gets \argmin_{x = (x^1,\ldots, x^N)} \sum_{i=1}^N \left(f_t^i(x^i) + \frac{\lambda_1^i}{2}\|x^i - x_{t-1}^{R,i}\|_2^2 \right) + \frac{\beta}{2}\|x\|_{G_t}^2$
% \EndFor
% \EndProcedure
% \end{algorithmic}
% \end{algorithm}
% \red{Write that \robd{} cannot be directly decentralized as it uses global parameters. Emphasize on the fact that $\lambda_1^i$ are local to the agent's $\mu_i$ and that this is the first step towards decentralization}

% \red{Write that the dissimilarity cost and the graph do not show up in the worst-case competitive ratio because the dissimilarity cost is degenerate along the $x_1 = \ldots = x_N$, which the adversary can take advantage of }
\begin{lemma}\label{lemma: decoupled-robd-strong-convexity_appendix}
Define $$F_t(\mathbf{x}) := \sum_{i=1}^N f_t^i(x^i) + \frac{\beta}{2}\|\mathbf{x}\|_{G_t}^2 + \sum_{i=1}^N \frac{\lambda_1^i}{2}\|x^i - x_{t-1}^{R,i}\|_2^2 .$$ It satisfies
    \begin{align*}
        F_t\left(\mathbf{x}_t^* \right) \geq F_t \left(x_t^R \right) + \sum_{i=1}^N \left(\frac{\mu_i + \lambda_1^i}{2} \right)\left\| x^{*,i}_t - x^{R,i}_t \right\|_2^2 %\red{+\frac{\beta}{2}\|x^R_t - x^*_t\|_{G_t}^2}
    \end{align*}
\end{lemma}
with
\begin{align}
    G_t = \begin{bmatrix}
        \sum_{j \in \mathcal{N}_1}\Tilde{A}^{(i,j)} & -\Tilde{A}^{(1,2)} \mathbbm{1}_{(1,2) \in \mathcal{E}_t} & \ldots & -\Tilde{A}^{(1,N)} \mathbbm{1}_{(1,N) \in \mathcal{E}_t}\\
        \vdots & \ldots & \ddots & \vdots\\
        -\Tilde{A}^{(N,1)} \mathbbm{1}_{(N,1) \in \mathcal{E}_t} & \ldots & -\Tilde{A}^{(N,N-1)} \mathbbm{1}_{(N,N-1) \in \mathcal{E}_t} & \sum_{j \in \mathcal{N}_N} \Tilde{A}^{(N,j)}
    \end{bmatrix}
\end{align}
where $\Tilde{A}^{(i,j)} = \left(A^{(i,j)}_t \right)^T A^{(i,j)}_t = \left(A^{(j,i)}_t \right)^T A^{(j,i)}_t = \Tilde{A}^{(j,i)}$ and $\mathcal{N}_i = \{j\in [N] \backslash \{i\}: (i,j) \in \mathcal{E}_t \}$ for each $i \in [N]$.
\begin{proof}
Through this simple fact,
\begin{align}
    \frac{1}{\beta}\nabla^2_{x^i_t, x^j_t} \sum_{(i,j) \in \mathcal{E}_t} s_t^{(i,j)}\left(x^i_t, x^j_t\right) &= \frac{1}{\beta}\nabla^2_{x^i_t, x^j_t} s_t^{(i,j)}\left(x^i_t, x^j_t\right)\\
    &= -\left(A^{(i,j)}_t \right)^T A^{(i,j)}_t,\\
    \frac{1}{\beta}\nabla^2_{x^i_t} \sum_{(i,j) \in \mathcal{E}_t} s_t^{(i,j)}\left(x^i_t, x^j_t\right) &= \frac{1}{\beta}\nabla^2_{x^i_t} \sum_{j \in \mathcal{N}_i} s_t^{(i,j)}\left(x^i_t, x^j_t\right)\\
    &= \sum_{j \in \mathcal{N}_1} \left(A^{(i,j)}_t \right)^T A^{(i,j)}_t
\end{align}
one can see that the dissimilarity cost defined in \eqref{eqn: dissimilarity_cost_defn} can be re-written as
\begin{align}
    \sum_{(i,j) \in \mathcal{E}_t} s_t^{(i,j)}\left(x^i_t, x^j_t\right) 
    = \sum_{(i,j) \in \mathcal{E}_t} \frac{\beta}{2}\left\|A_t^{(i,j)} x^i_t  - A^{(i,j)}_t x^j_t \right\|_2^2 = \frac{\beta}{2}\|\mathbf{x}\|_{G_t}^2.
\end{align}
We consider the following analysis for any set of hyper-parameters $\{\lambda_1^i\}_{i=1}^N \in (0,1]^N$. Since $x^R_t$ is the minimum of $F_t(x)$,
\begin{align}
    0 = \nabla F_t\left(\mathbf{x}_t^R \right) = \sum_{i=1}^N \left(\nabla f_t^i\left(x^{R,i}_t \right) + \lambda_1^i\left(x^{R,i}_t - x_{t-1}^{R,i} \right) \right) + \beta G_t x^R_t
\end{align}
Now, let $F_t^i(x^i):= f_t^i(x^i) + \frac{\lambda_1^i}{2}\|x^i - x_{t-1}^{R,i}\|_2^2$. This function is $(\mu_i + \lambda_1^i)$-strongly convex. Now,
\begin{align}
    F_t(x^*_t) &= F_t(x^*_t) - \left \langle \nabla F_t\left(x_t^R \right), x^*_t - x^R_t \right \rangle\\
    &= \sum_{i=1}^N \left(F_t^i\left(x^{*,i}_t \right) - \left \langle \nabla F_t^i\left(x_t^{R,i} \right), x^{*,i}_t - x^{R,i}_t \right \rangle \right) + \frac{\beta}{2}\|x^*_t\|_{G_t}^2 - \left \langle \beta G_t x^R_t, x^*_t - x^R_t \right \rangle\\
    &= \sum_{i=1}^N \left(F_t^i\left(x^{*,i}_t \right) - \left \langle \nabla F_t^i\left(x_t^{R,i} \right), x^{*,i}_t - x^{R,i}_t \right \rangle \right) + \frac{\beta}{2}\|x^R_t\|_{G_t}^2 \label{eqn1}\\ %\red{+\frac{\beta}{2}\|x^R_t - x^*_t\|_{G_t}^2} 
    &\geq \sum_{i=1}^N F_t^i\left(x^{R,i}_t \right) + \left(\frac{\mu_i + \lambda_1^i}{2} \right) \left\| x^{*,i}_t - x^{R,i}_t \right\|_2^2 + \frac{\beta}{2}\left\|x^R_t \right\|_{G_t}^2 \label{eqn2}\\ %\red{+\frac{\beta}{2}\|x^R_t - x^*_t\|_{G_t}^2} 
    &= F_t\left(\mathbf{x}^R_t\right) +  \sum_{i=1}^N \left(\frac{\mu_i + \lambda_1^i}{2} \right)\left\| x^{*,i}_t - x^{R,i}_t \right\|_2^2 %\red{+\frac{\beta}{2}\|x^R_t - x^*_t\|_{G_t}^2}
\end{align}
\end{proof}

\begin{theorem}\label{main_thm: decoup_robd}
    For hitting cost sequence $\{f^i_t\}_{i,t}$, squared $\ell_2$-norm switching costs and dissimilarity costs \eqref{eqn: dissimilarity_cost_defn}, $\beta \geq 0$ and graph $\{\mathcal{G}_t\}_{t=1}^T = \left\{\left([N],\mathcal{E}_t \right)\right\}_{t=1}^T$, the following action sequence,
    \begin{align}
        \mathbf{x}^R_t = \argmin_{\mathbf{x} = (x^1,\ldots, x^N)} &\sum_{i=1}^N \bigg[f_t^i(x^i) + \frac{\lambda_1^i}{2}\|x^i - x_{t-1}^{R,i}\|_2^2 + \frac{\beta}{2}\sum_{(i,j) \in \mathcal{E}_t}\left\|A^{(i,j)}_t x^i_t - A^{(i,j)}_t x^j_t \right\|_2^2\bigg]
    \end{align}
    has the optimal competitive ratio
    \begin{align}
        \text{CR} = \frac{1}{2} + \frac{1}{2}\sqrt{1+\frac{4}{\min_i \mu_i}}
    \end{align}
    with \textbf{local} hyperparameters $\{\lambda_1^i\}_{i=1}^N$,
    \begin{align}
        \lambda_1^i = \frac{2}{1+\sqrt{1+\frac{4}{\mu_i}}}.
    \end{align}
\end{theorem}
\begin{proof}
\begin{align}
    x^R_t = \argmin_{x = (x^1,\ldots, x^N)} \sum_{i=1}^N \left(f_t^i(x^i) + \frac{\lambda_1^i}{2}\|x^i - x_{t-1}^i\|_2^2 \right) + \frac{\beta}{2}\|x\|_{G_t}^2
\end{align}
and
\begin{align}
    F_t\left(\mathbf{x}^R_t \right) = \underbrace{\sum_{i=1}^N f_t^i(x^{R,i}_t) + \frac{\beta}{2}\|\mathbf{x}^R_t\|_{G_t}^2}_{H^R_t} + \sum_{i=1}^N \lambda_1^i \underbrace{\frac{1}{2}\|x^{R,i}_t - x_{t-1}^{R,i}\|_2^2}_{M^{R,i}_t}.
\end{align}
For the global hindsight action at $t$, $x_t^* = \left(x_t^{*,1},\ldots,x_t^{*,N} \right)$, define $H^*_t$ as the total hitting plus dissimilarity cost and $M^{i,*}_t$ as the switching cost of the $i^{th}$ agent. We have,
\begin{align}
    F_t\left(\mathbf{x}_t^* \right) \geq F_t \left(\mathbf{x}_t^R \right) + \sum_{i=1}^N \left(\frac{\mu_i + \lambda_1^i}{2} \right)\left\| x^{*,i}_t - x^{R,i}_t \right\|_2^2 %\red{+\frac{\beta}{2}\|x^R_t - x^*_t\|_{G_t}^2}
\end{align}
which means
\begin{align}
    H_t^R + \sum_{i=1}^N \lambda_1^i M^{R,i}_t + \sum_{i=1}^N \left(\frac{\mu_i + \lambda_1^i}{2} \right)\left\| x^{*,i}_t - x^{R,i}_t \right\|_2^2 %\red{+\frac{\beta}{2}\|x^R_t - x^*_t\|_{G_t}^2} 
    \leq H^*_t + \sum_{i=1}^N \frac{\lambda_1^i}{2} \left\| x^{*,i}_t - x^{R,i}_{t-1} \right\|_2^2  
\end{align}
Now, consider the potential function 
\begin{align}\label{eqn: potential_fn}
    \phi\left((x^R_t,x^*_t\right) = \sum_{i=1}^N \left(\frac{\mu_i + \lambda_1^i}{2} \right)\left\| x^{*,i}_t - x^{R,i}_t \right\|_2^2 %\red{+\frac{\beta}{2}\|x^R_t - x^*_t\|_{G_t}^2}
\end{align}
and subtract $\phi_{t-1}\left(x^R_{t-1},x^*_{t-1}\right)$ to have
\begin{align}\label{eqn:potential_eqn}
\begin{split}
    H_t^R &+ \sum_{i=1}^N \lambda_1^i M^{R,i}_t + \Delta \phi_t\\
    &\leq H^*_t + \underbrace{\sum_{i=1}^N \frac{\lambda_1^i}{2} \left\| x^{*,i}_t - x^{R,i}_{t-1} \right\|_2^2 - \sum_{i=1}^N \left(\frac{\mu_i + \lambda_1^i}{2} \right)\left\| x^{*,i}_{t-1} - x^{R,i}_{t-1} \right\|_2^2}_{X}
\end{split}
\end{align}
We take a closer look at $X$,
\begin{align}
    X &= \sum_{i=1}^N \frac{\lambda_1^i}{2} \left\| x^{*,i}_t - x^{R,i}_{t-1} \right\|_2^2 -  \left(\frac{\mu_i + \lambda_1^i}{2} \right)\left\| x^{*,i}_{t-1} - x^{R,i}_{t-1} \right\|_2^2\\ %\red{-\frac{\beta}{2}\|x^R_{t-1} - x^*_{t-1}\|_{G_{t-1}}^2}
    &\leq \sum_{i=1}^N \frac{\lambda_1^i}{2} \left\| x^{*,i}_t - x^{*,i}_{t-1} \right\|_2^2 + \lambda_1^i\left\| x^{*,i}_t - x^{*,i}_{t-1} \right\|\cdot \left\|x^{*,i}_{t-1} - x^{R,i}_{t-1} \right\|  -  \frac{\mu_i}{2} \left\| x^{*,i}_{t-1} - x^{R,i}_{t-1} \right\|_2^2 %\red{-\frac{\beta}{2}\|x^R_{t-1} - x^*_{t-1}\|_{G_{t-1}}^2}
    \end{align}
    % Now, one instead of writing $\lambda_1^i\left\| x^{*,i}_t - x^{*,i}_{t-1} \right\|\cdot \left\|x^{*,i}_{t-1} - x^{R,i}_{t-1} \right\|$, one can stop at the inner product between the two and split the inner product component wise in the eigenspace of $G_t$. Now, the AM-GM trick can be applied in each component to exactly cancel of each component in $\red{-\frac{\beta_{t-1}}{2}\|x^R_{t-1} - x^*_{t-1}\|_{G_{t-1}}^2}$. However, each component has a different weightage (the corresponding eigenvalue of $G_{t-1}$. This eigen value appears in the denominator along with $\mu_i$, ideally driving down the CR. 

    % However, some eigenvalues of $G_{t-1}$ are non-zero and therefore there will be components yielding only $\mu_i$ in the denominator. Now, we have split $x^{*,i}_t - x^{*,i}_{t-1}$ into the components, each having a different constant $\frac{\lambda_1^i}{\mu_i + (\geq 0)}$ multiplied to it. To combine all of them back into $M^{*,i}_t$, one needs to upper bound with a common constant, which will be the largest of all such $\frac{\lambda_1^i}{\mu_i + ()}$. Now remember that there components with zero eigenvalue in $G_{t-1}$, meaning this largest value will in fact be $\frac{\lambda_1^i}{\mu_i + 0}$. This nullifies the effect of graph $G_{t-1}$ on the CR.
    \begin{align}
    &= \sum_{i=1}^N \frac{\lambda_1^i}{2} \left\| x^{*,i}_t - x^{*,i}_{t-1} \right\|_2^2 + \frac{\lambda_1^i\left\| x^{*,i}_t - x^{*,i}_{t-1} \right\|}{\sqrt{\mu_i}}\cdot \sqrt{\mu_i}\left\|x^{*,i}_{t-1} - x^{R,i}_{t-1} \right\|  -  \frac{\mu_i}{2} \left\| x^{*,i}_{t-1} - x^{R,i}_{t-1} \right\|_2^2 \\
    &\leq \sum_{i=1}^N \frac{\lambda_1^i}{2} \left\| x^{*,i}_t - x^{*,i}_{t-1} \right\|_2^2 + \frac{\left(\lambda_1^i\right)^2\left\| x^{*,i}_t - x^{*,i}_{t-1} \right\|_2^2}{2\mu_i}\\
    &= \sum_{i=1}^N \lambda_1^i \left( 1+\frac{\lambda_1^i}{\mu_i} \right) M^{*,i}_t
\end{align}
Therefore, we have that
\begin{align}
    H_t^R + \sum_{i=1}^N \lambda_1^i M^{R,i}_t + \Delta \phi_t &\leq H^*_t + \sum_{i=1}^N \lambda_1^i \left( 1+\frac{\lambda_1^i}{\mu_i} \right) M^{*,i}_t
\end{align}
leading to
\begin{align}
    \left(\min_i \lambda_1^i \right) (H_t^R + m^R) + \Delta \phi_t &\leq H^R_t + \left( \max_i \left\{\lambda_1^i \left( 1+\frac{\lambda_1^i}{\mu_i} \right) \right\} \right) M^{*,i}_t
\end{align}
and finally,
\begin{align}
    H^R_t + M^R_t + \frac{\Delta \phi_t}{\min_i \lambda_1^i} \leq \max \left\{ \frac{1}{\min_i \lambda_1^i}, \max_i \left\{\lambda_1^i \left( 1+\frac{\lambda_1^i}{\mu_i} \right) \right\} \right\} \left( H^*_t + M^*_t \right).
\end{align}
Therefore, the general competitive ratio is
\begin{align}
    CR = \max \left\{ \frac{1}{\min_i \lambda_1^i}, \max_i \left\{\lambda_1^i \left( 1+\frac{\lambda_1^i}{\mu_i} \right) \right\} \right\}
\end{align}
Now, we find the optimal values of $\left\{\lambda_1^i \right\}_{i=1}^N$. First observe that
\begin{align}
    \max \left\{ \frac{1}{\min_i \lambda_1^i}, \max_i \left\{\lambda_1^i \left( 1+\frac{\lambda_1^i}{\mu_i} \right) \right\} \right\} \geq \max \left\{ \frac{1}{\lambda_1^i}, \lambda_1^i \left( 1+\frac{\lambda_1^i}{\mu_i} \right) \right\} \text{ } \forall \text{ } i \in \{1,\ldots, N\}
\end{align}
and
\begin{align}
    \max \left\{ \frac{1}{\lambda_1^i}, \lambda_1^i \left( 1+\frac{\lambda_1^i}{\mu_i} \right) \right\} \geq 1+\frac{1}{2}\left(\sqrt{1+\frac{4}{\mu_i}}-1\right).
\end{align}
This means 
\begin{align}
    CR \geq 1+\frac{1}{2}\left(\sqrt{1+\frac{4}{\min_i \mu_i}}-1\right).
\end{align}
Keeping this lower bound in mind, let's consider the values of $\lambda_1^i = \frac{2}{1 + \sqrt{1 + \frac{4}{\mu_i}}} = \frac{\mu_i}{2}\left(\sqrt{1 + \frac{4}{\mu_i}}-1\right)$. Observe that $$\lambda_1^i \left( 1+\frac{\lambda_1^i}{\mu_i} \right) = \frac{\mu_i}{2}\left(\sqrt{1 + \frac{4}{\mu_i}}-1\right) \frac{1}{2}\left(\sqrt{1+\frac{4}{\mu_i}}+1 \right) = 1 \text{ } \forall \text{ } i\in \{1,\ldots,N\}.$$
This means the competitive ratio is
\begin{align}
    CR = \frac{1}{\min_i \lambda_1^i} = \frac{1}{\frac{2}{1 + \sqrt{1 + \frac{4}{\min_i \mu_i}}}} = \frac{1}{2} + \frac{1}{2} \sqrt{1 + \frac{4}{\min_i \mu_i}}
\end{align}
which proves the best competitive ratio to be $CR = \frac{1}{2} + \frac{1}{2} \sqrt{1 + \frac{4}{\min_i \mu_i}}$.
\end{proof}

\section{Performance of \acordcaps{}}\label{appendix_sec: decental_robd_l2}
\subsection{Proof of Theorem \ref{main_thm:no_bias_CR}}\label{appendix_ssec: 3.1 proof}
% \begin{theorem}\label{thm: decentralized_robd_cr_sq}
%     The modified objective
%     $$
%     \mathbf{x}_t,\mathbf{z} = \argmin_{\substack{\mathbf{x} = \left(x^1,\ldots,x^N \right)\\ \mathbf{z} = (z_e)_{e\in \mathcal{E}_t}}} \sum_{i=1}^N \left\{ f_t^i(x^i) + \frac{\lambda_1^i}{2}\|x^i - x^i_{t-1}\|_2^2 + \beta \sum_{\mathcal{E}_t \ni e \ni i} \|x^i - z_e\|_{(A^e_t)^T A^e_t}^2 \right\}
%     $$
%     in each round of \\acord{} converges to the same action as that of Theorem \ref{appendix_thm: decoup_robd_cr_sq} and hence, has a competitive ratio of $\frac{1}{2} + \frac{1}{2}\sqrt{1 + \frac{4}{\min_i \mu_i}}$
% \end{theorem}
% \begin{proof}
The online action in Theorem \ref{main_thm: decoup_robd} is
\begin{align}
    \mathbf{x}^R_t &= \argmin_{\mathbf{x} = (x^1,\ldots,x^N)} \sum_{i=1}^N F_t^i(x^i) + \frac{\beta}{2}\|x\|_{G_t}^2\\
    &= \argmin_{\mathbf{x} = (x^1,\ldots,x^N)} \sum_{i=1}^N \left\{f_t^i(x^i) + \frac{\lambda_1^i}{2}\|x^i - x^i_{t-1}\|_2^2 \right\} + \frac{\beta}{2}\sum_{(i,j) \in \mathcal{E}_t} \|x^i - x^j\|_{(A^{(i,j)}_t)^T A^{(i,j)}_t}^2 \label{eqn:d-robd_onj}
\end{align}
but the presence of $G_t$ makes the sub-problem still coupled. We decouple this completely replacing the dissimilarity cost with one allows for decentralization, albeit creating extra variables,
\begin{align}
    x_t^i = \argmin_{x \in \R^d} f_t^i(x) + \frac{\lambda_1^i}{2}\|x - x^i_{t-1}\|_2^2 + \beta \sum_{\mathcal{E}_t \ni e \ni i} \|x - z_e\|_{(A^e_t)^T A^e_t}^2 \text{ } \forall \text{ } i \in \{1,\ldots,N\}
\end{align}
where $z_e$ is shared between the sub-problems for both $x_i^t$ and $x_j^t$ is $e = (i,j)$. If we combine all these into one,
\begin{align}\label{eqn:decentralized-robd-obj_appendix}
    \mathbf{x}_t,\mathbf{z} = \argmin_{\substack{\mathbf{x} = \left(x^1,\ldots,x^N \right)\\ \mathbf{z} = (z_e)_{e\in \mathcal{E}_t}}} \sum_{i=1}^N \left\{ f_t^i(x^i) + \frac{\lambda_1^i}{2}\|x^i - x^i_{t-1}\|_2^2 + \beta \sum_{\mathcal{E}_t \ni e \ni i} \|x^i - z_e\|_{(A^e_t)^T A^e_t}^2 \right\}
\end{align}
Now observe what the solution to this optimization sub-problem should satisfy
\begin{gather}
    \nabla_{x^i} f_t^i(x^i) + \lambda_1^i (x^i - x^i_{t-1}) + 2\beta \sum_{\mathcal{E}_t \ni e \ni i} (A^e_t)^T A^e_t (x^i - z_e) \text{ } \forall \text{ } i \in \{1,\ldots,N\}\\
    z_e = \frac{x^i + x^j}{2} \text{ where } e = (i,j)
\end{gather}
and putting these two together,
\begin{align}
    \nabla_{x^i} f_t^i(x^i) + \lambda_1^i (x^i - x^i_{t-1}) + 2\beta &\sum_{j \text{ s.t } (i,j) \in \mathcal{E}_t} (A^e_t)^T A^e_t\left(x^i - \frac{x^i + x^j}{2} \right) = 0 \text{ } \forall \text{ } i \in \{1,\ldots,N\}\label{eqn3}\\
    \implies \nabla_{x^i} f_t^i(x^i) + \lambda_1^i (x^i - x^i_{t-1}) + \beta &\sum_{j \text{ s.t } (i,j) \in \mathcal{E}_t} (A^e_t)^T A^e_t\left(x^i - x^j \right) = 0 \text{ } \forall \text{ } i \in \{1,\ldots,N\}.\label{eqn4}
\end{align}
Now, observe the optimality condition for the objective in \eqref{eqn:d-robd_onj}, 
\begin{align}\label{eqn5}
    \nabla_{x^i} f_t^i(x^i) + \lambda_1^i(x^i - x^i_{t-1}) + \beta\sum_{j \text{ s.t } (i,j) \in \mathcal{E}_t} (A^{(i,j)}_t)^T A^{(i,j)}_t (x^i - x^j) = 0
\end{align}
This proves that our method of decoupling the dissimilarity cost retains the optimality of condition for the \textit{still coupled} objective of Theorem \ref{main_thm: decoup_robd}.

\[
\pushQED{\qed} 
\qedhere
\popQED
\]  

\subsection{Proof of Proposition \ref{prop: alt_min_covergence}}

\begin{lemma}\label{lemma: strong-convexity}
    $\mathbbm{F}_t(\mathbf{x},\mathbf{z})$ is strongly convex in both $\mathbf{x}$ and $\mathbf{z}$. 
\end{lemma}
\begin{proof}
In this result, we only prove a strictly positive strong-convexity parameter. For the exact dependence of this parameter on the graph $\mathcal{G}_t$, please refer to Appendix \ref{appendix_sec: graph_dependence}. Now, continuing with the aforementioned lemma's proof. Consider the Hessian of $\sum_{i=1}^N \sum_{\mathcal{E}_t \ni e \ni i}\|A_t^{e} x^i - A_t^{e}z_e\|_{2}^2$ which can be written as
\begin{align}
    2\begin{bmatrix}
        \Tilde{\mathbf{D}}^G_{Nd \times Nd} & \Tilde{\mathbf{G}}_z\\
        \Tilde{\mathbf{G}}_z^T & 2 \Tilde{\mathbf{D}}^Z_{|\mathcal{E}_t|d}
    \end{bmatrix}
\end{align}
The matrix $\Tilde{\mathbf{G}}_z$ is $-\left(A^{q}_t \right)^T A^{q}_t$ at $(i,q)^{th}$ block if the $q^{th}$ edge has one end as agent $i$. All other $d \times d$ blocks of $\Tilde{\mathbf{G}}_z$ are zero. The matrix $\Tilde{\mathbf{D}}^G_{Nd \times Nd}$ is a block diagonal matrix with the $i^{th}$ block being $\sum_{\mathcal{E}_t \ni e \ni i} (A^e_t)^T A^e_t$. Lastly, $\Tilde{\mathbf{D}}^Z_{|\mathcal{E}_t|d}$ is also a diagonal matrix where the $q^{th}$ diagonal block is $\left(A^{q}_t \right)^T A^{q}_t$. 

Now, for any $\begin{bmatrix}
    \mathbf{x}\\ \mathbf{z}
\end{bmatrix}$, the following holds as $(A^e_t)^T A^e_t \succ mI$ $\forall$ $e \in \mathcal{E}_t$,
\begin{align}
    \begin{bmatrix}
    \mathbf{x} & \mathbf{z}
\end{bmatrix}\left(2\begin{bmatrix}
        \Tilde{\mathbf{D}}^G_{Nd \times Nd} & \Tilde{\mathbf{G}}_z\\
        \Tilde{\mathbf{G}}_z^T & 2 \Tilde{\mathbf{D}}^Z_{|\mathcal{E}_t|d}
    \end{bmatrix}\right) \begin{bmatrix}
    \mathbf{x}\\ \mathbf{z}
\end{bmatrix} &= 2\sum_{i=1}^N \sum_{\mathcal{E}_t \ni e \ni i}\|A_t^{e} x^i - A_t^{e}z_e\|_{2}^2\\
&\geq 2m \sum_{i=1}^N \sum_{\mathcal{E}_t \ni e \ni i}\|x^i - z_e\|_{2}^2\\
&= \begin{bmatrix}
    \mathbf{x} & \mathbf{z}
\end{bmatrix}\left(2m\begin{bmatrix}
        \mathbf{D}^G_{Nd \times Nd} & \mathbf{G}_z\\
        \mathbf{G}_z^T & 2\mathbf{I}_{|\mathcal{E}_t|d}
    \end{bmatrix}\right) \begin{bmatrix}
    \mathbf{x}\\ \mathbf{z}
\end{bmatrix}
\end{align}
The matrix $\mathbf{G}_z$ is $-\mathbf{I}_d$ at $(i,q)^{th}$ block if the $q^{th}$ edge has one end as agent $i$. All other $d \times d$ blocks of $\mathbf{G}_z$ are zero. The matrix $\mathbf{D}^G_{Nd \times Nd}$ is again a block diagonal matrix with the $i^{th}$ block being $d_i \mathbf{I}_d$ where $d_i$ is the $i^{th}$ agent's degree. This means
\begin{align}
    \begin{bmatrix}
        \Tilde{\mathbf{D}}^G_{Nd \times Nd} & \Tilde{\mathbf{G}}_z\\
        \Tilde{\mathbf{G}}_z^T & 2 \Tilde{\mathbf{D}}^Z_{|\mathcal{E}_t|d}
    \end{bmatrix} \succeq  2m\begin{bmatrix}
        \mathbf{D}^G_{Nd \times Nd} & \mathbf{G}_z\\
        \mathbf{G}_z^T & 2\mathbf{I}_{|\mathcal{E}_t|d}
    \end{bmatrix}
\end{align}
which is tight as the above holds for any $\{A^e_t\}_{e \in \mathcal{E}_t}$.
Now, consider the Hessian of $\mathbbm{F}_t(\mathbf{x},\mathbf{z})$. It will be tightly lower bounded as
\begin{align}\label{eqn: hessian}
    \nabla^2 \mathbbm{F}_t(\mathbf{x},\mathbf{z}) \succeq 
    \underbrace{\begin{bmatrix}
        \mathbf{D}^{f}_{Nd\times Nd} & \mathbf{0}_{Nd \times |\mathcal{E}_t|d}\\
        \mathbf{0}_{|\mathcal{E}_t|d \times Nd} & \mathbf{0}_{|\mathcal{E}_t|d \times |\mathcal{E}_t|d}
    \end{bmatrix}}_A + 2\beta m
    \underbrace{\begin{bmatrix}
        \mathbf{D}^G_{Nd \times Nd} & \mathbf{G}_z\\
        \mathbf{G}_z^T & 2\mathbf{I}_{|\mathcal{E}_t|d}
    \end{bmatrix}}_B
\end{align}
where $\mathbf{D}^f_{Nd \times Nd}$ is the diagonal block matrix where the $i^{th}$ $d \times d$ block is $(\mu_i + \lambda_1^i)\mathbf{I}_d$. Note that the matrix $B$ is positive semi-definite as $ \sum_{i=1}^N \sum_{\mathcal{E}_t \ni e \ni (i,j)} \|x^i - z_e\|_2^2 \geq 0$ for any $(x,z) \in \R^{(N+|\mathcal{E}_t|)d}$.

Now, we need to find the smallest eigenvalue of $(A+2\beta mB)$ to get the strong-convexity parameter. Suppose the associated eigenvector is $v \in \R^{(N+|\mathcal{E}_t|)d}$. We need to express this in some orthonormal basis for $\R^{(N+|\mathcal{E}_t|)d}$. Let's take the eigenbasis of the first matrix in \eqref{eqn: hessian}. One can easily see that the basis $\{y_i\}_{i=1}^{(N+|\mathcal{E}_t|)d}$ is of the form
\begin{align}
    \underbrace{\begin{bmatrix}
        u_1\\
        \mathbf{0}_{|\mathcal{E}_t|d}
    \end{bmatrix}, \ldots, 
    \begin{bmatrix}
        u_{Nd}\\
        \mathbf{0}_{|\mathcal{E}_t|d}
    \end{bmatrix}}_{S_1 = \{y_i\}_{i=1}^{Nd}}, 
    \underbrace{\begin{bmatrix}
        \mathbf{0}_{Nd}\\
        \mathbf{e}_{1}
    \end{bmatrix}, \ldots,
    \begin{bmatrix}
        \mathbf{0}_{Nd}\\
        \mathbf{e}_{|\mathcal{E}_t|d}
    \end{bmatrix}}_{S_2 = \{y_i\}_{i=Nd+1}^{(N+|\mathcal{E}_t|)d}}
\end{align}
where $\{u_i\}_{i=1}^{Nd}$ is the eigenbasis of $\mathbf{D}^{f}_{Nd\times Nd}$. Suppose $v = \sum_{i=1}^{(N+|\mathcal{E}_t|)d} c_i y_i$. Then the strong convexity parameter of $\mathbbm{F}_t(x,z)$ satisfies
\begin{align}
        \sigma &= v^T (A+2\beta mB) v\\
        &= v^TA v  + 2\beta m v^T B v\\
        &= \sum_{i=1}^{Nd} c_i^2 \mu_{1+\left \lfloor \frac{i-1}{N} \right \rfloor}    + 2\beta m v^T B v
    \end{align}

Now, there are two cases:
\begin{enumerate}
    \item $v \in Span(S_2)$ : This implies the first part of the sum above is zero and $v = \begin{bmatrix}
        \mathbf{0}_{Nd} \\ \mathbf{c} 
    \end{bmatrix}$. Therefore,
    \begin{align}
       \sigma = 2\beta v^TB v &= 2\beta m\left( \sum_{i=Nd+1}^{(N+|\mathcal{E}_t|)d} c_i y_i \right)^T B v\\
       &= 2\beta m\begin{bmatrix}
           \mathbf{c}^T G_z^T & 2\mathbf{c}^T 
       \end{bmatrix}\begin{bmatrix}
        \mathbf{0}_{Nd} \\ \mathbf{c} 
    \end{bmatrix}\\
    &= 4\beta m\| \mathbf{c}\|_2^2\\
    &= 4\beta m> 0
    \end{align}
    \item $v \not \in Span(S_2)$: This means there is $i \in \{1,\ldots,Nd\}$ such that $c_i > 0$. Since $B$ is positive semi-definite, $v^T B b \geq 0$ and,
    \begin{align}
        \sigma \geq c_i^2 \mu_{1+\left \lfloor \frac{i-1}{N} \right \rfloor} > 0
    \end{align}
\end{enumerate}
proving that $\mathbbm{F}_t(x,z)$ is strongly convex is both $x$ and $z$. Also note that $\sigma$ is at most $4\beta m$. The argument for this is similar to what we just saw,
\begin{align}
    \sigma &= v^T (A+2\beta mB)v\\
    &\leq \left(\begin{bmatrix}
        \mathbf{0}_{Nd} \\ \mathbf{c} 
    \end{bmatrix} \right)^T (A + 2\beta mB) \begin{bmatrix}
        \mathbf{0}_{Nd} \\ \mathbf{c} 
    \end{bmatrix}\\
    &= 2\beta m\cdot 2 \|\mathbf{c}\|_2^2\\
    &= 4\beta m
\end{align}
where $\mathbf{c} \in \R^{|\mathcal{E}_t|d}$.
\end{proof}

\begin{lemma}\label{lemma: smoothness}
$\mathbbm{F}_t(\mathbf{x},\mathbf{z})$ is $4\beta l$ smooth in $\mathbf{z}$.    
\end{lemma}
\begin{proof}
\begin{align}
\begin{split}
    \mathbbm{F}_t&(\mathbf{x},\mathbf{z}+\Delta \mathbf{z}) - \nabla \mathbbm{F}_t(\mathbf{x},\mathbf{z})^T ((\mathbf{x},\mathbf{z}+\Delta \mathbf{z})-(\mathbf{x},\mathbf{z})) = \mathbbm{F}_t(\mathbf{x},\mathbf{z}+\Delta \mathbf{z}) - \nabla \mathbbm{F}_t(\mathbf{x},\mathbf{z})^T ((0,\Delta \mathbf{z}))\\
    =& \mathbbm{F}_t(\mathbf{x},\mathbf{z}+\Delta \mathbf{z}) - \nabla_\mathbf{z} \mathbbm{F}_t(\mathbf{x},\mathbf{z}) ^T\Delta \mathbf{z}
\end{split}\\
\begin{split}
    =& \sum_{i=1}^N \left\{ f_t^i(x^i) + \frac{\lambda_1^i}{2}\|x^i - x^i_{t-1}\|_2^2 \right\} + \beta \sum_{e (=(i,j)) \in \mathcal{E}_t} \|z_e + \Delta z_e - x^i \|_{(A^e_t)^T A^e_t}^2 + \|z_e + \Delta z_e - x^j \|_{(A^e_t)^T A^e_t}^2\\
    & - \beta \sum_{e (=(i,j)) \in \mathcal{E}_t} 2(z_e  - x^i)^T (A^e_t)^T A^e_t \Delta z_e + 2(z_e - x^j)^T (A^e_t)^T A^e_t \Delta z_e
\end{split}\\
\leq& \sum_{i=1}^N \left\{ f_t^i(x^i) + \frac{\lambda_1^i}{2}\|x^i - x^i_{t-1}\|_2^2 \right\} + \beta \sum_{e (=(i,j)) \in \mathcal{E}_t} \|z_e - x^i \|_{(A^e_t)^T A^e_t}^2 + \|z_e - x^j \|_{(A^e_t)^T A^e_t}^2 + 2\beta l \sum_{e \in \mathcal{E}_t} \Delta z_e ^2\\
\leq& \mathbbm{F}_t(\mathbf{x},\mathbf{z}) + \frac{4 \beta l}{2} \|\Delta \mathbf{z}\|_2^2
\end{align}
proving $4\beta l$-smoothness in $\mathbf{z}$.
\end{proof}

\begin{proposition}\label{thm: alt_min_conv_l2}
Alternating Minimization applied to $\mathbbm{F}_t(\mathbf{x},\mathbf{z})$ with $\mathbf{x}$ being one block and $\mathbf{z}$ being the other allows for exponentially fast convergence to $\mathbf{\Tilde{x}}_t, \mathbf{z}_* = \argmin_{\mathbf{x},\mathbf{z}} \mathbbm{F}_t(x,z)$,
$$
\mathbbm{F}_t\left(\left(\mathbf{x}_t\right)_{k},\mathbf{z}_{k} \right) - \mathbbm{F}_t\left(\mathbf{\Tilde{x}}_t, \mathbf{z}_*\right) \leq \left(1-\frac{\sigma}{4\beta l} \right)^{k-1} \left(\mathbbm{F}_t\left(\left(\mathbf{x}_t\right)_{0},\mathbf{z}_{0} \right) - \mathbbm{F}_t\left(\mathbf{\Tilde{x}}_t, \mathbf{z}_*\right)\right)
$$
\end{proposition}
\begin{proof}
The proof is along the lines of the more general version in \cite{BeckTetruashvili13}. We do it specifically for the case where $\mathbbm{F}_t(x,z)$ is smooth only in $z$ and not in $x$. The algorithm goes as follows (where we drop the $t$ subscript for better readability)
\begin{align}
    \mathbf{x}_{k+1} &= \argmin_{\mathbf{x} \in \R^{Nd}} \mathbbm{F}_t(\mathbf{x},\mathbf{z}_{k})\\
    \mathbf{z}_{k+1} &= \argmin_{\mathbf{z} \in \R^{|\mathcal{E}_t|d}} \mathbbm{F}_t(\mathbf{x}_{k+1}, \mathbf{z}).
\end{align}
We define $\mathbf{w}_{k} := (\mathbf{x}_{k},\mathbf{z}_{k})$ and $\mathbf{w}_{k+\frac{1}{2}} := (\mathbf{x}_{k+1},\mathbf{z}_{k})$. It is evident from the algorithm structure that the following holds,
\begin{align}
    \mathbbm{F}_t(\mathbf{w}_0) \geq \mathbbm{F}_t\left(\mathbf{w}_{\frac{1}{2}}\right) \geq \mathbbm{F}_t(\mathbf{w}_1) \geq \mathbbm{F}_t\left( \mathbf{w}_{\frac{3}{2}} \right) \geq \mathbbm{F}_t(\mathbf{w}_2) \geq \ldots\geq \mathbbm{F}_t(\mathbf{w}_*) 
\end{align}
with $\mathbf{w}_* = \left(\mathbf{\Tilde{x}}_t, \mathbf{z}_*\right)$.

From Lemma \eqref{lemma: strong-convexity}, we know that $\mathbbm{F}_t(x,z)$ is strongly convex in both $(x,z)$. Suppose the strong-convexity parameter is $\sigma$ (which we showed in strictly positive). Therefore, we have
\begin{align}
    \mathbbm{F}_t(\mathbf{w}) \geq \mathbbm{F}_t\left(\mathbf{w}_{k+\frac{1}{2}}\right) + \nabla \mathbbm{F}_t\left(\mathbf{w}_{k+\frac{1}{2}}\right)^T \left(\mathbf{w} - \mathbf{w}_{k+\frac{1}{2}}\right) + \frac{\sigma}{2}\left\| \mathbf{w} - \left(\mathbf{w}_{k+\frac{1}{2}}\right) \right\|_2^2
\end{align}
Now, minimum of LHS is larger than the minimum of RHS (a quadratic in $\mathbf{w}$), giving the following
\begin{align}\label{eqn6}
    \mathbbm{F}_t(\mathbf{w}_*) \geq \mathbbm{F}_t\left(\mathbf{w}_{k+\frac{1}{2}}\right) - \frac{1}{2\sigma} \left\|\nabla \mathbbm{F}_t\left(\mathbf{w}_{k+\frac{1}{2}}\right) \right\|_2^2.
\end{align}
Further, we have that
\begin{align}
    \mathbbm{F}_t\left(\mathbf{w}_{k+\frac{1}{2}}\right) - \mathbbm{F}_t(\mathbf{w}_{k+1}) &= \mathbbm{F}_t(\mathbf{x}_{k+1},\mathbf{z}_k) - \mathbbm{F}_t(\mathbf{x}_{k+1},\mathbf{z}_{k+1})\\
    &\geq \mathbbm{F}_t(\mathbf{x}_{k+1},\mathbf{z}_k) - \mathbbm{F}_t\left(\mathbf{x}_{k+1},\mathbf{z}_k - \frac{\nabla_{\mathbf{z}}\mathbbm{F}_t(\mathbf{x}_{k+1},\mathbf{z}_k)}{4\beta l} \right) \label{eqn11}\\
    &\geq \frac{\|\nabla_{\mathbf{z}}\mathbbm{F}_t(\mathbf{x}_{k+1},\mathbf{z}_k)\|_2^2}{2\cdot 4\beta l}
\end{align}
where the first inequality is a consequence of $\mathbf{z}_{k+1} = \argmin_\mathbf{z} \mathbbm{F}_t(\mathbf{x}_{k+1},\mathbf{z})$ and second inequality is from $2\beta$ smoothness of $\mathbbm{F}_t(\mathbf{x},\mathbf{z})$ in $\mathbf{z}$. Now, observe that $\nabla_{\mathbf{x}} \mathbbm{F}_t(\mathbf{x}_{k+1},\mathbf{z}_k) = \mathbf{0}$ as $\mathbf{x}_{k+1} = \argmin_\mathbf{x} \mathbbm{F}_t(\mathbf{x},\mathbf{z}_k)$. Therefore,
\begin{align}
    \mathbbm{F}_t\left(\mathbf{w}_{k+\frac{1}{2}}\right) - \mathbbm{F}_t(\mathbf{w}_{k+1}) \geq \frac{\left\|\nabla \mathbbm{F}_t\left(\mathbf{w}_{k+\frac{1}{2}}\right) \right\|_2^2}{2\cdot 4\beta l}
\end{align} and we already know that $\mathbbm{F}_t(\mathbf{w}_{k+1}) \geq \mathbbm{F}_t\left(\mathbf{w}_{k+\frac{3}{2}}\right)$, giving us
\begin{align}\label{eqn7}
    \mathbbm{F}_t\left(\mathbf{w}_{k+\frac{1}{2}}\right) - \mathbbm{F}_t\left(\mathbf{w}_{k+\frac{3}{2}}\right) \geq \frac{\left\|\nabla \mathbbm{F}_t\left(\mathbf{w}_{k+\frac{1}{2}}\right) \right\|_2^2}{2\cdot 4\beta l}.
\end{align}
Now we combine \eqref{eqn6} and \eqref{eqn7} to give
\begin{align}
    \mathbbm{F}_t\left(\mathbf{w}_{k+\frac{1}{2}}\right) - \mathbbm{F}_t(\mathbf{w}_*) &\leq \frac{1}{2\sigma} \left\|\nabla \mathbbm{F}_t\left(\mathbf{w}_{k+\frac{1}{2}}\right) \right\|_2^2\\
    &\leq \frac{4\beta l}{\sigma} \left( \mathbbm{F}_t\left(\mathbf{w}_{k+\frac{1}{2}}\right) - \mathbbm{F}_t\left(\mathbf{w}_{k+\frac{3}{2}}\right) \right)\\
    &= \frac{4\beta l}{\sigma} \left( \left(\mathbbm{F}_t\left(\mathbf{w}_{k+\frac{1}{2}}\right) - \mathbbm{F}_t(\mathbf{w}_*)\right) - \left(\mathbbm{F}_t\left(\mathbf{w}_{k+\frac{3}{2}}\right) - \mathbbm{F}_t(\mathbf{w}_*)\right) \right)
\end{align}
which gives
\begin{align}
    \left(\mathbbm{F}_t\left(\mathbf{w}_{k+\frac{3}{2}}\right) - \mathbbm{F}_t(\mathbf{w}_*)\right) \leq \left(1-\frac{\sigma}{4\beta l} \right) \left(\mathbbm{F}_t\left(\mathbf{w}_{k+\frac{1}{2}}\right) - \mathbbm{F}_t(\mathbf{w}_*)\right).
\end{align}
Therefore, for $k \geq 1$,
\begin{align}
    \mathbbm{F}_t(\mathbf{w}_k) - \mathbbm{F}_t(\mathbf{w}_*) &\leq  \mathbbm{F}_t\left(\mathbf{w}_{k-\frac{1}{2}}\right) - \mathbbm{F}_t(\mathbf{w}_*)\\
    &\leq \left(1-\frac{\sigma}{4\beta l} \right)^{k-1} \left(\mathbbm{F}_t\left(\mathbf{w}_{\frac{1}{2}}\right) - \mathbbm{F}_t(\mathbf{w}_*)\right)\\
    &\leq \left(1-\frac{\sigma}{4\beta l} \right)^{k-1} \left(\mathbbm{F}_t(\mathbf{w}_0) - \mathbbm{F}_t(\mathbf{w}_*)\right) \label{eqn: exponential_convergence}
\end{align}
\end{proof}

\section{Approximation Framework for SOCO}\label{appendix_sec: approx_robd}
\subsection{Proof of Theorem \ref{main_thm: approx_robd_cr}}

$H_t$ is \alg{}'s hitting cost, $M_t$ is \alg{}'s switching cost, $H^*_t$ is \opt{}'s hitting cost and $M^*_t$ is \opt{}'s switching cost at round $t$. Now, we know from the proof of \robd{}'s competitive ratio (Theorem 8 \cite{GoelLinWierman19}) that
\begin{align}
    \Tilde{H}_t + \lambda_1 \Tilde{M}_t + (\phi(\Tilde{x}_t,x^*_t) - \phi(x_{t-1},x^*_{t-1})) &\leq H^*_t + \lambda_1\left(1 + \frac{\lambda_1}{\mu } \right) M^*_t
\end{align}
where $\Tilde{H}_t + \Tilde{M}_t = \min_{x \in \R^d} f_t(x) + \frac{\lambda_1}{2}\|x - x_{t-1}\|_2^2$. The approximate solution is such that $(H_t + \lambda_1 m) - (\Tilde{H}_t + \lambda_1 \Tilde{M}_t) \leq \epsilon_1$, meaning
\begin{align}
    H_t + \lambda_1 M_t + (\phi(\Tilde{x}_t,x^*_t) - \phi(x_{t-1},x^*_{t-1})) &\leq H^*_t + \lambda_1\left(1 + \frac{\lambda_1}{\mu} \right) M^*_t + \epsilon_1
\end{align}
and summing this up we have the LHS as
\begin{align}
    H_t + \lambda_1 M_t  + \sum_{t=1}^T \phi(\Tilde{x}_t,x^*_t) - \phi(x_t,x^*_t).
\end{align}
Let's look at the second summation more closely,
\begin{align}
    \phi(\Tilde{x}_t,x^*_t) - \phi(x_t,x^*_t) &= \left(\frac{\mu+\lambda_1}{2}\right)(\|\Tilde{x}_t - x^*_t\|_2^2 - \|x_t - x^*_t\|_2^2)\\
    &= \left(\frac{\mu+\lambda_1 }{2}\right) (\Tilde{x}_t - x_t)^T(\Tilde{x}_t + x_t - 2x^*_t)\\
    &= \left(\mu+\lambda_1 \right) (\Tilde{x}_t - x_t)^T(x_t - x^*_t) + \left(\frac{\mu+\lambda_1 }{2}\right) \|\Tilde{x}_t - x_t\|_2^2\\
    &\geq -\left(\mu+\lambda_1 \right) \|\Tilde{x}_t - x_t\|_2\cdot\|x_t - x^*_t\|_2 \label{eqn29}\\
    &\geq -\epsilon_2\sqrt{\mu + \lambda_1}  \cdot \left( \frac{1 + (\mu+\lambda_1)\|x_t - x^*_t\|_2^2}{2} \right)\\
    &\geq -(\epsilon_2/2) \sqrt{\mu + \lambda_1} -\epsilon_2\sqrt{\mu + \lambda_1}  \cdot \left( \frac{(\mu+\lambda_1)\|x_t - \Tilde{x}_t + \Tilde{x}_t - x^*_t\|_2^2}{2} \right)\\
    &\geq -(\epsilon_2/2) \sqrt{\mu + \lambda_1} - 2\epsilon_2 \sqrt{\mu + \lambda_1} (H_t + \lambda_1 M_t + H^*_t + \lambda_1 M^*_t)
\end{align}
where we assume the approximation is also measurable as $\|\Tilde{x}_t - x_t\|_2 = \epsilon_2$. Plugging this in, we have
\begin{align}
\begin{split}
    \sum_{t=1}^T H_t + \lambda_1 M_t \leq \sum_{t=1}^T &\left(H^*_t + \lambda_1\left(1 + \frac{\lambda_1}{\mu} \right) M^*_t \right) + 2\epsilon_2 \sqrt{\mu + \lambda_1} \left(\sum_{t=1}^T H_t + \lambda_1 M_t + \sum_{t=1}^T H^*_t + \lambda_1 M^*_t \right)\\
    &+ (\epsilon_2 T/2) \sqrt{\mu + \lambda_1} + \epsilon_1 T
\end{split}
\end{align}
Rearranging, we get
\begin{align}
     (1-2\epsilon_2 \sqrt{\mu + \lambda_1}) \sum_{t=1}^T H_t + \lambda_1 M_t \leq (1+2\epsilon_2 \sqrt{\mu + \lambda_1})\sum_{t=1}^T \left(H^*_t + \lambda_1\left(1 + \frac{\lambda_1}{\mu} \right) M^*_t \right) + (\epsilon_2 T/2) \sqrt{\mu + \lambda_1} + \epsilon_1 T
\end{align}
and with $\lambda_1 \in (0,1]$
\begin{align}
    \left(1 - 2\epsilon_2 \frac{\sqrt{\mu + \lambda_1}}{\lambda_1} \right)\left(\sum_{t=1}^T H_t + M_t \right) \leq \left( CR_{\robd{}} + 2\epsilon_2 \frac{\sqrt{\mu + \lambda_1}}{\lambda_1} \right) \left(\sum_{t=1}^T H^*_t + M^*_t\right) + (\epsilon_2 T/2) \frac{\sqrt{\mu + \lambda_1}}{\lambda_1} + \epsilon_1 \frac{T}{\lambda_1}
\end{align}
where, $CR_{\robd{}} = \min\left\{\frac{1}{\lambda_1},1+\frac{\lambda_1}{\mu}\right\}$. The final result is
\begin{align}\label{eqn30}
    \sum_{t=1}^T H_t + M_t \leq \left( \frac{CR_{\robd{}} + 2\epsilon_2 \frac{\sqrt{\mu + \lambda_1}}{\lambda_1}}{1 - 2\epsilon_2 \frac{\sqrt{\mu + \lambda_1}}{\lambda_1}} \right) \left(\sum_{t=1}^T H^*_t + M^*_t\right) + \left(\frac{\frac{\epsilon_2 \sqrt{\mu + \lambda_1}}{2\lambda_1} + \frac{\epsilon_1}{\lambda_1}}{1 - 2\epsilon_2 \frac{\sqrt{\mu + \lambda_1}}{\lambda_1}} \right)T
\end{align}
\[
\pushQED{\qed} 
\qedhere
\popQED
\]  

\subsection{Proof of Theorem \ref{main_thm:CR}}\label{appendix_ssec: ACORD_CR}
Now for the multi-agent setting, $H_t:= \sum_{i=1}^N f_t^i(x_t^i) + \frac{\beta}{2}\|\mathbf{x}_t\|_{G_t}^2$, $H_t^*:= \sum_{i=1}^N f_t^i(x_t^{i,*}) + \frac{\beta}{2}\|\mathbf{x}_t^*\|_{G_t}^2$, $M_t^i := \frac{1}{2}\|x^i_t - x^i_{t-1}\|_2^2$ and $M_t^{i,*} := \frac{1}{2}\|x^{i,*}_t - x^{i,*}_{t-1}\|_2^2$.

To apply the approximate \robd{} framework, one needs to verify that multi-agent potential functions behave in the same way as \robd{}. Using the multi-agent potential function defined in \eqref{eqn: potential_fn},
\begin{align}
    \phi(\Tilde{\mathbf{x}}_t,\mathbf{x}^*_t) - \phi(\mathbf{x}_t,\mathbf{x}^*_t) &= \sum_{i=1}^N \left(\frac{\mu_i + \lambda_1^i}{2} \right)\left\| x^{*,i}_t - \Tilde{x}^{i}_t \right\|_2^2 - \sum_{i=1}^N \left(\frac{\mu_i + \lambda_1^i}{2} \right)\left\| x^{*,i}_t - x^{i}_t \right\|_2^2\\
    &= \sum_{i=1}^N \left(\frac{\mu_i + \lambda_1^i}{2} \right)\left\| x^{*,i}_t - \Tilde{x}^{i}_t \right\|_2^2 - \left(\frac{\mu_i + \lambda_1^i}{2} \right)\left\| x^{*,i}_t - x^{i}_t \right\|_2^2\\
    &\geq \sum_{i=1}^N -\left(\mu_i+\lambda_1^i \right) \|\Tilde{x}^i_t - x^i_t\|_2\cdot\|x^i_t - x^{*,i}_t\|_2
\end{align}
from \eqref{eqn29}.
Now, we apply Cauchy-Schwartz in the following way,
\begin{align}
    \phi(\Tilde{\mathbf{x}}_t,\mathbf{x}^*_t) - \phi(\mathbf{x}_t,\mathbf{x}^*_t) &\geq \sum_{i=1}^N -\left(\mu_i+\lambda_1^i \right) \|\Tilde{x}^i_t - x^i_t\|_2\cdot\|x^i_t - x^{*,i}_t\|_2\\
    &= -\sum_{i=1}^N \sqrt{\mu_i+\lambda_1^i} \|\Tilde{x}^i_t - x^i_t\|_2\cdot \sqrt{\mu_i+\lambda_1^i}\|x^i_t - x^{*,i}_t\|_2\\
    &\geq - \sqrt{\sum_{i=1}^N \left(\mu_i+\lambda_1^i \right)\|\Tilde{x}^i_t - x^i_t\|_2^2} \sqrt{\sum_{i=1}^N \left(\mu_i+\lambda_1^i \right)\|x^i_t - x^{*,i}_t\|_2^2}\\
    &\geq - \sqrt{\left(\max_i \mu_i+\lambda_1^i \right) \|\Tilde{\mathbf{x}}_t - \mathbf{x}_t\|_2^2} \sqrt{\sum_{i=1}^N \left(\mu_i+\lambda_1^i \right)\|x^i_t - x^{*,i}_t\|_2^2}\\
    &\geq - \sqrt{\left(\max_i \mu_i+\lambda_1^i \right)}\|\Tilde{\mathbf{x}}_t - \mathbf{x}_t\|_2 \cdot \left(\frac{1}{2} + \sum_{i=1}^N \frac{\mu_i+\lambda_1^i}{2}\|x^i_t - x^{*,i}_t\|_2^2 \right)\\
    &\geq - \sqrt{\max_i \mu_i+\lambda_1^i} \epsilon_2/2 -2\epsilon_2 \sqrt{\max_i \mu_i+\lambda_1^i} \left(  H_t + \sum_{i=1}^N \lambda_1^i M^i_t +  H^{*}_t + \sum_{i=1}^N \lambda_1^i M^{i,*}_t\right)
\end{align}
Therefore, \eqref{eqn30} with $\lambda_1^i = \frac{2}{1+\sqrt{1+\frac{4}{\mu_i}}}$ will translate to
\begin{align}\label{eqn31}
    \sum_{t=1}^T \left(H_t + \sum_{i=1}^N M^{i}_t \right) \leq \frac{CR_{*} + 2\epsilon_2 \frac{\sqrt{\max_i \mu_i + \lambda_1^i}}{\min_i \lambda_1^i}}{1-2\epsilon_2 \frac{\sqrt{\max_i \mu_i + \lambda_1^i}}{\min_i \lambda_1^i}} \sum_{t=1}^T \left(H^*_t + \sum_{i=1}^N M^{i,*}_t \right)  + \left( \frac{\frac{\epsilon_2\sqrt{\max_i \mu_i + \lambda_1^i}}{2\min_i \lambda_1^i} + \frac{\epsilon_1}{\min_i \lambda_1^i}}{1-2\epsilon_2 \frac{\sqrt{\max_i \mu_i + \lambda_1^i}}{\min_i \lambda_1^i}} \right)T
\end{align}
Now, we set $\epsilon_2$ and $\epsilon_1$. Recall from Theorem \ref{thm: alt_min_conv_l2} the exponentially fast convergence of Alternating Minimization for 
\begin{align*}
    \mathbbm{F}_t(\mathbf{x},\mathbf{z}) = \sum_{i=1}^N \left\{ f_t^i(x^i) + \frac{\lambda_1^i}{2}\|x^i - x^i_{t-1}\|_2^2 + \beta \sum\limits_{\mathcal{E}_t \ni e \ni i}\|A_t^{e}(x) - A_t^{e}\left(z_e\right)\|_{2}^2 \right\}
\end{align*}
that is,
\begin{align}
    \mathbbm{F}_t(\mathbf{x}_k,\mathbf{z}_k) - \mathbbm{F}_t(\mathbf{\Tilde{x}}_t,\mathbf{z}_*) \leq \left(1-\frac{\sigma_t}{4\beta l} \right)^{k-1} (\mathbbm{F}_t(\mathbf{x}_0,\mathbf{z}_0) - \mathbbm{F}_t(\mathbf{x}^R_t,\mathbf{z}_*)).
\end{align}
Now, use the strong convexity and smoothness (in $\mathbf{z}$) of $F(\mathbf{x},\mathbf{z})$ to have
\begin{align}
    \epsilon_2 = \|\mathbf{x}_{k+1} - \mathbf{\Tilde{x}}_t\|_2 &\leq \left\|(\mathbf{x}_{k+1},\mathbf{z}_{k}) - (\mathbf{\Tilde{x}}_t,\mathbf{z}_*)\right\|_2\\
    &= \left\|\mathbf{w}_{k+\frac{1}{2}} - \mathbf{w}_*\right\|_2\\
    &\leq \frac{\left\|\nabla \mathbbm{F}_t\left(\mathbf{w}_{k+\frac{1}{2}}\right) \right\|_2}{\sigma_t}\\
    &\leq \frac{\sqrt{8\beta l\left(\mathbbm{F}_t\left(\mathbf{w}_{k+\frac{1}{2}}\right) - \mathbbm{F}_t\left(\mathbf{w}_{k+\frac{3}{2}}\right) \right) }}{\sigma_t}\\
    &\leq \frac{\sqrt{8\beta l \left(\mathbbm{F}_t\left(\mathbf{w}_{k+\frac{1}{2}}\right) - \mathbbm{F}_t\left(\mathbf{w}_{*}\right) \right) }}{\sigma_t}\\
    &\leq \frac{\sqrt{8\beta l \left(\mathbbm{F}_t\left(\mathbf{w}_{k}\right) - \mathbbm{F}_t\left(\mathbf{w}_{*}\right) \right) }}{\sigma_t}\\
    &\leq \frac{\sqrt{8\beta l \left(1-\frac{\sigma_t}{4\beta l} \right)^{k-1} \left(\mathbbm{F}_t(\mathbf{w}_0) - \mathbbm{F}_t(\mathbf{w}_*)\right) }}{\sigma_t} = \left(\frac{\min_i \lambda_1^i}{2 \sqrt{\max_i \mu_i + \lambda_1^i}}\right)\frac{1}{2T^2}\label{eqn12}
\end{align}

Now, we control $\epsilon_1$. Observe that $(\Tilde{z}_e)_k = \frac{(x^i_t)_k +  (x^j_t)_k}{2} \mathbbm{1}_{e = (i,j)}$. This, along with $(\mathbf{x}_t)_k$, in \eqref{eqn12}
We see that many iterations will handle $\epsilon_1$ too. We start with
\begin{align}
    F_t((\mathbf{x}_t)_k) =& \sum_{i=1}^N \left\{ f_t^i((x^i_t)_k) + \frac{\lambda_1^i}{2}\|(x^i_t)_k - x^i_{t-1}\|_2^2\right\} + \frac{\beta}{2} \sum_{(i,j) \in \mathcal{E}} \|A^{(i,j)}_t \left((x^i_t)_k\right) - A^{(i,j)}_t \left((x^j_t)_k\right)\|_2^2\\
    \begin{split}
        =& \sum_{i=1}^N \left\{ f_t^i((x^i_t)_k) + \frac{\lambda_1^i}{2}\|(x^i_t)_k - x^i_{t-1}\|_2^2\right\}\\
        &+ \frac{\beta}{2} \sum_{e=(i,j) \in \mathcal{E}} \|A^{(i,j)}_t \left((x^i_t)_k\right) -A^{(i,j)}_t\left((z_{e})_k\right) + A^{(i,j)}_t\left((z_{e})_k\right) - A^{(i,j)}_t \left((x^j_t)_k\right)\|_2^2
    \end{split}\\
    \begin{split}
        \leq& \sum_{i=1}^N \left\{ f_t^i((x^i_t)_k) + \frac{\lambda_1^i}{2}\|(x^i_t)_k - x^i_{t-1}\|_2^2\right\}\\
        &+ \frac{\beta}{2} \sum_{e=(i,j) \in \mathcal{E}} 2\|A^{(i,j)}_t \left((x^i_t)_k\right) - A^{(i,j)}_t\left((z_{e})_k\right)\|_2^2 + 2\|A^{(i,j)}_t\left((z_{e})_k\right) - A^{(i,j)}_t \left((x^j_t)_k\right)\|_2^2
    \end{split}\\
    \begin{split}
        =& \sum_{i=1}^N \left\{ f_t^i((x^i_t)_k) + \frac{\lambda_1^i}{2}\|(x^i_t)_k - x^i_{t-1}\|_2^2 + \beta \sum_{\mathcal{E} \ni e \ni i} \|A^{e}_t \left((x^i_t)_k\right) - A^{e}_t\left((z_{e})_k\right)\|_2^2\right\}
    \end{split}\\
    =& \mathbbm{F}_t(\left(\mathbf{x}_t\right)_k,\mathbf{z}_k).
\end{align}

\begin{align}
    \epsilon_1 = \left(H_t + \sum_{i=1}^N \lambda_1^i M_t^i \right) - \left(\Tilde{H}_t + \sum_{i=1}^N \frac{\lambda_1^i}{2}\|\Tilde{x}_t^i - x_{t-1}^i\|_2^2\right) &= F_t((\mathbf{x}_t)_k) - F_t(\Tilde{\mathbf{x}_t})\\
    &= F_t((\mathbf{x}_t)_k) - \mathbbm{F}_t(\mathbf{\Tilde{x}}_t,\mathbf{z}_*) \\
    &\leq \mathbbm{F}_t(\left(\mathbf{x}_t\right)_k,\mathbf{z}_k) - \mathbbm{F}_t(\mathbf{\Tilde{x}}_t,\mathbf{z}_*) \\
    &\leq \left(1-\frac{\sigma_t}{4\beta l} \right)^{k-1} (\mathbbm{F}_t(\mathbf{x}_t)_0,\mathbf{z}_0) - \mathbbm{F}_t(\mathbf{\Tilde{x}}_t,\mathbf{z}_*))\\
    &\leq \frac{\sigma_t^2}{8 \beta l}\left(\frac{\min_i \lambda_1^i}{2 \sqrt{\max_i \mu_i + \lambda_1^i}}\right)^2 \frac{1}{4T^4}\\
    &\leq 2\beta (m^2/l)\left(\frac{\min_i \lambda_1^i}{2 \sqrt{\max_i \mu_i + \lambda_1^i}}\right)^2\frac{1}{4T^4}    
\end{align}
where we use the fact that $\sigma_t \leq 4 \beta m$ from \eqref{eqn32}. Putting $\epsilon_1$ and $\epsilon_2$ back into \eqref{eqn31},
\begin{align}
    \sum_{t=1}^T \left(H_t + \sum_{i=1}^N M^{i}_t \right) &\leq \left(\frac{CR_{*} + 1/2T^2}{1-1/2T^2} \right) \sum_{t=1}^T \left(H^*_t + \sum_{i=1}^N M^{i,*}_t \right) + \left( \frac{\frac{1}{8T^2} + \left(\frac{\min_i \lambda_1^i}{\max_i \mu_i + \lambda_1^i}\right)\frac{\beta (m^2/l)}{8T^4}}{1-1/2T^2} \right)T\\
    &\leq \left(\frac{CR_{*} + 1/2T^2}{1-1/2T^2} \right) \sum_{t=1}^T \left(H^*_t + \sum_{i=1}^N M^{i,*}_t \right) + \left( \frac{1}{4} + \left(\frac{\min_i \lambda_1^i}{\max_i \mu_i + \lambda_1^i}\right)\frac{\beta (m^2/l)}{4T^2} \right)\frac{T}{2T^2-1}\\
    &\leq \left(\frac{CR_{*} + 1/2T^2}{1-1/2T^2} \right) \sum_{t=1}^T \left(H^*_t + \sum_{i=1}^N M^{i,*}_t \right) + \frac{1+\beta (m^2/l)}{4T}
\end{align}
\[
\pushQED{\qed} 
\qedhere
\popQED
\]  

\section{Dependence of $\sigma_t$ on graph: Proof of Theorem \ref{main_thm:graph_dependence}}\label{appendix_sec: graph_dependence}
Recall that strong-convexity parameter $\sigma$ (we drop the $t$ subscript as it is understood here) of $\mathbbm{F}_t(\mathbf{x},\mathbf{z})$ is the smallest eigenvalue of
\begin{align*}
        \nabla^2 F(x,z) \succeq 
        \underbrace{\begin{bmatrix}
            \mathbf{D}^{f}_{Nd\times Nd} & \mathbf{0}_{Nd \times |\mathcal{E}_t|d}\\
            \mathbf{0}_{|\mathcal{E}_t|d \times Nd} & \mathbf{0}_{|\mathcal{E}_t|d \times |\mathcal{E}_t|d}
        \end{bmatrix}}_A + 2\beta m 
        \underbrace{\begin{bmatrix}
            \mathbf{D}^G_{Nd \times Nd} & \mathbf{G}_z\\
            \mathbf{G}_z^T & 2\mathbf{I}_{|\mathcal{E}_t|d}
        \end{bmatrix}}_B.
\end{align*}
Let $M = A + 2\beta m B$. 
\begin{align}\label{eqn18}
    M = \begin{bmatrix}
            \mathbf{D}^{f}_{Nd\times Nd} + 2\beta m \mathbf{D}^G_{Nd \times Nd} & 2\beta m \mathbf{G}_z\\
            2\beta m \mathbf{G}_z^T & 4\beta m \mathbf{I}_{|\mathcal{E}_t|d}
        \end{bmatrix} \succeq \begin{bmatrix}
            \min_i \{ \mu_i + \lambda_1^i + 2\beta m \mathcal{D}_i \} \mathbf{I}_{Nd \times Nd} & 2\beta m \mathbf{G}_z\\
            2\beta m \mathbf{G}_z^T & 4\beta m \mathbf{I}_{|\mathcal{E}_t|d}
        \end{bmatrix} = M_1
\end{align}
Let the SVD of $\mathbf{G}_z = U \Sigma V^T$ with singular values $\sigma_1 \geq \sigma_2 \geq \ldots \sigma_r$ . The next steps are for assuming $d=1$ but the analysis can be extended giving the same result for $d>1$ . The left and right eigenvectors be $\{u_1,\ldots,u_r\} \subset \R^{N}$ and $\{v_1,\ldots,v_r\} \subset \R^{|\mathcal{E}_t|}$ respectively. Further, let the remaining eigenvectors in each space be $\{u_{r+1},\ldots,u_N\}$ and $\{v_{r+1},\ldots,v_{|\mathcal{E}_t|}\}$ resp. The following are the eigenvectors of $M_1$,
\begin{align}
    \underbrace{\begin{bmatrix}
        u_{r+1} \\ \mathbf{0}
    \end{bmatrix}, \ldots, \begin{bmatrix}
        u_{N} \\ \mathbf{0}
    \end{bmatrix}}_{(1)}, \underbrace{\begin{bmatrix}
        \mathbf{0} \\ v_{r+1}
    \end{bmatrix}, \ldots,
    \begin{bmatrix}
        \mathbf{0} \\ v_{|\mathcal{E}_t|},
    \end{bmatrix}}_{(2)}, \underbrace{\begin{bmatrix}
        b^1_1 u_1 \\
        b^1_2 v_1
    \end{bmatrix},\ldots,
    \begin{bmatrix}
        b^r_1 u_r \\
        b^r_2 v_r
    \end{bmatrix}}_{(3)},\underbrace{\begin{bmatrix}
        c^1_1 (-u_1) \\
        c^1_2 v_1
    \end{bmatrix},\ldots,\begin{bmatrix}
        c^r_1 (-u_r) \\
        c^r_2 v_r
    \end{bmatrix}}_{(4)}.
\end{align}
For $(1)$ and $(2)$, it is easy to see from the fact that $\mathbf{G}_z^T u_i = 0$ for $i>r$ and $\mathbf{G}_z v_i = 0$ for $i>r$. The corresponding eigenvalues are $\min_i \{ \mu_i + \lambda_1^i + 2\beta m d_i \}$ and $4\beta m$ resp. Denote $a = \min_i \{ \mu_i + \lambda_1^i + 2\beta m d_i \}$. For $(3)$, we need to find the appropriate $(b^k_1,b^k_2)$ 
\begin{align}
    \begin{bmatrix}
            a \mathbf{I}_{N} & 2\beta m \mathbf{G}_z\\
            2\beta m \mathbf{G}_z^T & 4\beta m \mathbf{I}_{|\mathcal{E}_t|}
        \end{bmatrix}\begin{bmatrix}
        b^k_1 u_k \\
        b^k_2 v_k
    \end{bmatrix} = \begin{bmatrix}
        a b^k_1 u_k + 2\beta m \sigma_k b^k_2 u_k\\
        2\beta m \sigma_k b^k_1 v_k + 4 \beta m b^k_2 v_k
    \end{bmatrix}
\end{align}
Therefore, $(b^k_1,b^k_2)$ need to satisfy
\begin{align}
\frac{a b^k_1 + 2\beta m \sigma_k b^k_2}{2\beta m \sigma_k b^k_1 + 4 \beta m b^k_2}
 = \frac{b^k_1}{b^k_2}
\end{align}
and using same method, $(c^k_1,c^k_2)$ need to satisfy
\begin{align}
    \frac{a c^k_1 - 2\beta m \sigma_k c^k_2}{-2\beta m \sigma_k c^k_1 + 4 \beta m c^k_2}
 = \frac{c^k_1}{c^k_2}
\end{align}
Let's solve for $b^k_2/b^k_1 > 0$ as the negative version will be handled by $(c^k_1,c^k_2)$,
\begin{align}
    \frac{a + 2\beta m \sigma_k x}{2\beta m \sigma_k (1/x) + 4\beta m} = 1
\end{align}
or
\begin{align}
    \frac{a}{2\beta m \sigma_k} + x = \frac{1}{x} + \frac{4\beta m}{2\beta m \sigma_k}
\end{align}
\begin{align}
    x^2 - x\left(\frac{4\beta m - a}{2\beta m \sigma_k}\right) - 1 = 0
\end{align}
\begin{align}
    \frac{b^k_2}{b^k_1} = \frac{1}{2}\left( \frac{4\beta m - a}{2\beta m \sigma_k} + \sqrt{\left(\frac{4\beta m - a}{2\beta m \sigma_k}\right)^2  +4} \right)
\end{align}
and the eigenvalue is
\begin{align}
    \frac{a b^k_1 + 2\beta m \sigma_k b^k_2}{b^k_1} &= a + 2\beta m \sigma_k x\\
    &= a + \frac{1}{2}\left( 4\beta m - a + \sqrt{(4\beta m - a)^2 + 4(2\beta m \sigma_k)^2} \right)\\
    &= \frac{1}{2}\left( 4\beta m + a + \sqrt{(4\beta m - a)^2 + 4(2\beta m \sigma_k)^2} \right)
\end{align}
Now, we solve for $c^k_2/c^k_1 > 0$,
\begin{align}
    \frac{a - 2\beta m \sigma_k x}{-2\beta m \sigma_k (1/x) + 4\beta m} = 1
\end{align}
or
\begin{align}
    \frac{a}{2\beta m \sigma_k} -x = -\frac{1}{x} + \frac{4\beta m}{2\beta m \sigma_k}
\end{align}
\begin{align}
    x^2 + x\left(\frac{4\beta m - a}{2\beta m\sigma_k} \right) -1 = 0
\end{align}
giving
\begin{align}
    \frac{c^k_2}{c^k_1} = \frac{1}{2} \left(-\left(\frac{4\beta m - a}{2\beta m\sigma_k} \right) + \sqrt{\left(\frac{4\beta m - a}{2\beta m\sigma_k} \right)^2 + 4} \right)
\end{align}
and the eigenvalue being
\begin{align}
    \frac{a c^k_1 - 2\beta m \sigma_k c^k_2}{c^k_1} &= a - 2\beta m \sigma_k x\\
    &= a - \frac{1}{2}\left( a-4\beta m +\sqrt{(4\beta m-a)^2 + 4(2\beta m\sigma_k)^2} \right)\\
    &= \frac{1}{2}\left( a + 4\beta m - \sqrt{(4\beta m-a)^2 + 4(2\beta m\sigma_k)^2} \right).
\end{align}
Therefore, $\sigma$ being the smallest eigenvalue being satisfies
\begin{align}
    \sigma \geq \frac{1}{2}\left( a + 4\beta m - \sqrt{(4\beta m-a)^2 + 4(2\beta m\sigma^{\mathbf{G}_z}_{\max})^2} \right)
\end{align}
For $\mathcal{D}$-regular graphs, we can evaluate further. First,
\begin{align}\label{eqn15}
    a = \min_i \{2\beta m \mathcal{D}_i + \mu_i + \lambda_1^i\} = 2\beta m (\mathcal{D} + \min_i \kappa_i),
\end{align}
where $\kappa_i = \frac{\mu_i + \lambda_1^i}{2\beta m}$ represents the relative influence of local costs v/s dissimilarity costs. Further,
\begin{align}\label{eqn16}
    \sigma \geq \frac{1}{2} \left(\frac{16 a \beta m - 16 (\beta m)^2 (\sigma^{\mathbf{G}_z}_{\max})^2}{a + 4\beta m + \sqrt{(4\beta m-a)^2 + 4(2\beta m\sigma^{\mathbf{G}_z}_{\max})^2}} \right).
\end{align}
For a $\mathcal{D}$-regular graph,
\begin{align}
    \mathbf{G}_z \mathbf{G}_z^T = \mathcal{D}\cdot I + \mathbf{A}^G
\end{align}
implying
\begin{align}
    \sigma^{\mathbf{G}_z \mathbf{G}_z^T}_i = \mathcal{D}+  \sigma^{\mathbf{A}^G}_i
\end{align}
and the largest eigenvalue of the adjacency matrix for a $d$-regular graph is $d$ and smallest is larger than $-d$ (by Perron-Frobeneus Theorem). Therefore,
\begin{align}\label{eqn17}
    (\sigma^{\mathbf{G}_z}_{\max})^2 = \sigma^{\mathbf{G}_z \mathbf{G}_z^T}_{\max} = 2\mathcal{D}.
\end{align}
Plugging \eqref{eqn15} and \eqref{eqn17} into \eqref{eqn16},
\begin{align}
    \sigma &\geq \frac{1}{2} \left(\frac{32 (\beta m)^2 (\mathcal{D} + \min_i \kappa_i) - 32 (\beta m)^2 \mathcal{D}}{a + 4\beta m + \sqrt{(4\beta m-a)^2 + 32 (\beta m)^2 \mathcal{D}}} \right)\label{eqn19}\\
    &= \frac{(16 \beta m)^2 \min_i \kappa_i}{a + 4\beta m + \sqrt{(4\beta m-a)^2 + 32 (\beta m)^2 \mathcal{D}}}\\
    &= 4\beta m \left( \frac{\min_i \kappa_i}{1 + \frac{a}{4\beta m} + \sqrt{\left(1-\frac{a}{4\beta m}\right)^2 + 2\mathcal{D}}} \right)\\
    &= 4\beta m \left( \frac{\min_i \kappa_i}{1 + \frac{\mathcal{D}+ \min_i \kappa_i}{2} + \sqrt{\left(1-\frac{\mathcal{D}+ \min_i \kappa_i}{2}\right)^2 + 2\mathcal{D}}} \right)\label{eqn32}\\
\end{align}
and this is tight. We can verify it by seeing that \eqref{eqn18} (and consequently \eqref{eqn19}) is an equality for $\mu_i = \mu$ $\forall$ $i$. We can further upper bound it for $\mu_i = \mu$ as
\begin{align}
    \sigma &= 4\beta m \left( \frac{\min_i \kappa_i}{1 + \frac{\mathcal{D}+ \min_i \kappa_i}{2} + \sqrt{\left(1-\frac{\mathcal{D}+ \min_i \kappa_i}{2}\right)^2 + 2\mathcal{D}}} \right)\\
    &\leq \frac{2(\mu + \lambda_1)}{\mathcal{D} +\kappa}
\end{align}
with $\kappa = \frac{\mu+\lambda_1}{2\beta m}$, where $\lambda_1 = \frac{2}{1+\sqrt{1+\frac{4}{\mu}}}$. Also for $\mathcal{D}\geq 2$ we have,
\begin{align}
    \sigma &= \frac{\mu + \lambda_1}{1 + \frac{\mathcal{D}+ \mu + \lambda_1}{2} + \sqrt{\left(1-\frac{\mathcal{D}+ \kappa}{2}\right)^2 + 2\mathcal{D}}}\\
    &\geq 4\beta m  \frac{\mu + \lambda_1}{1 + \frac{\mathcal{D}+ \mu + \lambda_1}{2} + \sqrt{\left(1-\frac{\mathcal{D}+ \kappa}{2}\right)^2} + \sqrt{2\mathcal{D}}}  \\
    &\geq 4\beta m \left( \frac{\mu + \lambda_1}{2\mathcal{D} + \kappa}  \right)\\
    &= \frac{2(\mu+\lambda_1)}{2\mathcal{D} + \kappa}
\end{align}

We have therefore shown that for $\mu_i = \mu$ and $\mathcal{D}$-regular graph $\mathcal{G}_t$ for $\mathcal{D}\geq 2$,
\begin{align}
    \frac{2(\mu+\lambda_1)}{2\mathcal{D} + \kappa} \leq \sigma_t \leq \frac{2(\mu + \lambda_1)}{\mathcal{D} +\kappa}
\end{align}

The effect on $K_t$ is as follows
\begin{align}
    K_t &= \frac{\log \left(\frac{T^4 \cdot 32\beta l N(M_f + M_s^2/2)(\max_i \mu_i + \lambda_1^i)}{(\sigma_t\min_i \lambda_1^i)^2} \right)}{\log\left( \frac{4\beta l}{4\beta l - \sigma_t} \right)}\\
    &= \frac{\log (N/\sigma_t^2) + \log \left(\frac{T^4 \cdot 32\beta l (M_f + M_s^2/2)(\max_i \mu_i + \lambda_1^i)}{(\min_i \lambda_1^i)^2} \right)}{\log \left(\frac{4\beta l}{4\beta l - \sigma_t}\right)}\\
    &= \frac{-\log (N/\sigma_t^2) - \log \left(\frac{T^4 \cdot 32\beta l (M_f + M_s^2/2)(\max_i \mu_i + \lambda_1^i)}{(\min_i \lambda_1^i)^2} \right)}{\log \left(1-\frac{\sigma_t}{4\beta l}\right)}
\end{align}
For large enough $\mathcal{D}$ and small $\mu$, $\frac{\sigma_t}{4\beta l}$ is small. $\frac{-1}{\log \left(1-\frac{1}{x}\right)} \approx b\cdot x$ for large $x$, meaning
\begin{align}
    K_t = \frac{4b\beta l}{\sigma_t} \left( \log (N/\sigma_t^2) + \underbrace{\log \left(\frac{T^4 \cdot 32\beta l (M_f + M_s^2/2)(\max_i \mu_i + \lambda_1^i)}{(\min_i \lambda_1^i)^2} \right)}_{c} \right)
\end{align}
and hence,
\begin{align}
    \left(\frac{\mathcal{D} +\kappa}{2(\mu + \lambda_1)}\right) \left(\log \left(N \cdot \left(\frac{\mathcal{D} +\kappa}{2(\mu + \lambda_1)}\right)^2 \right) +c\right) \leq \frac{K_t}{4 b\beta l} \leq \left(\frac{2\mathcal{D} +\kappa}{2(\mu + \lambda_1)}\right) \left(\log \left(N \cdot \left(\frac{2\mathcal{D} +\kappa}{2(\mu + \lambda_1)}\right)^2 \right) +c\right)
\end{align}
proving that $K_t = \Theta(\mathcal{D} \log N \mathcal{D}^2)$.
\[
\pushQED{\qed} 
\qedhere
\popQED
\]

\section{Lower Bound}\label{appendix_sec: lower_bounds}

\begin{theorem}
    In the decentralized setting, consider the agents' hitting costs $\{f^i_t\}_{i,t}$ to be $\mu$-strongly convex and the switching costs to be squared $\ell_2$-norm. Further, consider any dissimilarity that is degenerate along the $x^1 = \ldots = x^N$ line. For weight $\beta > 0$, the competitive ratio is lower bounded as 
    \begin{align}
        CR_{ALG} \geq \frac{1}{2} + \frac{1}{2}\sqrt{1+\frac{4}{\mu}}
    \end{align}
    with no helping effect from the dissimilarity cost.
\end{theorem}
\begin{proof}
    We consider $N$ agents in the action space $\R$. The dissimilarity cost class considered is the most general. The degeneracy along $x^1 = \ldots = x^N$ is necessary to ensure the fact that all actions being the same translates to zero dissimilarity cost.

    The hitting costs for all agents is as follows,
    \begin{align}
        f^i_t(x) = \begin{cases}
            \frac{\mu}{2}(x^i)^2 & t \in \{1,\ldots,n\}\\
            \frac{\mu'}{2}(x^i -1)^2 & t = n+1
        \end{cases} \text{ } \forall \text{ } i \in \{1,\ldots,N\}.
    \end{align}
    This means the total hitting cost for the system can be written as
    \begin{align}
        f_t(\mathbf{x}) = \begin{cases}
            \frac{\mu}{2}\|\mathbf{x}\|_2^2 & t \in \{1,\ldots,T\}\\
            \frac{\mu'}{2}\|\mathbf{x} - \mathbbm{1}\|_2^2 & t = T+1
        \end{cases}
    \end{align}
    where $\begin{bmatrix}
        x^1 & \ldots & x^N
    \end{bmatrix}^T = \mathbf{x} \in \R^{N}$ and $\mathbbm{1} = \begin{bmatrix}
        1 & 1 &\ldots & 1
    \end{bmatrix}^T$. Now, any ALG has two options, either to move at $t = T+1$ or before. If it moves before, the environment stops at that moment and the competitive ratio  = $\frac{>0}{0} = \infty$ as the adversary will not be moving. So, the only option a competitive algorithm has is to move at $t = T+1$ and stay at $\mathbf{0}_N$ before that.

    At $t=T+1$, ALG will move to the point $\mathbf{x} \in \R^N$ such that its total hitting plus switching cost is minimum, that is, 
\begin{align}
    cost(ALG) \geq \min_{x^1,\ldots,x^N} \frac{\mu'}{2}\|\mathbf{x} - \mathbbm{1}\|_2^2 +\frac{1}{2}\|\mathbf{x}^1\|_2^2 + \frac{\beta}{2}s(\mathbf{x})
\end{align}
where $s(\cdot)$ is the general dissimilarity cost class. Based on the symmetry of the optimization problem above with respect to the hitting and switching costs and the fact that $s(\mathbf{x}) = 0$ when all components are same, we conclude that the minimum above occurs at $\mathbf{x} = x\cdot \mathbbm{1}$. Therefore,
\begin{align}
    cost(ALG) \geq N\cdot \min_{x} \frac{\mu'}{2}(x - 1)^2 +\frac{1}{2}x^2 = \frac{N}{2\left(1+\frac{1}{\mu'}\right)}
\end{align}
Instead of the hindsight optimal $OPT$, an algorithm $ADV$ that has hindsight knowledge is used, meaning $cost(ADV) \geq cost(OPT)$. It is clear that for $m' \to \infty$, $ADV$ has to be at $\mathbbm{1}$. We consider ADV to be a sequence of actions of the following form
\begin{align}
    \mathbf{0} = x^*_0 \cdot \mathbbm{1} \to x^*_1 \cdot \mathbbm{1} \to \ldots \to x^*_T \cdot \mathbbm{1} \to x^*_{T+1} \cdot \mathbbm{1} = \mathbbm{1}
\end{align}
which will have a cost of 
\begin{align}
    cost(ADV) = \frac{N \cdot a_T}{2} =  N \cdot \left(\min_{(x^{*}_0,\ldots,x^{*}_{T+1}) \in \mathcal{K}(T,1)} \left\{ \sum_{t=1}^T \frac{\mu}{2} \left(x^{*}_t \right)^2 + \sum_{t=1}^{T+1}\frac{1}{2} \left(x^{*}_t - x^{*}_{t-1} \right)^2 \right\} \right)
\end{align}
where $\mathcal{K}(T,1) = \{(x^0,\ldots,x^{T+1}) \in \R^{T+2}: x_i \leq x_{i+1}, x_0 = 0, x_{T+1} = 1\}$ and $s(x\cdot\mathbbm{1}) = 0$. From Lemma 7 of \cite{GoelLinWierman19}, we know
\begin{align}
    \lim_{T \to \infty} a_T = \frac{-\mu + \sqrt{\mu^2 + 4\mu}}{2}
\end{align}
which gives $cost(ADV) = \frac{N}{2}\left(\frac{-\mu + \sqrt{\mu^2 + 4\mu}}{2}\right)$. The competitive ratio lower bound is
\begin{align}
    CR_{ALG} \geq \lim_{T\to \infty, \mu' \to \infty} \frac{cost(ALG)}{cost(ADV)} \geq \frac{1}{2} + \frac{1}{2}\sqrt{1+\frac{4}{\mu}} 
\end{align}
\end{proof}

\begin{theorem}
For the class of strongly convex hitting costs of the form $f_t(\mathbf{x}) = \sum_{i=1}^N f_t^i \left({x^i} \right) + \beta g(\mathbf{x})$ where $f_t^i(\cdot)$ is $\mu_i$-strongly convex and $g(\cdot)$ is convex (and not strongly convex), the competitive ratio is lower bounded
\begin{align*}
    CR \geq \frac{1}{2} + \frac{1}{2}\sqrt{1 + \frac{4}{\min_i \mu_i}}
\end{align*}
\end{theorem}
\begin{proof}
The proof is along the same lines as that in \cite{GoelLinWierman19} but in multi-dimension with change in function form (across time) only along the direction of the minimum strong-convexity parameter. We consider the following instance of the problem with $N$ agents in $\R$ with strong-convexity parameters $\mu_1 \leq \mu_2 \leq \ldots \leq \mu_N$:
\begin{align*}
    f^1_t(x^1_t) = \begin{cases}
        \frac{\mu_1}{2}(x^1_t)^2 & t \in \{1,\ldots,T\}\\
        \frac{m'}{2}(x^1_t - 1)^2 & t = T+1
    \end{cases}
\end{align*}
and for all other agents $i\geq 2$, $f^i_t(x^i_t) = \frac{\mu_i}{2}(x^i_t)^2$ $\forall$ $ t \in \{1,\ldots,T+1\}$. We consider $\beta = 0$ throughout this instance. Now, any ALG has two options, either to move at $t = T+1$ or before. If it moves before, the environment stops at that moment and the competitive ratio  = $\frac{>0}{0} = \infty$ as the adversary will not be moving. So, the only option a competitive algorithm has is to move at $t = T+1$ and stay at $\mathbf{0}_N$ before that.

At $t=T+1$, ALG will move to the point $\mathbf{x} \in \R^N$ such that its total hitting plus switching cost is minimum, that is
\begin{align}
    cost(ALG) \geq \min_{x^1,\ldots,x^N} \frac{m'}{2}(x^1 - 1)^2 +\frac{(x^1)^2}{2} + \sum_{i=2}^N \frac{\mu_i}{2}(x^i)^2 + \frac{(x^i)^2}{2} 
\end{align}
For $x^i$, $i\neq 1$, the optimal point is still zero and so is their hitting plus switching cost, meaning
\begin{align}
    cost(ALG) \geq \min_{x^1} \frac{m'}{2}(x^1 - 1)^2 +\frac{(x^1)^2}{2} = \frac{1}{2\left(1+\frac{1}{m'}\right)}
\end{align}
Instead of the hindsight optimal $OPT$, an algorithm $ADV$ that has hindsight knowledge is used, meaning $cost(ADV) \geq cost(OPT)$. It is clear that for $m' \to \infty$, $ADV$ has to be at $1$ along the first axis. The stationary nature of hitting costs along all other axes $i \geq 2$ make it moot for moving from zero along those directions. Meaning, $ADV$ takes a sequence of points $$(x^{1*}_0 \to \mathbf{0}_{N-1}) \to (x^{1*}_1,\mathbf{0}_{N-1})\to \ldots \to  (x^{1*}_T,\mathbf{0}_{N-1}).$$
Since, the contribution to the cost is zero from other agents, cost is a function of $(x^{1*}_0,\ldots,x^{1*}_{T+1}) \in \R^{T+2}$ with the initial point as $0$ and the final point constrained to be $1$ (to accommodate for $m' \to \infty$). This is exactly same as Lemma 7 in \cite{GoelLinWierman19}, where the following optimization problem is solved to find $cost(ADV)$,
\begin{align}
   cost(ADV) = \frac{a_T}{2} =  \min_{(x^{1*}_0,\ldots,x^{1*}_{T+1}) \in \mathcal{K}(T,1)} \left( \sum_{t=1}^T \frac{\mu_1}{2} \left(x^{1*}_t \right)^2 + \sum_{t=1}^{T+1}\frac{1}{2} \left(x^{1*}_t - x^{1*}_{t-1} \right)^2 \right)
\end{align}
where $\mathcal{K}(T,1) = \{(x^0,\ldots,x^{T+1}) \in \R^{T+2}: x_i \leq x_{i+1}, x_0 = 0, x_{T+1} = 1\}$. The lemma states that $ \lim_{n\to \infty} a_n = \frac{-\mu_1+\sqrt{\mu_1^2 + 4\mu_1}}{2}$.

The values of $cost(ALG)$ and $cost(ADV)$ are exactly the same as in Lemma 7 of \cite{GoelLinWierman19}, leading to the conclusion that
\begin{align}
    CR_{ALG} &= \frac{cost(ALG)}{cost(OPT}\\
    &\geq \frac{cost(ALG)}{cost(ADV)} = \lim_{n,m' \to \infty} \frac{\frac{1}{2\left(1+\frac{1}{m'}\right)}}{\frac{a_n}{2}} = \frac{1}{2} + \frac{1}{2}\sqrt{1+\frac{4}{\mu_1}} = \frac{1}{2} + \frac{1}{2}\sqrt{1+\frac{4}{\min_i \mu_i}}
\end{align}
\end{proof}

\section{\lpccaps{} comparison and Numerical Experiments Set-up}
\subsection{Proof of Corollary \ref{corr: resource_utlization}}\label{proof: resource_comp}
We consider the most basic hitting costs to deal with, under our strongly convex assumptions: quadratic costs and, a static $\mathcal{D}$-regular graph. Then, \lpc{} requires each agent to invert a matrix of size $|\mathcal{N}_i^r| \cdot k \cdot d \times |\mathcal{N}_i^r| \cdot k \cdot d$. On the contrary, \acord{} needs to each agent to invert only a $d \times d$ matrix at most $\Theta\left(\mathcal{D}\log N\mathcal{D}^2 T^4\right)$ times. The former involves a complexity of $\Omega(d^3 |\mathcal{N}_i^r|^3 k^3)$ while the latter exhibits $\mathcal{O}(d^3 \mathcal{D}\log (N\mathcal{D}^2 T^4)) \approx \Tilde{\mathcal{O}}(d^3 \mathcal{D})$, as $T \ll d$ in the large scale applications we consider, like decentralized management of data centers and power grids. Even without future predictions ($k=1$),
\begin{align}
    \mathcal{D} < \mathcal{D}^3 \ll |\mathcal{N}_i^r|^3 = |\mathcal{N}^r(\mathcal{D})|^3
\end{align}
for large $\mathcal{D}$-regular networks. 

\subsection{Set-up for Numerical Experiments}\label{appendix_sec: num_exp_descp}
As previously stated, we considered one-dimensional action space and quadratic hitting costs $\alpha_t^i (x-v^i_t)^2$ for each agent $i \in \{1,\ldots,N\}$. For each agent $i$, the sequence of $\{\alpha_t^i\}_t$ is independently generated in the following manner:
\begin{align}
    \alpha_t^i = \begin{cases}
        U[0,1] & \text{ with prob. } p\\
        U[0,1] + 2^t & \text{ otherwise}
    \end{cases}
\end{align}
where we set $p = 0.7$. This creates a sequence of quadratics where the hitting costs are occasionally very steep, throwing off the online agent if it is looking for patterns. Further, the the minimizer sequence $\{v_t^i\}_t$ is also generated independently for each agent in the following manner:
\begin{align}
    v_t^i = \begin{cases}
        U[-10,10] & \text{ with prob. } $q$\\
        U[-10,10] + \epsilon \cdot 1.1^t & \text{ otherwise}
    \end{cases}
\end{align}
where we set $q = 0.9$ and $\epsilon$ is an independent Radamacher random variable ($-1$ or $1$ with equal probability). This again creates random spikes in the positions of the hitting cost in the action space, making it difficult to track for the online agent. 

The different network topologies considered across various plots in Section \ref{sec:num_exp} have been illustrated in Figure \ref{fig:graph_ill}.
\begin{figure}[]% [hpbt] what you need
    \centering
     \subfigure[$\mathcal{D}=2$]{%
        \includegraphics[width=0.3\linewidth]{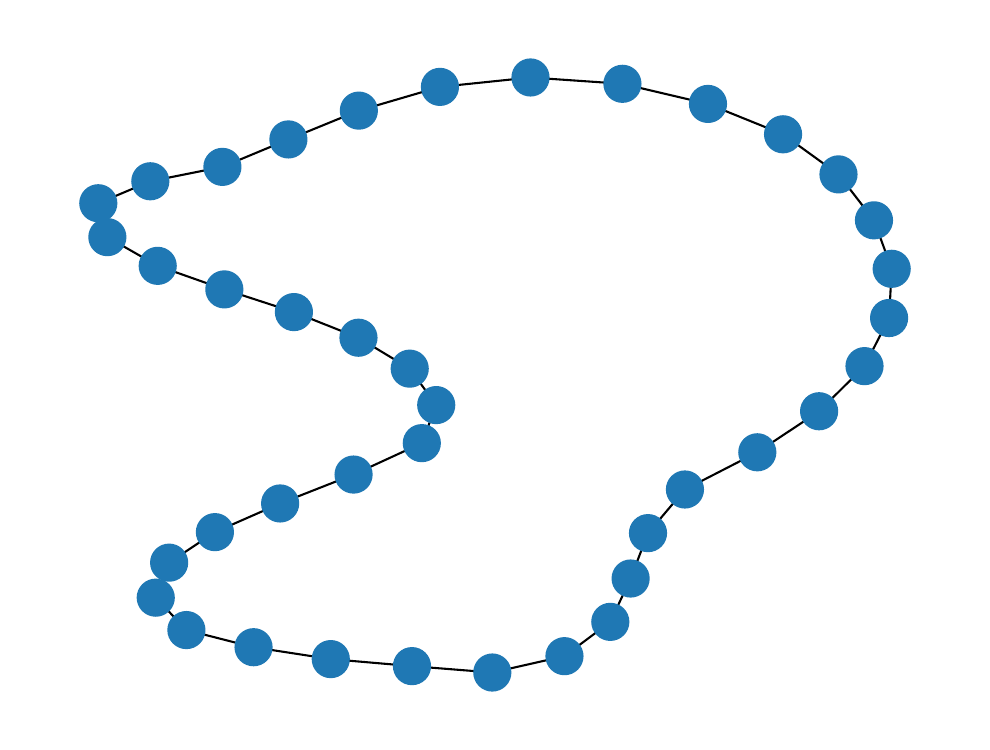}
    \label{fig:graph_ill_1}}
     \subfigure[$\mathcal{D}=10$]{%{0.45\textwidth}
        \includegraphics[width=0.3\linewidth]{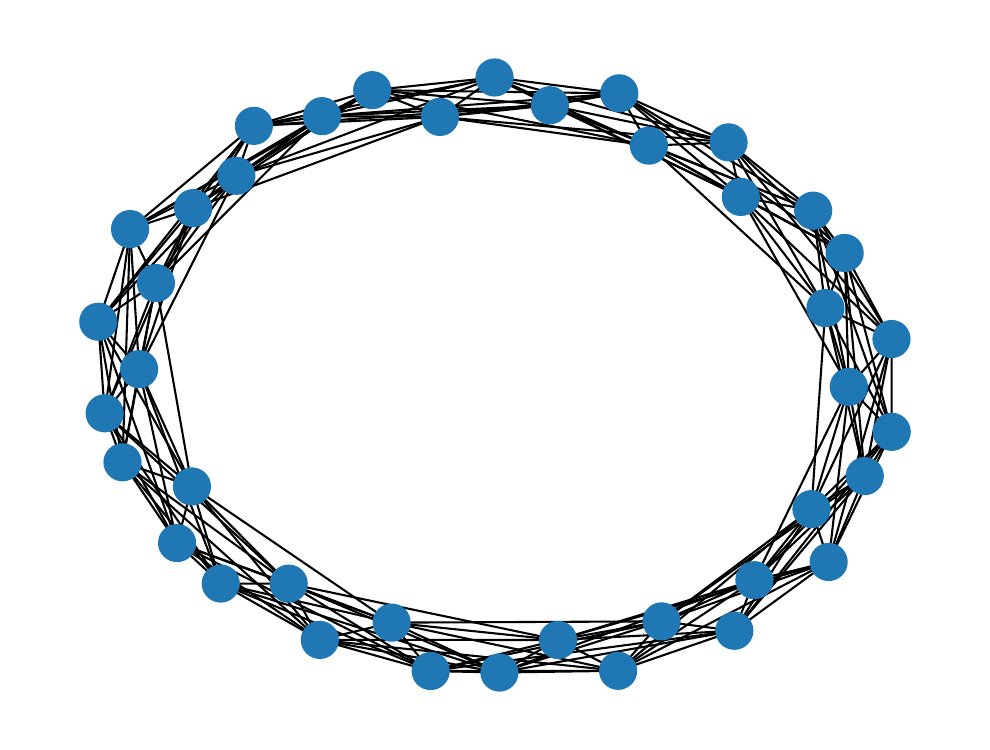}
        \label{fig:graph_ill_2}}
    \subfigure[$\mathcal{D}=20$]{%{0.45\textwidth}
        \includegraphics[width=0.3\linewidth]{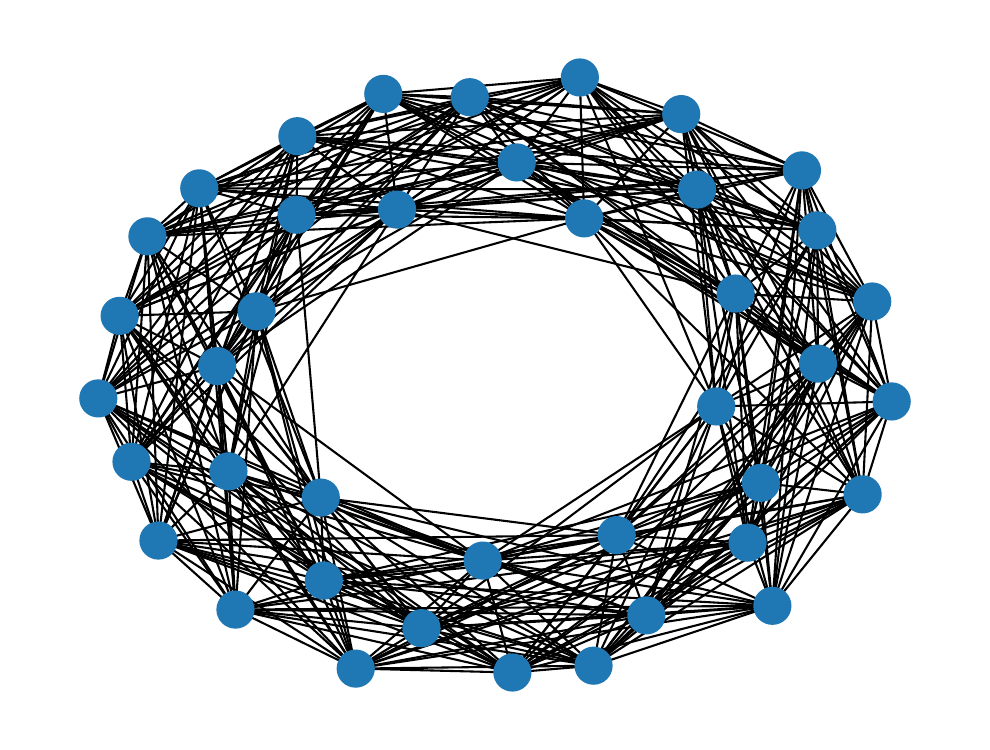}
        \label{fig:graph_ill_3}}
    \subfigure[$\mathcal{D}=30$]{%{0.45\textwidth}
        \includegraphics[width=0.3\linewidth]{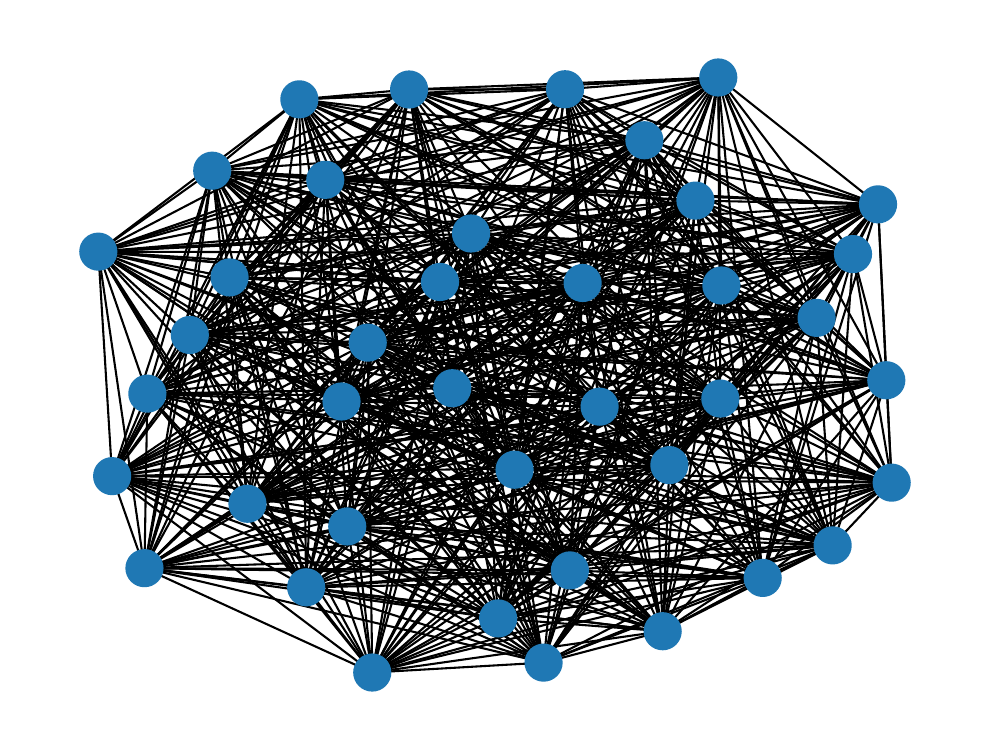}
        \label{fig:graph_ill_4}}
    \subfigure[$\mathcal{D}=39$]{%{0.45\textwidth}
        \includegraphics[width=0.3\linewidth]{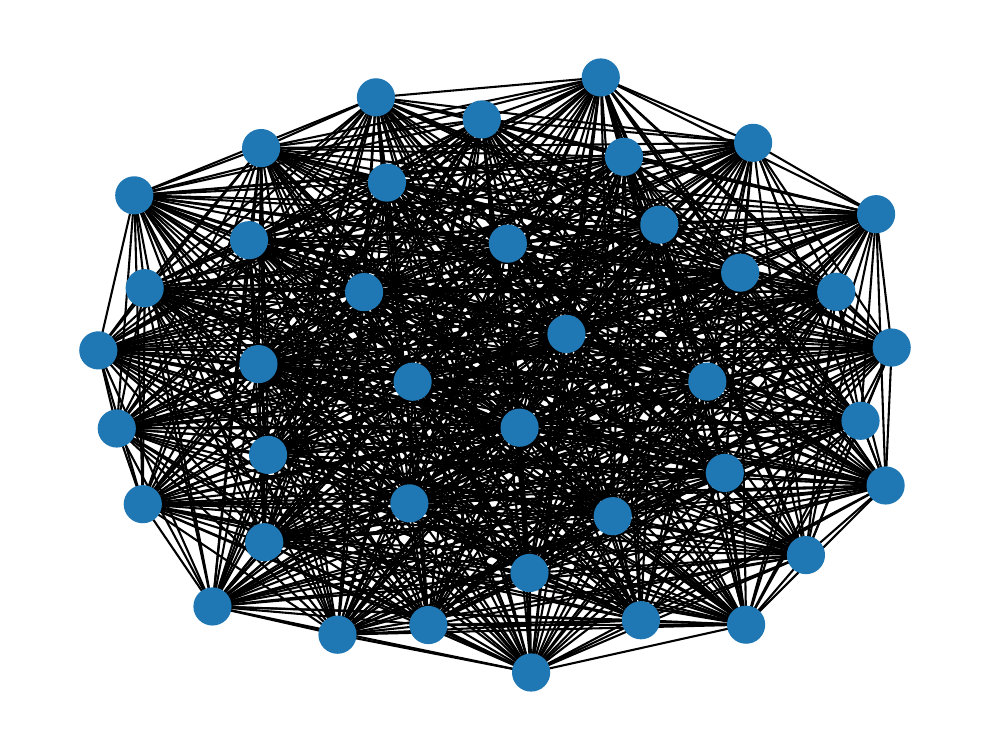}
        \label{fig:graph_ill_5}}
    \caption{Graph illustrations for $N=40$}
    \label{fig:graph_ill}
\end{figure}

Implementation of \acord{} is straightforward, especially, for quadratic hitting costs. This is because the local optimization step's optimality criteria for each agent is linear and can be vectorized together. Iteratively doing this with edge-averaging (step 2 in Algorithm \ref{alg:ACORD}'s inner loop) does the job. For \lpc{$(r)$} one needs to first calculate each agent's $r$-hop neighborhood by employing \textit{breadth-first-search}. Then for each agent, the neighborhood optimization problem involves a matrix inversion of size $|\mathcal{N}_i^r|\times |\mathcal{N}_i^r|$, which leads to the slower performance than \acord{} as $r$ increases.

The timing analysis is performed by averaging the clock-time difference between function calls to the \acord{} and \lpc{$(r)$} algorithms. Since \acord{}'s implementation is vectorized and involves inversion of an $N\times N$ diagonal matrix, the per-agent runt-time can be found by normalizing by $N$. \lpc{$(r)$}'s implementation involves a for-loop over the agents' action calculation, so the same normalization applies here too.

\section{Naive approaches to Dissimilarity Cost}\label{section: naive}
\begin{corollary}
    Suppose each agent follows a localized version of \robd{} with no regard of the dissimilarity cost. There exists uncountable instances where the competitive ratio is $\Omega(\beta)$.
\end{corollary}
\begin{proof}
    Take two agents in $\R$ which are connected by one edge and $\beta = \beta$ $\forall$ $t$. Therefore, the dissimilarity cost is $\frac{\beta}{2}(x^1_t - x^2_t)^2$. Now, consider an instance with $f^1_t(x) = \frac{(x - v^1_t)^2}{2}$ and $f^2_t(x) = \frac{(x - v^2_t)^2}{2}$. If the two agents follow \robd{} individually and locally,
    \begin{align}
        x^1_t = \argmin_x \frac{(x - v^1_t)^2}{2} + \left(\frac{2}{1 + \sqrt{5}}\right) \frac{(x - x^1_{t-1})^2}{2}\\
        x^2_t = \argmin_x \frac{(x - v^2_t)^2}{2} + \left(\frac{2}{1 + \sqrt{5}}\right) \frac{(x - x^2_{t-1})^2}{2}
    \end{align}
    meaning,
    \begin{align}
        x^1_t = \frac{2 x^1_{t-1} + (1+\sqrt{5})v^1_t}{3+\sqrt{5}}\\
        x^2_t = \frac{2 x^2_{t-1} + (1+\sqrt{5})v^2_t}{3+\sqrt{5}}
    \end{align}
    and the dissimilarity cost is $$\frac{\beta}{2(3+\sqrt{5})^2}\left((1+\sqrt{5})(v^1_t - v^2_t) + 2 (x^1_{t-1} - x^2_{t-1}) \right)^2.$$
    Now consider the instance where $\{v^1_t\}_t = \N$ and $\{v^1_t\}_t = -\N$. This means $$x^1_t  > 0 > x^2_t \text{ } \forall \text{ } t$$ if start point is zero. This leads to the dissimilarity cost being at least
    \begin{align}
        \frac{2\beta (1+\sqrt{5})^2}{(3+\sqrt{5})^2}t^2 \propto \beta
    \end{align}
    at round $t$. Now, suppose we have $\beta \to \infty$. Then \\acord{} is forced to have actions across all agents same, leading to a cost independent of $\beta$ . Vanilla \robd{} doesn't take this into account and will have cost still linearly dependent on $\beta$, having a competitive ratio
    \begin{align}
        CR_{\robd{}} = \frac{cost(\robd{})}{cost(\opt{}))} \geq \frac{cost(\robd{})}{cost(\acord{})} &= \Omega(\beta).
    \end{align}
\end{proof}

\begin{corollary}
    Suppose $x^1_t = x^2_t = \ldots = x^N_t = x_t$ $\forall$ $t$, then there exist instances such that the competitive ratio blows up to $\infty$.
\end{corollary}
\begin{proof}
Consider $N$ agents to be in $\R$. Take any hitting cost sequence $\left(\{\left\{ f_1^i(\cdot) \right\}_{i=1}^N, \{\left\{ f_2^i(\cdot) \right\}_{i=1}^N, \ldots, \{\left\{ f_T^i(\cdot) \right\}_{i=1}^N \right)$ with $\beta$ capped at some value. Now, the adversary can increase the strong-convexity parameter to $\infty$ for all hitting costs while keeping their minimizers $v^i_t$ far from each other. In such a scenario, having $x^1_t = x^2_t = \ldots = x^N_t$ leads to infinite hitting costs as the minimizers of each are away from each other. While a simple algorithm like following the (local) minimizer $v^i_t$ will have finite cost, making the competitive ratio $\infty$ for an algorithm enforcing zero dissimilarity cost.
\end{proof}

\section{Related Works}\label{appendix_sec: lit_review}
\subsection{Single-agent SOCO}
The existing research on smoothed online optimization primarily focuses on single-agent scenarios. Initial results, such as those by \cite{LinWiermanAndrew11, Bansal15}, addressed the problem in the one-dimensional setting. Subsequent studies, like \cite{ChenWierman18}, demonstrated that achieving meaningful results necessitates assumptions beyond convexity, such as $\alpha$-polyhedrality or $m$-strong convexity. In particular, \cite{ZhangYang21} showed that for $\alpha$-polyhedral hitting costs and $\ell_1$-norm switching costs, following the minimizer can achieve an impressive competitive ratio of $\left(\max\left\{ 1, \frac{2}{\alpha} \right\}\right)$. Additionally, \cite{GoelWierman19} introduced an order-optimal algorithm, \robd{}, for scenarios involving $m$-strongly convex hitting costs and $\ell_2$-norm switching costs, achieving a competitive ratio of $1 + \mathcal{O}\big(\frac{1}{\sqrt{m}}\big)$. The analytical methods used in these studies, such as potential function analysis and receding horizon control, are tailored to optimizing single-agent performance and do not leverage the additional structure in hitting costs that a multi-agent framework could offer.

\subsection{Multi-agent SOCO}
In the decentralized competitive algorithms domain, there are very few algorithms in the literature. The only ones in our setting of spatio-temporal costs are \cite{lin2022decentralized} and \cite{LiShaoleiWierman23}, with the latter one being a meta-process assuming the existence of a competitive decentralized algorithm. The Localized Predictive Control (LPC) algorithm presented in \cite{lin2022decentralized} lacks efficiency across various fronts. First, the Model Predictive Control (MPC) framework used to design the algorithm renders it dependent on perfect predictions of future hitting costs. Second, the algorithm requires agents to exchange infinite dimensional hitting costs with amongst each other over an $r$-hop neighborhood. Lastly, each agent is required to solve a high-dimensional optimization sub-problem each round that scales with the size of the neighborhood. All these issues prevent LPC from being scalable despite being designed as a decentralized algorithm.

As observed above, a notable limitation in the existing literature is the reliance on agents having access to partial information about the time-varying global loss functions, necessitating the exchange of local cost functions among neighboring agents.  To the best of our knowledge, this paper is the first to establish decentralized competitive guarantees without requiring function exchange between agents.

\subsection{Comparison to ``distributed'' online convex optimization}
Our study contributes to the expanding body of research on distributed online convex optimization (OCO) with time-varying cost functions across multi-agent networks. Recent progress in this field includes work on distributed OCO with delayed feedback \cite{XuanyuTamer21}, coordination of behaviors among agents \cite{XiuxianLihua21, XuanyuTamer21}, and approaches that address distributed OCO with a time-varying communication graph \cite{HosseiniChapman16, AkbariLinder15, YuanProutiere21, LiXie20, YiJohansson20}. However, the above works do not penalize agents for changing actions (switching costs) or for the lack of coordination (dissimilarity costs) and, hence, do not deal with the non-trivial spatio-temporal coupling that comes with it. Further, the above guarantees either use static regret or dynamic regret as the performance metric.

\subsection{Comparison to ``decentralized'' online learning}\label{appendix_sec: DOCO}
Contrary to Decentralized SOCO, there has been significant advancements in traditional decentralized online learning. However, the set-up considered in this line work is different than SOCO. The action space considered is bounded, that is only policies involving $\|x_t\| \leq M$ for some $M>0$ are considered. The gradients of the hitting costs are restricted to be bounded, that is, $\|\nabla f_t^i\|\leq G$. Further, these works do not study switching costs and its relationship with spatial dissimilarity costs. Finally, the benchmark usually considered is the static offline optimal instead of the dynamic hindsight optimal.

The latest in this line of work is \cite{wang2023distributed}, where the authors consider $G$-Lipshchitz hitting costs for each agent $i$ in a bounded action space $\mathcal{K}$ over a finite horizon $T$, where agents can communicate over a \textit{static} undirected graph $\mathcal{G}$. As is the case in traditional OCO, \textit{no switching costs} are considered and the performance metric is the \textit{static} regret. It is worthwhile to note that \cite{wang2023distributed} does not consider any form of spatial costs between neighboring agents.

Their proposed approach: Distributed Online Smooth Projection-Free
Algorithm (\dospa{}) combines the principles of Follow the Perturbed Leader (\ftpl{}) and gradient accumulation plus averaging over the neighborhood, to update the action sparingly. The idea is that one can divided the horizon into $\mathcal{O}(T^{1/2})$ partitions, playing the same "good" action through the partition and only updating when going into a new partition. \dospa{} achieves a static regret of $\mathcal{O}(T^{2/3})$ while communicating $O(T^{1/2})$ times through the horizon $T$. This guarantee is possible by exploiting three aspects of the environment: (i) Bounded space (ii) Lipschitz functions (iii) absence of spatial coupling costs.

Bounded space and Lipschitzness of the hitting costs allow the player to take a sub-optimal action throughout a partition and handle the accumulating error in one-shot. The error handling is further aided by the absence of spatial costs.

In contrast, our setting involves a much broader class of hitting costs, as we assume only strong-convexity, avoiding both zero-order (function level) and first-order (gradient level) Lipschitzness assumption. The presence of spatial costs in our setting requires decoupling the objective, a major hurdle that is not present in the above set-up. Lastly, our benchmark is the optimal sequence of actions in hindsight, a stronger yard-stick than static regret.

\end{appendices}

\end{document}